\documentclass{article}
\usepackage{geometry}
\usepackage{authblk}
\usepackage[utf8]{inputenc}
\usepackage[T1]{fontenc}
\usepackage{amsmath}
\usepackage{todonotes}
\usepackage{mathtools}
\usepackage{amssymb}
\usepackage{amsthm}
\usepackage[hyphens]{url}
\usepackage{hyperref}
\usepackage[hyphenbreaks]{breakurl}
\usepackage{cleveref}
\usepackage{mathdots}
\usepackage{mathrsfs}
\usepackage{color}
\usepackage{xcolor}
\usepackage{tikz}
\usetikzlibrary{arrows.meta}
\usepackage{pgfplots} 
\usepackage{tcolorbox}
\usepackage{pdfpages}
\usepackage[width=14cm,format=hang]{caption}
\usepackage{subcaption}

\newcommand*\samethanks[1][\value{footnote}]{\footnotemark[#1]}
\newcommand{\mres}{\mathbin{\vrule height 1.6ex depth 0pt width 0.13ex\vrule height 0.13ex depth 0pt width 1.3ex}}
\newcommand{\Ha}{\mathcal{H}^1}
\newcommand{\Le}{\mathcal{L}}
\newcommand{\dHa}{\,\mathrm{d}\Ha}
\newcommand{\dLe}{\,\mathrm{d}\mathcal{L}}
\newcommand{\R}{\mathbb{R}}
\newcommand{\ws}{\stackrel{*}{\rightharpoonup}}
\newcommand{\F}{\mathcal{F}}
\newcommand{\G}{\mathcal{G}}
\DeclareMathOperator*{\esssup}{ess\,sup}

\newtheorem{corr}{Corollary}[subsection]
\newtheorem{prop}[corr]{Proposition}
\newtheorem{lem}[corr]{Lemma}
\newtheorem{rem}[corr]{Remark}
\newtheorem{examp}[corr]{Example}
\newtheorem{defin}[corr]{Definition}
\newtheorem{assump}[corr]{Assumption}
\newtheorem{conv}[corr]{Convention}
\newtheorem{thm}[corr]{Theorem}

\crefname{examp}{example}{examples}
\crefname{prop}{proposition}{propositions}
\crefname{lem}{lemma}{lemmas}
\crefname{defin}{definition}{definitions}
\crefname{assump}{assumption}{assumptions}
\crefname{figure}{figure}{figures}

\newcommand{\notinclude}[1]{}
\newcommand{\dualityApproach}[1]{}

\usepackage[backend=biber,
maxnames=5,
style=alphabetic,
]{biblatex}
\nocite{*}

\title{Formulation of branched transport as geometry optimization}
\author{Julius Lohmann\thanks{Institute for Numerical and Applied Mathematics, University of Münster, Einsteinstraße 62, 48149 Münster, Germany}\hspace{0.35em}\thanks{juliuslohmann@uni-muenster.de}\qquad Bernhard Schmitzer\thanks{Institute for Computer Science, University of Göttingen, Goldschmidtstraße 7, 37077 Göttingen, Germany. schmitzer@cs.uni-goettingen.de}\qquad Benedikt Wirth\samethanks[1]\hspace{0.35em}\thanks{benedikt.wirth@uni-muenster.de}}

\begin{document}
	\maketitle
	\begin{abstract}
		The branched transport problem, a popular recent variant of optimal transport, is a non-convex and non-smooth variational problem on Radon measures.
		% It involves a subadditive cost function $m\mapsto\tau (m)$ that describes the effort to move mass $m$ per unit distance.
		The so-called urban planning problem, on the contrary, is a shape optimization problem that seeks the optimal geometry of a street or pipe network.
		We show that the branched transport problem with concave cost function is equivalent to a generalized version of the urban planning problem.
		% For this purpose, we introduce a generalization of the so-called urban metric and prove that the corresponding Wasserstein-$1$-distance can be written as a Beckmann-type problem. Furthermore, we define a network maintenance cost using the Legendre-Fenchel transformation of $-\tau$.
		Apart from unifying these two different models used in the literature,
		another advantage of the urban planning formulation for branched transport is that it provides a more transparent interpretation of the overall cost by separation into a transport (Wasserstein-$1$-distance) and a network maintenance term,
		and it splits the problem into the actual transportation task and a geometry optimization.
		%Our generalized model includes a friction coefficienting function, which is a function of the position on the network to be optimized.
	\end{abstract}
	
	%\begin{abstract}
	%The branched transport problem is a non-convex and non-smooth variational problem on %Radon measures. It involves a subadditive cost function $\tau (m)$ that describes the %effort to move mass $m$ per distance. We show that the branched transport problem with %concave cost function is equivalent to a generalized version of the so-called urban %planning problem. For this purpose, we introduce a network maintenance cost using the %Legendre-Fenchel transformation of $-\tau$. The urban planning problem provides a %more transparent interpretation of the overall cost by separation into a transport %(Wasserstein-$1$-distance) and a maintenance term. Our generalized model includes a %network transport cost, which is a function of the position on the network to be %optimized.\\\\
	%\textbf{Keywords:} optimal transport, optimal networks, branched transport, urban %planning, Wasserstein distance, geometric measure theory
	%\end{abstract}
	
	\small
	\tableofcontents
	\normalsize
	
	\section{Introduction}
	\label{intro}
	Branched transport and urban planning are distinct models developed during the past two decades that both describe transportation networks;
	the textbooks by Bernot et al.\ \cite{BCM} and by Buttazzo et al.\ \cite{BPSS} are devoted to either model and provide a good starting point into the literature.
	
	The main motivation for \emph{branched transport} is a variational explanation of the high complexity and ramification found in many natural transportation systems
	such as river networks, vascular anatomy (like the blood vessel or the bronchial system) or botanical structures (like roots or leaf venation).
	The model is based on the assumption that the (biological or energetic) cost incurred by the transport is subadditive in the transported mass
	so that it is cost-efficient to merge originally separate material flows into few large material flows.
	This tendency of flow-merging then automatically leads to network-like material streams with many branchings.
	
	\emph{Urban planning} on the other hand was devised as an optimal control or shape optimization problem.
	Here one optimizes the street layout or the public transport routes in order to allow efficient commuting of the population between their homes and their workplaces.
	The cost of a street or public transport network then is composed
	of its maintenance cost (in the original model simply the total network length)
	and the cost of the population for commuting
	(measured as the optimal transport or Wasserstein-$1$-distance between the distributions of homes and workplaces in a metric that depends on the street network).
	
	Even though both model formulations are fundamentally different
	(branched transport is a non-convex optimization problem over $1$-currents, while urban planning can be seen as a bilevel shape optimization problem)
	the resulting network structures behave in a phenomenologically similar way.
	In \cite{BW16} it was then shown that the (original) urban planning problem can equivalently be formulated as a specific branched transport problem.
	The aim of the current work is to greatly generalize this result:
	We will introduce a natural generalization of the urban planning problem (of which the original urban planning problem is a specific case)
	and then show that every branched transport problem with concave transportation cost is equivalent to a generalized urban planning problem and vice versa.
	In particular, optimizers of one problem induce optimizers of the other.
	
	We think that the equivalence between both models is not just useful because it unites different strands of literature.
	It also has implications for the modelling and the numerics of such problems.
	As for the modelling, the urban planning formulation clearly separates two different contributions to the overall cost:
	the cost for the actual transportation as well as the cost for building and maintaining the transport network.
	This is not only easier to interpret than the lumped cost of branched transport,
	it also allows to consider (potentially more realistic) variants in which transportation and maintenance cost are payed by different parties (such as commuters and transport companies),
	leading to games between different players.
	As for numerics, there exist phase field approximations of branched transport \cite{CFM,FDW,W} that can now be applied to solve urban planning problems numerically.
	Similarly, a bilevel optimization seems an attractive alternative numerical approach (though not yet implemented for such problems to the best of our knowledge)
	which now becomes available also for branched transport.
	
	In the remainder of the introduction we briefly state the branched transport and the urban planning model as well as our main results.
	In \cref{sect2} we then analyse the Wasserstein distance with respect to the so-called urban metric, a (pseudo-)metric that depends on a street or transport network and that occurs in urban planning.
	In particular, we will prove properties of this urban metric and derive an equivalent Beckmann formulation. Finally, in \cref{sec4} the equivalence between the branched transport and the urban planning problem is shown.

	\subsection{Generalized branched transport}
	\label{euler}
	There are various ways to describe branched transport, in particular a \emph{Eulerian formulation} due to Xia \cite{X}, which uses vector-valued Radon measures or $1$-currents on $\R^n$, and a \emph{Lagrangian formulation} due to Maddalena, Solimini and Morel \cite{MSM03} based on so-called irrigation patterns. We here only present the former (irrigation patterns will be introduced later in \cref{subs31}).
	
	The cost for moving an amount of mass $m$ per unit distance will be described by a \emph{transportation cost} $\tau$.
	
	\begin{defin}[Transportation cost]
		A \textbf{transportation cost} is a non-decreasing concave function $\tau:[0,\infty)\to[0,\infty)$ with $\tau(0)=0$.
	\end{defin}
	
	The monotonicity of $\tau$ as well as $\tau(0)=0$ are natural requirements for a cost.
	The concavity could in principle be relaxed to the weaker condition of subadditivity,
	\begin{equation*}
	\tau(m_1+m_2)\leq\tau(m_1)+\tau(m_2),
	\end{equation*}
	which encodes an efficiency gain if mass is transported in bulk.
	Different examples for $\tau$ are presented in \cref{fig1}.
	Originally only $\tau(m)=m^\alpha$ for $\alpha\in(0,1)$ was used but was generalized to the above in \cite{BW}.
	The borderline choice $\tau(m)=m$ does not exhibit any preference for transport in bulk and is known to lead to classical Wasserstein-$1$ transport.
	
	The material flows from a source distribution $\mu_+$ to a sink distribution $\mu_-$ (without loss of generality probability measures) are described by so-called mass fluxes.
	
	\begin{defin}[Polyhedral mass flux and branched transport cost]
		Assume that $\mu_+$ and $\mu_-$ are finite sums of weighted Dirac measures, i.e.,
		\begin{equation*}
		\mu_+=\sum_{i=1}^{M}f_i\delta_{x_i}\textup{\qquad and\qquad}\mu_-=\sum_{j=1}^{N}g_j\delta_{y_j},
		\end{equation*}
		where $f_i,g_j\in[0,1]$ satisfy $\sum_if_i=\sum_jg_j=1$ and $x_i,y_j\in\R^n$. A \textbf{polyhedral mass flux} between $\mu_+$ and $\mu_-$ is a vector-valued Radon measure $\F\in\mathcal{M}^n(\R^n)$ which satisfies $\textup{div}(\F)=\mu_+-\mu_-$ in the distributional sense and can be expressed as
		\begin{equation*}
		\F=\sum_em_e\vec{e}\Ha\mres e,
		\end{equation*}
		where the sum is over finitely many edges $e=x_e+[0,1](y_e-x_e)\subset\R^n$ with orientation $\vec{e}=(y_e-x_e)/|y_e-x_e|$, the coefficients $m_e$ are real weights, and $\Ha\mres e$ is the one-dimensional Hausdorff measure restricted to $e$. The \textbf{branched transport cost} of $\F$ with respect to a transportation cost $\tau$ is defined as
		\begin{equation*}
		\mathcal{J}^{\tau,\mu_+,\mu_-}[\F]=\sum_e\tau(m_e)\Ha(e).
		\end{equation*}
	\end{defin}
	A polyhedral mass flux can equivalently be represented as a weighted directed graph with edges $e$ and weights $m_e$.
	The condition $\textup{div}(\F)=\mu_+-\mu_-$ encodes Kirchhoff's law of mass preservation: Let $v$ be any vertex of the weighted directed graph associated with $\F$ such that $v$ is not contained in $\textup{supp}(\mu_+)\cup\textup{supp}(\mu_-)$. Then the condition implies
	\begin{equation*}
	%-\int_{\R^n}\nabla\phi\cdot\,\mathrm{d}\F=\int_{\R^n}\phi\,\mathrm{d}(\mu_+-\mu_-)\qquad\textup{implies}\qquad
	\sum_{e_{\text{in}}}m_{e_{\text{in}}}=\sum_{e_{\text{out}}}m_{e_{\text{out}}},
	\end{equation*}
	% where $\vec{e}_{\text{in}}=(v-x_{e_{\text{in}}})/|v-x_{e_{\text{in}}}|$ and $\vec{e}_{\text{out}}=(y_{e_{\text{out}}}-v)/|y_{e_{\text{out}}}-v|$.
	where we sum over all incoming edges $e_{\text{in}}$ and outgoing edges $e_{\text{out}}$ at $v$.
	
	Using the idea of (discrete) polyhedral mass fluxes we can pass to the continuous case using weak-$*$ convergence.
	\begin{defin}[Mass flux, approximating graph sequence and branched transport cost]
		\label{defags}
		A vector-valued Radon measure $\F\in\mathcal{M}^n(\R^n)$ is called \textbf{mass flux} between two probability measures $\mu_+$ and $\mu_-$ on $\R^n$ if there exist two sequences of probability measures $\mu_+^k,\mu_-^k$ and a sequence of polyhedral mass fluxes $\F_k$ with $\textup{div}(\F_k)=\mu_+^k-\mu_-^k$ such that $\F_k\xrightharpoonup{*}\F$ and $\mu_{\pm}^k\xrightharpoonup{*}\mu_{\pm}$, where $\xrightharpoonup{*}$ indicates the weak-$*$ convergence in duality with continuous functions. The sequence $(\F_k,\mu_+^k,\mu_-^k)$ is called \textbf{approximating graph sequence}, and we write $(\F_k,\mu_+^k,\mu_-^k)\xrightharpoonup{*}(\F,\mu_+,\mu_-)$. If $\F$ is a mass flux, then the \textbf{branched transport cost} of $\F$ is defined as
		\begin{equation*}
		\mathcal{J}^{\tau,\mu_+,\mu_-}[\F]=\inf\left\{\liminf_k\mathcal{J}^{\tau,\mu_+^k,\mu_-^k}[\F_k]\,\bigg|\, (\F_k,\mu_+^k,\mu_-^k)\xrightharpoonup{*}(\F,\mu_+,\mu_-)\right\}.
		\end{equation*}
	\end{defin}
	The branched transport problem seeks the optimal mass fluxes between $\mu_+$ and $\mu_-$.
	\begin{defin}[Branched transport problem]
		\label{defbtp}
		The \textbf{branched transport problem} is given by
		\begin{equation*}
		\inf\{ \mathcal{J}^{\tau,\mu_+,\mu_-}[\F]\,|\,\F\in\mathcal{M}^n(\R^n),\textup{div}(\F)=\mu_+-\mu_- \}.
		\end{equation*}
	\end{defin}
	A minimizer is known to exist in case the problem is finite, which is true under mild growth conditions on $\tau$ \cite{BW}.
	
	\subsection{Generalized urban planning problem}
	\label{GurbWas}
	\label{Guppdef}
	The urban planning problem was proposed by Brancolini and Buttazzo \cite{BB} as well as Buttazzo, Pratelli, Solimini and Stepanov \cite{BPSS}.
	We directly state our generalization and relate it to the original model afterwards.
	The basic idea is to optimize a street network, which is represented by a set $S\subset\R^n$
	as well as a function $b:S\to[0,\infty)$ that describes how costly it is to travel along each part of the network (one may view it as the inverse road quality).
	The cost for travelling outside the network $S$ is assumed to be a fixed constant $a\in[0,\infty]$ per distance.
	Given such a street network, the cost for travelling from $x$ to $y$ is described by the \emph{urban metric}.
	
	\begin{defin}[Generalized urban metric]
		\label{gum}
		Let $S\subset\R^n$ be countably $1$-rectifiable and Borel measurable, $b:S\to[0,\infty)$ lower semi-continuous and $a\in[0,\infty]$ with $b\leq a$ on $S$. The associated \textbf{generalized urban metric} is defined as
		\begin{equation*}
		d_{S,a,b}(x,y)=\inf_{\gamma\in\Gamma^{xy}}\int_{\gamma([0,1])\cap S}b\,\mathrm{d}\Ha+a\Ha(\gamma([0,1])\setminus S),
		\end{equation*}
		where $\Gamma^{xy}$ denotes the set of all Lipschitz paths $\gamma:[0,1]\to\R^n$ with $\gamma(0)=x$ and $\gamma(1)=y$.
	\end{defin}
	
	Without any further restrictions, $d_{S,a,b}$ is actually only a pseudometric since positive definiteness and finiteness cannot be guaranteed.
	
	We will call the function $b$ the \emph{friction coefficient} of the road or pipe network
	since it obviously describes how difficult motion on the network is.
	Of course, the meaning is the same as the previously mentioned inverse road quality.
	
	Now let $\mu_+$ and $\mu_-$ be probability measures on $\R^n$ describing the initial and final distribution of a quantity that is to be transported (or of homes and workplaces). The total cost for transporting $\mu_+$ onto $\mu_-$ is given by the Wasserstein distance with respect to $d_{S,a,b}$.
	
	\begin{defin}[Wasserstein distance, transport plans]
		\label{WdisTp}
		Let $S,a,b$ be as in \cref{gum}. The \textbf{Wasserstein distance} between $\mu_+$ and $\mu_-$ with respect to $d_{S,a,b}$ is defined as
		\begin{equation*}
		W_{d_{S,a,b}}(\mu_+,\mu_-)=\inf_\pi\int_{\R^n\times\R^n}d_{S,a,b}\,\mathrm{d}\pi,
		\end{equation*}
		where the infimum is taken over all probability measures $\pi$ on $\R^n\times\R^n$ with $\pi(B\times\R^n)=\mu_+(B)$ and $\pi(\R^n\times B)=\mu_-(B)$ for all Borel sets $B\in\mathcal{B}(\R^n)$. Any such measure is called a \textbf{transport plan}. The set of transport plans is denoted by $\Pi(\mu_+,\mu_-)$.
	\end{defin}
	
	Given a transport plan $\pi$, the quantity $\pi(A,B)$ indicates how much mass is transported from $A\subset\R^n$ to $B\subset\R^n$
	so that the Wasserstein distance is nothing else than the accumulated travel cost of all mass particles.
	This Wasserstein distance will form part of the urban planning cost, the other part comes from the maintenance of the network $(S,b)$.
	The maintenance cost of a unit street segment of inverse quality $\hat b$ shall be described by $c(\hat b)$ for a function $c:[0,\infty)\to[0,\infty]$.
	Since maintenance cost naturally increases with road quality, $c$ shall be non-increasing.
	
	\begin{defin}[Generalized urban planning cost]\label{def:urbPlnCost}
		Given a non-increasing maintenance cost $c:[0,\infty)\to[0,\infty]$, set
		\begin{equation*}
		a=\inf c^{-1}(0).
		\end{equation*}
		The \textbf{generalized urban planning cost} of a street network $(S,b)$ (as in \cref{gum}) is given by
		\begin{equation*}
		\mathcal{U}^{c,\mu_+,\mu_-}[S,b]=W_{d_{S,a,b}}(\mu_+,\mu_-)+\int_Sc(b)\,\mathrm{d}\Ha.
		\end{equation*}
	\end{defin}
	Note that $c(b)$ is Borel measurable as a composition of Borel measurable functions. Recall that the function $b$ describes the cost for travelling on the network while the constant $a$ is the cost for travelling outside the network.
	Hence we must have $c(a)=0$ as there is no road to be maintained, which explains the relation $a=\inf c^{-1}(0)$.
	The urban planning problem now seeks the optimal street network.
	
	\begin{defin}[Generalized urban planning problem]\label{def:urbPlnProblem}
		The \textbf{generalized urban planning problem} is given by
		\begin{equation*}
		\inf\left\{\mathcal{U}^{c,\mu_+,\mu_-}[S,b]\,\middle|\,S\subset\R^n\text{ countably $1$-rectifiable and Borel measurable, }b:S\to[0,a]\text{ lower semi-continuous}\right\}.
		\end{equation*}
	\end{defin}
	
	The generalized urban planning problem can be seen as a bilevel optimization problem,
	where the outer problem optimizes the shape $S$ and friction coefficient $b$ of the network and the inner one solves the optimal transport problem.
	
	The original urban planning problem from \cite{BB,BPSS} is obtained by the specific choice
	\begin{equation*}
	c(\hat b)=\begin{cases}
	\infty&\text{if }\hat b<\bar b\\
	\bar c&\text{if }\bar b\leq\hat b<\bar a\\
	0&\text{if }\bar a\leq\hat b
	\end{cases}
	\end{equation*}
	for fixed parameters $\bar a$, $\bar b$ and $\bar c$.
	In that model only a single type of roads is built, namely roads with friction coefficient $\bar b$:
	A better quality is impossible due to infinite maintenance cost, and there is no gain in using worse streets as their maintenance costs the same.
	
	\begin{figure}
		\centering
		\begin{tikzpicture}
		\begin{axis}[axis on top=true,
		x = 2.5em, y = 2.5em,
		samples = 500, 
		domain = 0:5, 
		ymax = 3,
		axis x line=bottom, 
		axis y line=left, 
		xlabel = { $m$}, 
		ylabel = { $\tau (m)$}, 
		x label style={at={(axis description cs:1,0)},anchor=south},
		y label style={at={(axis description cs:.07,0.9)},rotate=270,anchor=west},
		xtick = {1.25},
		xticklabels = {$\frac{\bar{c}}{\bar{a}-\bar{b}}$},
		ytick = {1},
		yticklabels = {$\bar{c}$},
		legend style = {cells={anchor=west},legend pos=outer north east},
		]
		% Sprung erster Ordnung
		\addplot[lightgray, line width = 2.9pt] expression {min(1.8*x,x+1)};% node[left, xshift=-0.5cm,pos=0.3] { $\min (\bar{a}m,\bar{b}m+\bar{c})$};
		\addplot[gray, line width = 1.8pt] expression {x};% node[below right, pos=0.4] { $m$};
		\addplot[black] expression {1};% node[below,pos=0.7] { $\bar{c}1_{(0,\infty)}(m)$};
		\addplot[dashed, mark=none] expression {pow(x,0.5)};% node[below right,pos=0.6] { $m^\alpha$};
		\addplot[only marks,mark=*,mark options={fill=black,scale=.8},text mark as node=true] coordinates{(0,0)};
		\addplot[only marks,mark=*,mark options={fill=white,scale=.7},text mark as node=true] coordinates{(0,1)};
		\legend{urban planning cost $\min(\bar am,\bar{b}m+\bar{c})$\\
			Wasserstein cost $\bar bm$\\
			discrete cost $\bar{c}\cdot 1_{(0,\infty)}(m)$\\
			branched transport cost $m^\alpha$\\};
		\end{axis}
		\end{tikzpicture}%
		\quad
		\begin{tikzpicture}
		\begin{axis}[axis on top=true,
		x = 2.5em, y = 2.5em,
		samples = 500, 
		domain = 0:5, 
		ymax = 3,
		axis x line=bottom, 
		axis y line=left, 
		xlabel = { $b$}, 
		ylabel = { $\varepsilon (b)$}, 
		x label style={at={(axis description cs:1,0)},anchor=south},
		y label style={at={(axis description cs:.1,0.9)},rotate=270,anchor=west},
		xtick = {1,1.8},
		xticklabels = {$\bar{b}$,$\bar{a}$},
		ytick = {1},
		yticklabels = {$\bar{c}$},
		]
		\addplot[only marks,mark=*,mark options={black,scale=.8},text mark as node=true] coordinates{(1,1)};
		\addplot[black] expression {1};
		\addplot[dashed, mark=none,domain=0:4.9] expression {1/(4*x)};
		\addplot[lightgray, line width = 2.9pt,domain=1:4.9] expression {max(0,2.25-1.25*x)};
		\addplot[gray, line width = 1.8pt,domain=1:4.9] expression {0};
		\addplot[only marks,mark=*,mark options={fill=black,scale=.7},text mark as node=true] coordinates{(0,1)};
		\addplot[only marks,mark=*,mark options={black,scale=.7},text mark as node=true] coordinates{(1,0)};
		%		\addplot[mark=none] expression {1};
		%		\addplot[mark=none] expression {1/(4*x)};
		%		\addplot[domain=1:1.5] expression {(1/(1-1.5))*(x-1.5)};
		%		\addplot[domain=1.5:5] expression {0};			
		%		\addplot coordinates{(1,1)(1,5)};
		% \addplot[color = OrangeRed, dashed][domain=0:1] expression {1};
		% Sprung zweiter Ordnung
		%\addplot expression { 1 - exp(-x) * cos(200*x) };
		\end{axis}
		\end{tikzpicture}
		\caption{Classical examples for the transportation cost $\tau(m)$ and the corresponding maintenance cost $\varepsilon(b)=(-\tau)^*(-b)$.
			%The urban planning cost corresponds to movement on (paying $\tilde{b}$ and additional constant cost $\tilde{c}$) and outside a network (penalized by $\tilde{a}$). The scrictly concave branched transport cost with ``branching degree''  $\alpha\in(0,1)$ models the effect that particles always have the tendency to interact.
			The discrete cost corresponds to the Steiner problem of finding an optimal network with minimal total length.}
		\label{fig1}
	\end{figure}
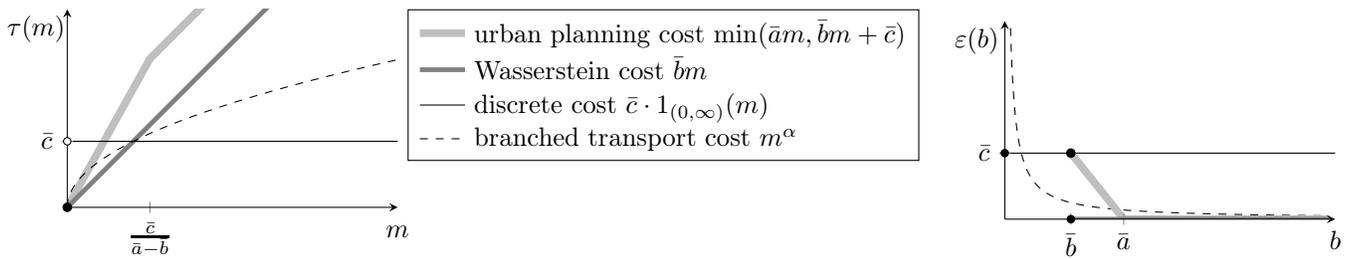
	
	\subsection{Summary of results}

	Our main results are
	\begin{itemize}
		\item
		a Beckmann-type formulation of the Wasserstein distance $W_{d_{S,a,b}}$ (\cref{BWfinal}) and
		\item
		the urban planning formulation of the branched transport problem (\cref{finalthm}).
	\end{itemize}
	In the following we briefly state and discuss both results as well as a few auxiliary results of independent interest.
	Let $\mu_+$ and $\mu_-$ be probability measures on $\R^n$ with bounded supports, without loss of generality contained in $\mathcal{C}=[-1,1]^n$.
	The Wasserstein-1-distance $W_1(\mu_+,\mu_-)=\min_{\pi\in\Pi(\mu_+,\mu_-)}\int_{\mathcal C\times\mathcal C}|x-y|\,\mathrm{d}\pi(x,y)$
	between $\mu_+$ and $\mu_-$ with respect to the Euclidean (or a similarly smooth geodesic) metric is known to equal the minimum cost
	of a material flux from the source $\mu_+$ to the sink $\mu_-$ \cite[Thm.\,4.6]{San},
	\begin{equation*}
	W_1(\mu_+,\mu_-)=\min_{\G}|\G|(\mathcal C),
	\end{equation*}
	where the minimum is taken over all $\R^n$-valued Radon measures $\G\in\mathcal M^n(\mathcal C)$ that satisfy $\textup{div}(\G)=\mu_+-\mu_-$
	and $|\G|(\mathcal C)$ denotes the total variation or total mass of $\G$.
	This minimum cost flow problem is also known as Beckmann formulation.
	The proof essentially consists of two applications of standard convex Fenchel--Rockafellar duality
	(the first dualization yields the so-called Kantorovich--Rubinstein formula, from which the second dualization derives the Beckmann formulation).
	We show that an analogous formulation holds for the Wasserstein distance $W_{d_{S,a,b}}(\mu_+,\mu_-)$ with respect to our urban metric $d_{S,a,b}$.
	To avoid pathological situations in which the street network connects any two points at arbitrarily small cost, we assume the following.
	\begin{assump}
		\label{Sassump}
		For a given pair $(S,b)$ denote the part of the network with friction coefficient no larger than $\lambda$ by $$S_\lambda=\{ z\in S\,|\, b(z)\leq\lambda \}.$$
		We assume that $S_\lambda$ has finite Hausdorff measure, $\Ha(S_\lambda)<\infty$, for all $\lambda\in [0,a)$.
	\end{assump}
	\begin{thm}[Beckmann-type formulation of $W_{d_{S,a,b}}(\mu_+,\mu_-)$]
		\label{BWfinal}
		Let $S\subset\mathcal{C}$ countably $1$-rectifiable and Borel measurable, $a\in [0,\infty]$ and $b:S\to [0,\infty)$ lower semi-continuous with $b\leq a$ on $S$. Suppose that \cref{Sassump} is satisfied. Then we have
		\begin{equation*}
		W_{d_{S,a,b}}(\mu_+,\mu_-)=\inf_{\xi,\F^\perp}\int_S b|\xi|\,\mathrm{d}\Ha+a|\F^\perp|(\mathcal{C}),
		\end{equation*}
		where the infimum is taken over $\xi\in L^1(\Ha\mres S;\R^n)$ and $\F^\perp\in\mathcal{M}^n(\mathcal{C})$ with $\F^\perp\mres S=0$ and $\textup{div}(\xi\Ha\mres S+\F^\perp)=\mu_+-\mu_-$.
	\end{thm}
	While the result is not unexpected (the flux $\G$ here takes the form $\xi\Ha\mres S+\F^\perp$),
	its proof is quite technical and substantially more involved than for the Euclidean Wasserstein-1-distance.
	Indeed, in the Kantorovich--Rubinstein formula one typically needs to jump back and forth (using density arguments)
	between Lipschitz functions with respect to the metric and differentiable functions whose gradient is bounded in terms of the Lipschitz constant and the local metric.
	However, for discontinuous metrics $d_{S,a,b}$ and in particular for $a=\infty$ this becomes difficult.
	Instead, it turns out easier to prove the equality directly, without passing to an intermediate dual problem, by contructing a minimizer of one problem from a minimizer of the other.
	\dualityApproach{\todo[inline]{If Section 3 stays, add the condition $a<\infty$ as an alternative to \cref{Sassump} in the theorem and uncomment the commented sentences.}\notinclude{
			In \cref{sect2} and \cref{subs6} we will provide two different proofs of the Beckmann formulation for the Wasserstein distance. In each of those two sections at least one of the following assumptions is satisfied.
			\begin{assump}
				\label{assump3}
				The parameter $a$ satisfies $a<\infty$.
			\end{assump}
		}%\notinclude
	}%\dualityApproach
	In essence, if $W_{d_{S,a,b}}(\mu_+,\mu_-)<\infty$ and $\pi\in\Pi(\mu_+,\mu_-)$ is a minimizer (which will exist by \cref{exoptplan}), then under \cref{Sassump} an optimal mass flux for the Beckmann problem can be defined as $\F_{\pi,\rho}$ with
	\begin{equation*}
	\langle\varphi,\F_{\pi,\rho}\rangle=\int_{\Theta}\int_{[0,1]}\varphi(\gamma)\cdot\dot{\gamma}\,\mathrm{d}\mathcal{L}\,\mathrm{d}(\rho_\#\pi)(\gamma)
	\quad\text{for all }\varphi\in C(\mathcal{C};\R^n),
	\end{equation*}
	where $\Theta$ is a space of (equivalence classes of) Lipschitz paths (it will be defined in \cref{subs4}) and $\rho_\#\pi$ denotes the push-forward of $\pi$ under $\rho:\mathcal{C}\times\mathcal{C}\to\Theta$, which assigns to a pair of points a shortest connecting path (its existence and Borel measurability will be shown in \cref{final}). Conversely, if the Beckmann problem is finite, then under \cref{Sassump} there exists a minimizer $\G=\xi\Ha\mres S+\F^\perp$ which can be associated with a mass flux measure $\eta$ on $\Theta$ moving $\mu_+$ onto $\mu_-$ (cf. \cref{mfmeas}) via
	\begin{equation*}
	\langle\varphi,\G\rangle=\int_{\Theta}\int_{[0,1]}\varphi(\gamma)\cdot\dot{\gamma}\,\mathrm{d}\mathcal{L}\,\mathrm{d}\eta(\gamma)
	\quad\text{for all }\varphi\in C(\mathcal{C};\R^n).
	\end{equation*}
	This $\G$ then induces an optimal transport plan by
	\begin{equation*}
	\langle\varphi,\pi_\G\rangle=\int_{\Theta}\varphi (\gamma (0),\gamma (1))\,\mathrm{d}\eta (\gamma)
	\quad\text{for all }\varphi\in C(\mathcal{C}\times\mathcal{C} ).
	\end{equation*}
	\notinclude{The proof in \cref{subs6} does not yield constructions of minimizers.}
	
	As for the second main result, we show that any branched transport problem with transportation cost $\tau$ and source and sink $\mu_+$ and $\mu_-$
	can equivalently be written as a generalized urban planning problem with a particular maintenance cost.
	\begin{defin}[Maintenance cost associated with $\tau$]
		\label{maintcost}
		Let $\tau:[0,\infty)\to[0,\infty)$ be a transportation cost. We extend $\tau$ to a function on $\R$ via $\tau(m)=-\infty$ for all $m<0$. The associated \textbf{maintenance cost} ist defined by $\varepsilon(b)=(-\tau)^*(-b)=\sup_{m\geq 0}\tau(m)-bm$ for any $b\in\R$.
	\end{defin}
	By definition $\varepsilon$ equals $+\infty$ on $(-\infty,0)$ and is decreasing by the properties of $\tau$. We use the maintenace cost $c=\varepsilon$ in \cref{def:urbPlnCost}.
	The constant $a=\inf\varepsilon^{-1}(0)$ then equals the right derivative of $\tau$ in $0$, $a=\tau'(0)$.
	Examples are provided in \cref{fig1}.
	\begin{thm}[Bilevel formulation of the branched transport problem with concave transportation cost $\tau$]
		\label{finalthm}
		The branched transport problem can equivalently be written as urban planning problem,
		\begin{equation*}
		\inf_{\F}\mathcal{J}^{\tau,\mu_+,\mu_-}[\F]=\inf_{S,b}\mathcal{U}^{\varepsilon,\mu_+,\mu_-}[S,b],
		\end{equation*}
		where the infima are taken over $\F\in\mathcal{M}^n(\mathcal{C})$ with $\textup{div}(\F)=\mu_+-\mu_-$, countably $1$-rectifiable and Borel measurable $S\subset\mathcal{C}$ and lower semi-continuous functions $b:S\to[0,a]$.
	\end{thm}
	In fact, we do not only show equality of the infima, but from each admissible $\F$ we construct an admissible pair $(S,b)$ with nongreater cost and vice versa
	so that optimizers of one problem induce optimizers of the other.
	In more detail, let $(S,b)$ be admissible for the urban planning problem with $\mathcal{U}^{\varepsilon,\mu_+,\mu_-}[S,b]<\infty$. The latter implies that \cref{Sassump} is automatically satisfied and a minimizer $\pi$ of $W_{d_{S,\tau'(0),b}}(\mu_+,\mu_-)$ exists (which we will show in \cref{exoptplan}). The mass flux $\F_{\pi,\rho}$ from above then can be shown to satisfy
	\begin{equation*}
	\mathcal{J}^{\tau,\mu_+,\mu_-}[\F_{\pi,\rho}]\leq\mathcal{U}^{\varepsilon,\mu_+,\mu_-}[S,b].
	\end{equation*}  
	Conversely, if $\G$ is admissible for the branched transport problem, then there exists a mass flux $\F$ (induced by removing divergence-free parts of $\G$) with
	\begin{equation*}
	\mathcal{J}^{\tau,\mu_+,\mu_-}[\F]\leq\mathcal{J}^{\tau,\mu_+,\mu_-}[\G],
	\end{equation*}
	which can be written as $\F=\xi\Ha\mres S+\F^\perp$ with $\xi\in L^1(\Ha\mres S;\R^n)$ and $\F^\perp\in\mathcal{M}^n(\mathcal{C})$. Further, as we will show, $\xi$ can be represented such that the street network $(S,b=-\max(\partial(-\tau)(|\xi|)))$ is admissible for the urban planning problem and
	\begin{equation*}
	\mathcal{U}^{\varepsilon,\mu_+,\mu_-}[S,b]\leq\mathcal{J}^{\tau,\mu_+,\mu_-}[\F].
	\end{equation*}
	%	A slight variant of this last step is a result on the conjugation of the total transportation cost on $S$:
	%	If the transportation cost $\tau$ is right-continuous in $0$, $S\subset\mathcal{C}$ is countably $1$-rectifiable and Borel measurable and $\xi\in L^1(\Ha\mres S;\R^n)$, then
	%	\begin{equation*}
	%		\int_S\tau (|\xi |)\,\mathrm{d}\Ha=\inf_b\int_Sb|\xi |\,\mathrm{d}\Ha+\int_S\varepsilon (b)\,\mathrm{d}\Ha,
	%	\end{equation*}
	%	where the infimum is taken over $\Ha$-measurable functions $b: S\to [0,a]$ on $S$ (see \cref{subst2}).
	
	Especially the proof of \cref{BWfinal} requires a number of lower semi-continuity results for path lengths and related functionals which are also of their own interest,
	so we list some of them below.
	We assume to be given $S\subset\mathcal{C}$ countably $1$-rectifiable and Borel measurable, $a\in[0,\infty]$ and $b:S\to[0,a]$ lower semi-continuous such that \cref{Sassump} holds.
	Consider a sequence $\gamma_j:[0,1]\to\mathcal{C}$ of Lipschitz paths with uniformly bounded Lipschitz constant that converges uniformly to some $\gamma:[0,1]\to\mathcal{C}$.
	Then the following holds.
	\begin{itemize}
		\item A version of Go\l\k{a}b's theorem holds (see \cref{Golabim}): For all Lebesgue-measurable sets $T\subset [0,1]$ one has
		\begin{equation*}
		\Ha(\gamma(T))\leq\liminf_j\Ha(\gamma_j(T)).
		\end{equation*}
		\item If the $\gamma_j$ have constant speed, the path length associated with $d_{S,a,b}$ is lower semi-continuous (see \cref{Lalmostlow}),
		\begin{equation*}
		L_{S,a,b}(\gamma)\leq\liminf_jL_{S,a,b}(\gamma_j)
		\qquad\text{for}\qquad
		L_{S,a,b}(\gamma)=\int_{\gamma^{-1}(S)}b(\gamma)|\dot{\gamma}|\mathrm{d}\mathcal{L}+a\int_{[0,1]\setminus\gamma^{-1}(S)}|\dot{\gamma}|\mathrm{d}\mathcal{L}.
		\end{equation*}
		\item If $L_{S,a,b}(\gamma_j)$ is uniformly bounded and $a=\infty$, then for each $\delta>0$ there exists a $\lambda\in[0,\infty)$ such that
		\begin{equation*}
		\Ha(\gamma\setminus S_\lambda)\leq\delta
		\end{equation*}
		(see \cref{corro}). In particular, we have $\Ha(\gamma\setminus S)=0$.
		\item For all $x,y\in\mathcal{C}$ there exists an injective Lipschitz path $\psi:[0,1]\to\mathcal{C}$ with $\psi(0)=x$, $\psi(1)=y$ and $L_{S,a,b}(\psi)=d_{S,a,b}(x,y)$. Moreover, $\psi$ may be chosen to have a constant speed, bounded in terms of $d_{S,a,b}(x,y)$ and the Hausdorff measure of subsets of $S$ (see \cref{minimiz}).
		If all $\gamma_j$ are such optimal paths with $L_{S,a,b}(\gamma_j)=d_{S,a,b}(\gamma_j(0),\gamma_j(1))$ uniformly bounded,
		then also $\gamma$ is optimal with $L_{S,a,b}(\gamma)=d_{S,a,b}(\gamma(0),\gamma(1))$ (see \cref{optpathlimit}).
		\item In fact, $d_{S,a,b}=L\circ\rho$ for a Borel measurable path selection $\rho:\mathcal{C}\times\mathcal{C}\to\Theta$ (see \cref{final}; the topology on the space $\Theta$ of paths will be specified in \cref{subs4}).
		\item The urban metric $d_{S,a,b}$ is lower semi-continuous and for $a<\infty$ even continuous (see \cref{dLsc}).
	\end{itemize}
%\dualityApproach
	
	\subsection{General notation and definitions}
	\label{notdefs}
	Throughout the article, we will use the following notation and definitions.
	\begin{itemize}
		\item $I= [0,1]$ denotes the \textit{unit interval}. We will use this notation if $I$ represents the domain of a path.
		\item $\mathcal{S}^{n-1}$ denotes the \textit{unit sphere}.
		\item $\mathcal{C}$ denotes the \textit{hypercube} $[-1,1]^n$.
		\item We write $B_r(x)$ for the \textit{open Euclidean ball} with radius $r>0$ and center $x\in\R^n$.
		\item $\mathcal{L}^k$ denotes the $k$-dimensional \textit{Lebesgue measure}. We write $\mathcal{L}=\mathcal{L}^1$.
		\item $\mathcal{H}^k$ indicates the $k$-dimensional \textit{Hausdorff measure}.
		%		\item Let $V$ be a normed vector space. We write $V^*$ for the \textit{topological dual space}. For $v\in V$ and $v^*\in V^*$ we denote the pairing between $v$ and $v^*$ by $\langle v,v^*\rangle$.
		\item Let $A$ be a topological space. We write $\mathcal{B}(A)$ for the \textit{$\sigma$-algebra of Borel subsets} of $A$.
		\item Assume that $(\Omega,\mathcal{A},\mu)$ is a measure space. We write $L^1(\mu;\R^n)$ for the \textit{Lebesgue space} of equivalence classes of $\mathcal{A}$-$\mathcal{B}(\R^n)$-measurable functions $f:\Omega\to\R^n$ with $\int_\Omega|f|\,\mathrm{d}\mu<\infty$, where two such functions belong to the same class if they coincide $\mu$-almost everywhere. For $\sigma$-finite $\mu$ this definition corresponds to the quotient of the Lebesgue space $L_1(\mu,\R^n)$ defined in \cite[\S\,2.4.12]{F} by the subspace $\{ f\,|\, f=0\, \mu\textup{-almost everywhere} \}$.
		\item Let $\mu:\mathcal{A}\to X$ be a map on a $\sigma$-algebra $\mathcal{A}$ to some set $X$ (e.g., a scalar- or vector-valued measure). For any $A\in\mathcal{A}$ we define the \textit{restriction} $\mu\mres A:\mathcal{A}\to X$ of $\mu$ to $A$ by
		\begin{equation*}
		(\mu\mres A) (B)=\mu(A\cap B)
		\quad\text{for all }B\in\mathcal{A}.
		\end{equation*}
		\item A set $S\subset\R^n$ is said to be \textit{countably $k$-rectifiable} (following \cite[p.\,251]{F}) if it is the countable union of $k$-rectifiable sets. More precisely, 
		\begin{equation*}
		S=\bigcup_{i=1}^\infty f_i(A_i),
		\end{equation*}
		where $A_i\subset\R^k$ is bounded and $f_i:A_i\to\R^n$ Lipschitz continuous. If $S$ is countably $k$-rectifiable and $\mathcal{H}^k$-measurable, then we can apply \cite[Lem.\,3.2.18]{F} which yields the existence of bi-Lipschitz functions $g_i:C_i\to S$ with $C_i\subset\R^k$ compact, $T_i=g_i(C_i)$ pairwise disjoint and 
		\begin{equation*}
		S=T_0\cup \bigcup_{i=1}^\infty T_i
		\end{equation*}
		with $\mathcal{H}^k(T_0)=0$. The sequence 
		\begin{equation*}
		S^N=\bigcup_{i=1}^NT_i
		\end{equation*}
		will be called an \textit{approximating sequence for} $S$.
		\item $\mathcal{M}^k(A)=\{\F:\mathcal{B}(A)\to\R^k\,\sigma\textup{-additive}\}$ denotes the set of $\R^k$-valued \textit{Radon measures} on a Polish space $A$. Note that every $\F\in\mathcal{M}^k(A)$ is automatically regular and of bounded variation (cf.\ \cite[p.\,343]{Els} and \cite[XI, 4.5., Thm.\,8]{L}). More specifically, the \textit{total variation measure} $|\F|$ is regular and satisfies $|\F|(A)<\infty$. We indicate the weak-$*$ convergence of Radon measures by $\ws$. The measure $\F\in\mathcal{M}^k(A)$ is called $\mathcal{H}^l$\textit{-diffuse} if $\F(B)=0$ for all $B\in\mathcal{B}(A)$ with $\mathcal{H}^l(B)<\infty$ \cite[p.\,2]{Sil}.
		\item For any closed subset $A\subset\R^n$ we write $\mathcal{DM}^n(A)=\{\F\in\mathcal{M}^n(A)\,|\,\textup{div}(\F)\in\mathcal{M}^1(A) \}$, where $\textup{div}$ denotes the distributional divergence. These vector-valued Radon measures were termed \textit{divergence measure vector fields} in \cite[p.\,2]{Sil}.
		\item $\Theta^{*k}(\mu,.)$ denotes the \textit{upper $k$-dimensional density} of a Radon measure  $\mu:\mathcal{B}(\R^n)\to[0,\infty)$ \cite[p.\,13]{LS}. It is for every $x\in\R^n$ given by
		\begin{equation*}
		\Theta^{*k}(\mu,x)=\limsup_{r\searrow 0}\frac{\mu(B_r(x))}{r^k\omega_k},
		\end{equation*}
		where $\omega_k$ denotes the volume of the $k$-dimensional unit ball.
		\item The \textit{pushforward} $f_\#\mu$ of a measure $\mu$ on $X$ under a measurable map $f : X \to Y$ is the measure defined by $f_\#\mu (A)=\mu (f^{-1}(A))$ for all measurable subsets $A\subset Y$.
		\item $p_i:A_1\times\ldots\times A_k\to A_i$ abbreviates the \emph{projection} on the $i$-th component.
		\item We write the \textit{arc length} of a Lipschitz path $\gamma:[t_1,t_2]\to\mathcal{C}$ as $\textup{len}(\gamma)=\int_{[t_1,t_2]}|\dot{\gamma}|\,\mathrm{d}\mathcal{L}$ and denote the \textit{Lipschitz constant} by $\textup{Lip}(\gamma)=\sup_{t\neq\tilde{t}}|\gamma (t)-\gamma (\tilde{t})|/|t-\tilde{t}|$.
		\item $\Gamma$ denotes the set of all \textit{Lipschitz paths} mapping $I$ onto $\mathcal{C}$. We write $\Gamma^{xy}=\{f\in\Gamma\,|\,f(0)=x,f(1)=y\}$ for $x,y\in\mathcal{C}$. Further, for $x,y\in\mathcal{C}$ and $C>0$ let 
		\begin{equation*}
		\Gamma_C=\{ f\in\Gamma\,|\,\textup{Lip}(f)\leq C\}\textup{\qquad and\qquad}\Gamma_C^{xy}=\{f\in\Gamma_C\,|\,f(0)=x,f(1)=y\}.
		\end{equation*}
		\item For any Lipschitz path $\gamma:[t_1,t_2]\to\mathcal{C}$ we write $\textup{md}(\gamma,t_0)$ for the \textit{metric differential} of $\gamma$ at $t_0\in (t_1,t_2)$ \cite[p.\,115]{Kirch}, which can be applied to $u\in\R$ by
		\begin{equation*}
		\textup{md}(\gamma,t_0)(u)=\lim_{h\searrow 0}\frac{|\gamma(t_0+hu)-\gamma(t_0)|}{h}
		\end{equation*}
		if the limit exists. Further, the \textit{metric derivative} \cite[p.\,24]{AGS} of $\gamma$ at $t_0$ is given by 
		\begin{equation*}
		|\gamma'|(t_0)=\lim_{h\to 0}\frac{|\gamma(t_0+h)-\gamma(t_0)|}{|h|}
		\end{equation*}
		if this limit exists. Note that by Rademacher's theorem $\dot{\gamma}(t_0)$ exists for $\mathcal{L}$-almost all $t_0$, and for those $t_0$ we have
		\begin{equation*}
		|\dot{\gamma}(t_0)|=|\gamma'|(t_0)=\textup{md}(\gamma,t_0)(1).
		\end{equation*}
		\item We will frequently identify the image of a path $\gamma:I\to\mathcal{C}$ with its parameterization, i.e., we simply write $\gamma$ instead of $\gamma(I)$ when no confusion is possible
		(for instance when we integrate over $\gamma(I)$).
		\item The \textit{Euclidean distance} between two sets $A,B\subset\R^n$ is denoted
		\begin{equation*}
		\textup{dist}(A,B)=\inf_{x\in A,y\in B}|x-y|.
		\end{equation*}
		We write $\textup{diam}(A)$ for the \textit{diameter} of $A$, i.e.,
		\begin{equation*}
		\textup{diam}(A)=\sup_{x,y\in A}|x-y|.
		\end{equation*}
		Moreover, $d_H(A,B)$ denotes the \textit{Hausdorff-distance} between $A$ and $B$, given by
		\begin{equation*}
		d_H(A,B)=\max\left(\sup_{x\in A}\textup{dist}(x,B),\sup_{y\in B}\textup{dist}(A,y)\right),
		\end{equation*}
		where we use the notation $\textup{dist}(x,B)=\textup{dist}(B,x)=\textup{dist}(\{ x \},B)$.
		\item For any set $A$ we write $1_A$ for the \emph{characteristic function},
		\begin{equation*}
		1_A(x)=\begin{cases*}
		1&if $x\in A$,\\
		0&else.
		\end{cases*}
		\end{equation*}
		\item For $x,y\in\mathcal{C}$ we define $[x,y]$ as the \textit{line segment} $\{ x+t(y-x)\,|\, t\in I \}$. The sets $(x,y],[x,y)$ and $(x,y)$ are defined similarly, e.g., $(x,y]=[x,y]\setminus\{ x \}$.
		\item For any function $f:X\to V$ with values in some normed vector space $(V,\|.\|)$ and $A\subset X$ we write
		\begin{equation*}
		|f|_{\infty,A}=\sup_{x\in A}\|f(x)\|.
		\end{equation*}
		%More precisely, $|.|_{\infty,A}$ corresponds to the \textit{uniform norm} (if restricted to bounded functions).
		%		\item The \textit{effective domain} of a convex function $f:\R\to\R\cup\{ \infty\}$ is denoted $\textup{dom}(f)=\{ x\in\R\,|\, f(x)<\infty \}$.
		\item A sequence $x:\mathbb{N}\to M$ of elements in some set $M$ will be indicated by the notation $(x_i)\subset M$ with $x_i=x(i)$. If $x_i$ actually stems from a subset $M_i$ we instead speak of a sequence $x_i\in M_i$.
		\item If a sequence of Lipschitz paths $(\gamma_j)\subset\Gamma$ converges uniformly to some $\gamma$, i.e., $|\gamma_j-\gamma|_{\infty,I}\to 0$, we write $\gamma_j\rightrightarrows\gamma$. 
	\end{itemize}
	\section{Wasserstein distance with generalized urban metric as min-cost flow}
	\label{sect2}
	In this section we prove \cref{BWfinal}, a Beckmann-type formula for the Wasserstein distance from \cref{WdisTp} between two probability measures $\mu_+,\mu_-$.
	Due to simple domain rescaling arguments we may assume $\mu_+,\mu_-$ to be supported in $\mathcal C$ without loss of generality (cf.\ \cite[Lem.\,2.4]{BW}).
	% $\mu_+,\mu_-$ can be interpreted as the initial and final distribution of some material or similar.
	Throughout the section we will fix $S\subset\mathcal{C}$ countably $1$-rectifiable and Borel measurable as well as $a\in [0,\infty ]$ and $b:S\to [0,a]$ lower semi-continuous.
	Therefore we may denote the (pseudo-)metric $d_{S,a,b}$ from \cref{gum} simply by $d$
	(while in \cref{sec4} we will return to the notation $d_{S,a,b}$ as $(S,b)$ varies in the urban planning problem).
	Recall that $S$ can be seen as a transportation network with a friction coefficient $b$ describing the necessary effort to move on the network, while motion outside the network is penalized by the parameter $a$. To simplify notation, we extend $b$ to $\mathcal{C}\setminus S$ with value $a$ so that we may write
	% We define the generalized urban metric as a pseudometric $d:\mathcal{C}\times\mathcal{C}\to [0,\infty]$ which describes the cost for travelling from one point to another. Hence, $d$ is given by
	\begin{equation*}
	%d:\mathcal{C}\times\mathcal{C}\to [0,\infty],\qquad
	d(x,y)=d_{S,a,b}(x,y)= \inf_{\gamma\in\Gamma^{xy}}\int_\gamma b\,\mathrm{d}\Ha.
	\end{equation*}
	%We write $d=d_{S,a,b}$ throughout \cref{sect2,subs6}, whereas in \cref{sec4} we will not use this notation, because the street network $(S,b)$ varies in the urban planning problem.
	% The Wasserstein $1$-distance between $\mu_+,\mu_-$ is now defined as
	% \begin{equation*}
	% W_d(\mu_+,\mu_-)=\inf_{\pi\in\Pi(\mu_+,\mu_-) }\int_{\mathcal{C}\times\mathcal{C}}d\,\mathrm{d}\pi,
	% \end{equation*}
	% where $\Pi(\mu_+,\mu_-)$ denotes the set of all probability measures $\pi$ on $\mathcal{C}\times\mathcal{C}$ with $(p_1)_\#\pi=\mu_+$ and $(p_2)_\#\pi=\mu_-$. The value $\pi (B_1\times B_2)$ describes to the amount of mass being transported from $B_1$ to $B_2$. Therefore, $W_d(\mu_+,\mu_-)$ is the total optimal transportation cost of the transfer from $\mu_+$ to $\mu_-$.
	A related important quantity is the length of a path, which in contrast to $d$ measures the travel distance with multiplicity.
	\begin{defin}[Path length associated with $d$]
		Let $S,a,b$ as above. For Lipschitz paths $\gamma :[t_1,t_2]\to\mathcal{C}$ we write the cost for travelling along $\gamma$ as
		\begin{equation*}
		L(\gamma )=\int_{[t_1,t_2]}b(\gamma )|\dot{\gamma }|\,\mathrm{d}\mathcal{L}.
		\end{equation*}
	\end{defin}
	% An intuitive result is the following expression for $d$ (cf.\ \cref{alt}):  
	% \begin{equation*}
	% d(x,y)=\inf_{\gamma\in\Gamma^{xy}}L(\gamma).
	% \end{equation*}
	In \cref{subs1} we will show that $L$ is lower semi-continuous in a certain sense. In \cref{subs3} we will then prove the intuitive statement $d(x,y)=\inf_{\gamma\in\Gamma^{xy}}L(\gamma)$ and exploit the lower semi-continuity of $L$ to show that there exists a minimizer $\gamma_{opt}\in\Gamma^{xy}$. Additionally, we will prove further properties of $d$ such as lower semi-continuity. A key consequence will be the existence of a Borel measurable path map $\rho$ with $d(x,y)=L(\rho (x,y))$ for all $x,y\in\mathcal{C}$ in \cref{subs4}. \Cref{subs5} then provides the proof of \cref{BWfinal}.
	\notinclude{the following formula for the Wasserstein distance between $\mu_+$ and $\mu_-$ is highlighted:
		\begin{equation*}
		W_d(\mu_+,\mu_-)=B^{\mu_+,\mu_-}=B_{S,a,b}^{\mu_+,\mu_-}=\inf_{\xi,\F}\int_S b|\xi|\,\mathrm{d}\Ha+a|\F|(\mathcal{C}),
		\end{equation*}
		where $\xi\in L^1(\Ha\mres S;\R^n)$ and $\F\in\mathcal{M}^n(\mathcal{C})$ with $\F\mres S=0$ satisfy $\textup{div}(\xi\Ha\mres S+\F)=\mu_+-\mu_-$. In other words, the Wasserstein distance corresponds to the total variation of an optimal mass flux between $\mu_+$ and $\mu_-$ (which exists if $W_d(\mu_+,\mu_-)$ is finite, cf.\ construction in the proof of \cref{BWfinal} in \cref{subs5}). Clearly, using the definition of $b$ and the fact that $\F$ is singular with respect to $\Ha\mres S$ we may also write
		\begin{equation*}
		B^{\mu_+,\mu_-}=\inf_{\xi,\F}\int_{\mathcal{C}} b\,\mathrm{d}|\xi\Ha\mres S+\F|.
		\end{equation*}
	}%\notinclude
	
	Throughout the section \cref{Sassump} will be made frequent use of. It is natural with regard to the urban planning problem. Indeed, for any pair $(S,b)$ with finite urban planning cost $\mathcal{U}^{\varepsilon,\mu_+,\mu_-}[S,b]$ it holds
	\begin{equation*}
	\infty>\mathcal{U}^{\varepsilon,\mu_+,\mu_-}[S,b]\geq\int_S\varepsilon (b)\,\mathrm{d}\Ha\geq\int_{S_\lambda}\varepsilon (b)\,\mathrm{d}\Ha\geq \varepsilon(\lambda)\Ha(S_\lambda)
	\end{equation*}
	for every $\lambda<a$, while $\varepsilon(\lambda)>0$ due to the relation $a=\inf\varepsilon^{-1}(0)$.
	\dualityApproach{\todo[inline]{If Section 3 stays, add commented sentences on alternative proof.}
		%\textcolor{blue}{In \cref{subs6} we will give an alternative proof of the Beckmann formula for the case $a<\infty$ without making use of \cref{Sassump}. By a density argument we show a version of the Kantorovich\textendash Rubinstein formula and apply Fenchel's duality theorem which then yields the desired expression.}
	}%\dualityApproach
	
	Let us now briefly collect some basic statements which will predominantly be used in \cref{sect2}. %\textcolor{blue}{and \cref{subs6}}.
	\begin{lem}[Lower semi-continuity of $\textup{Lip}$]
		\label{linf}
		Assume that $(\gamma_j)\subset\Gamma$ is a sequence such that $\gamma_j\rightrightarrows\gamma:I\to\mathcal{C}$. Then we obtain
		\begin{equation*}
		\textup{Lip}(\gamma)\leq\liminf_j\textup{Lip}(\gamma_j).
		\end{equation*}
	\end{lem}
	\begin{proof}
		For all $t_1,t_2\in I$ we have
		\begin{equation*}
		|\gamma (t_1)-\gamma (t_2)|\leq |\gamma (t_1)-\gamma_j(t_1)|+\textup{Lip}(\gamma_j)|t_1-t_2|+|\gamma_j(t_2)-\gamma (t_2)|
		\end{equation*}
		and thus
		\begin{equation*}
		|\gamma (t_1)-\gamma (t_2)|\leq |t_1-t_2|\liminf_j\textup{Lip}(\gamma_j).\qedhere
		\end{equation*}
	\end{proof}
	The next remark gives a relation between the measures of the image and the preimage of Lipschitz paths.
	\begin{rem}[$\Ha$-measure under Lipschitz paths]
		\label{measund}
		For any $\gamma\in\Gamma$ and $\mathcal{L}$-measurable $A\subset I$ we have (cf.\ \cite[Thm.\,7.5]{M})
		\begin{equation*}
		\Ha(\gamma(A))\leq\int_A|\dot{\gamma}|\,\mathrm{d}\mathcal{L}\leq\textup{Lip}(\gamma)\mathcal{L}(A).
		\end{equation*}
		If $\gamma$ is injective and has constant speed, then we directly get \cite[pp.\,241-244]{F}
		\begin{equation*}
		\Ha(\gamma(A))=\textup{Lip}(\gamma)\mathcal{L}(A).
		\end{equation*} 
	\end{rem}
	The content of the next remark follows directly from \cite[Thm.\,2]{MM}. Recall that $\mathcal{H}^0$ is the counting measure.
	\begin{rem}[Area formula]
		\label{cov}
		Let $\gamma\in\Gamma$ and $f:\gamma\to\R$ Borel measurable. For any $\Le$-measurable $A\subset[0,1]$ we have
		\begin{equation*}
		\int_{\gamma(A)}f(x)\mathcal{H}^0(\gamma^{-1}(x)\cap A)\,\mathrm{d}\Ha(x)=\int_Af(\gamma)|\dot{\gamma}|\,\mathrm{d}\mathcal{L}
		\end{equation*}
		if one of the two sides is well-defined. For $\Ha(\gamma(A))=0$ we obtain $\dot{\gamma}=0$ $\mathcal{L}$-almost everywhere in $A$.
		Moreover, for injective $\gamma$ we obtain
		\begin{equation*}
		\int_{\gamma(A)}f(x)\,\mathrm{d}\Ha(x)=\int_Af(\gamma)|\dot{\gamma}|\,\mathrm{d}\mathcal{L},
		\qquad\text{in particular}\quad
		\Ha(\gamma(A))=\int_A|\dot{\gamma}|\,\mathrm{d}\mathcal{L}.
		\end{equation*}
	\end{rem}
	We finally remind the reader of the following compactness result.
	\begin{rem}[Arzel\`{a}\textendash Ascoli theorem]
		\label{unisub}
		Let $C>0$ and $(\gamma_j)\subset\Gamma_C$ be a sequence. The $\gamma_j$ are uniformly equicontinuous by Rademacher's theorem,
		\begin{equation*}
		\textup{Lip}(\gamma_j)=\esssup_{I}|\dot{\gamma}_j|\leq C,
		\end{equation*}
		%Here the $N_j$ denote the $\mathcal{L}$-null sets of points at which the $\gamma_j$ are not differentiable (Rademacher's theorem).
		and they are pointwise bounded. The Arzel\`{a}\textendash Ascoli Theorem thus implies $\gamma_j\rightrightarrows\gamma$ up to a subsequence. Additionally, by \cref{linf} we have
		\begin{equation*}
		\textup{Lip}(\gamma)\leq\liminf_j\textup{Lip}(\gamma_j)\leq C
		\end{equation*}
		$\mathcal{L}$-almost everywhere and thus $\gamma\in\Gamma_C$.
	\end{rem}
	\subsection{Properties of the path length $L$}
	\label{subs1}
	The following statement will be the main result of this section.
	\begin{thm}[Lower semi-continuity property of $L$]
		\label{Lalmostlow}
		Let \cref{Sassump} be satisfied and $(\gamma_j)\subset\Gamma_C$ be a sequence of paths with constant speed such that $\gamma_j\rightrightarrows\gamma$. Then we have
		\begin{equation*}
		L(\gamma)\leq\liminf_jL(\gamma_j).
		\end{equation*}
	\end{thm}
	We first prove a version of Go\l\k{a}b's theorem for images under Lipschitz paths following the proof of \cite[Thm.\,3.3]{PS}.
	% Afterwards, we will need \cref{Sassump} to prove \cref{Lalmostlow}.
	The next lemma will be helpful.
	\begin{lem}[]
		\label{vrx}
		Let $J\subset I$ be an interval and $\gamma\in\Gamma$. Then for $\Ha$-almost all $x\in\gamma(J) $ there exists some $\delta>0$ and a function $\varrho_x:[-\delta,\delta]\to\gamma(J)$ such that
		\begin{enumerate}
			\item $|\dot{\varrho}_x|=1$ $\mathcal{L}$-almost everywhere on $[-\delta,\delta]$,
			\item $\dot{\varrho}_x(0)$ exists and $\varrho_x(0)=x$,
			\item \label{point1} for all $\varepsilon>0$ there is some $r>0$ such that
			\begin{equation*}
			|t_1-t_2|-r\varepsilon\leq|\varrho_x(t_1)-\varrho_x(t_2)|
			\end{equation*}
			for all $t_1,t_2\in[-r,r]$.
		\end{enumerate}
	\end{lem}
	\begin{proof}
		Let $\eta:[0,l]\to\mathcal{C}$ be a reparameterization of $\gamma$ by arc length, thus $l=\textup{len}(\gamma)$. By \cite[Thm.\,2]{Kirch} there exists a $\mathcal{L}$-null set $N_1\subset[0,l]$ such that for all $t\in[0,l]\setminus N_1$ the metric differential $\textup{md}(\eta,t)$ is a seminorm on $\R$ and
		\begin{equation*}
		|\eta(t_1)-\eta(t_2)|-\textup{md}(\eta,t)(t_1-t_2)=o(|t_1-t|+|t_2-t|)
		\end{equation*}
		for all $t_1,t_2\in[0,l]$. Moreover, by Rademacher's theorem there is some $\mathcal{L}$-null set $N_2\subset[0,l]$ such that $\dot{\eta}$ exists on $[0,l]\setminus N_2$. $N=N_1\cup N_2\cup\{ 0,l \}$ clearly satisfies $\Ha(\eta( N))\leq\textup{Lip}(\eta)\mathcal{L}( N)=0$. Fix any $x\in\gamma(J)\setminus\eta(N)$ and $t_x\in[0,l]\setminus N$ with $\eta(t_x)=x$. Choose $\delta>0$ sufficiently small such that $[t_x-\delta,t_x+\delta]\subset[0,l]$ and define $\varrho_x:[-\delta,\delta]\to\gamma(J)$ by $\varrho_x(t)=\eta(t_x+t)$. The first two statements follow from the properties of $\eta$\notinclude{ (we may also apply \cite[Thm.\,4.4.5]{AT} instead)}. Moreover, by the above identity we observe
		\begin{equation*}
		|\varrho_x(t_1)-\varrho_x(t_2)|-\textup{md}(\varrho_x,0)(t_1-t_2)=|\eta(t_x+t_1)-\eta(t_x+t_2)|-\textup{md}(\eta,t_x)(t_1-t_2)=o(|t_1|+|t_2|)
		\end{equation*}
		for all $t_1,t_2\in[-\delta,\delta]$. Thus, for every $\varepsilon>0$ there exists some $r>0$ such that
		\begin{equation*}
		|\varrho_x(t_1)-\varrho_x(t_2)|\geq\textup{md}(\varrho_x,0)(t_1-t_2)-\frac\varepsilon2(|t_1|+|t_2|)
		\end{equation*}
		for all $t_1,t_2\in[-r,r]$. Using $\textup{md}(\varrho_x,0)(t_1-t_2)=|t_1-t_2||\dot{\eta}(t_x)|=|t_1-t_2|$ we get
		\begin{equation*}
		|\varrho_x(t_1)-\varrho_x(t_2)|\geq-\frac\varepsilon2(|t_1|+|t_2|)+|t_1-t_2|\geq-\varepsilon r+|t_1-t_2|.\qedhere
		\end{equation*}
	\end{proof}
	Next, we use the idea in \cite[Thm.\,3.3]{PS} to prove a version Go\l\k{a}b's theorem for Lipschitz images of finitely many relatively open intervals in $I$ (\cref{versionGolab}). The result for general sets will then follow immediately by the regularity of the Lebesgue measure (\cref{Golabim}).
	\begin{lem}[Version of Go\l\k{a}b's theorem]
		\label{versionGolab}
		Let $(\gamma_j)\subset\Gamma_C$ with $\gamma_j\rightrightarrows\gamma\in\Gamma_C$. Further, let $O_1,\ldots,O_k\subset I$ be relatively open intervals. Then we have
		\begin{equation*}
		\Ha(\gamma(O_1\cup\ldots\cup O_k))\leq\liminf_j\Ha(\gamma_j(O_1\cup\ldots\cup O_k)).
		\end{equation*} 
	\end{lem}
	\begin{rem}[Straight limit path]
		If $\gamma$ maps onto a straight line $\ell=\textup{span}(p)$ for some $p\in \mathcal{S}^{n-1}$, then it is easy to see that
		\begin{equation*}
		\Ha(\gamma_j(O_1\cup\ldots\cup O_k))\geq\Ha( \textup{proj}_\ell(\gamma_j(O_1\cup\ldots\cup O_k)) )\to\Ha(\gamma(O_1\cup\ldots\cup O_k)),
		\end{equation*}
		where $\textup{proj}_\ell$ denotes the orthogonal projection onto $\ell$.% Hence, one could perform the following argument also by projecting all paths onto $\ell$.
	\end{rem}
	\begin{proof}[Proof of \cref{versionGolab}]
		We follow the proof of \cite[Thm.\,3.3]{PS}.~It is easy to see that $\Ha(\gamma(O_1\cup\ldots\cup O_k))=\Ha(\gamma(\overline{O}_1\cup\ldots\cup \overline{O}_k))$. Thus, we can replace $O_1\cup\ldots\cup O_k$ by a union of closed intervals $J=J_1\cup\ldots\cup J_l\subset I$. Furthermore, without loss of generality we may assume $\textup{diam}(\gamma(J_i))>0$ for $i=1,\ldots,l$ since by ignoring a $J_i$ with $\textup{diam}(\gamma(J_i))=0$ we only decrease the right-hand side of the inequality to be proved, while the left-hand side stays unchanged. Define a sequence of Radon measures by $\mu_j(B)=\Ha(\gamma_j(J)\cap B)$ for $B\in\mathcal{B}(\R^n)$. Clearly, the $\mu_j$ are uniformly bounded by $C$, and the Banach\textendash Alaoglu theorem implies $\mu_j\ws\mu$ up to a subsequence. For $\Ha$-almost all $x\in\gamma(J)$ we can choose $\varrho_x$ as in \cref{vrx}. Fix any such $x$, and for every $\varepsilon>0$ let $r=r(\varepsilon)>0$ as in the third point of \cref{vrx}. The set $C_j=\gamma_j(J)\cap \overline{B_r(x)}$ is compact, and $\textup{dist}(\varrho_x(t),C_j)\leq r\varepsilon$ for all $t\in[-r,r]$ and $j$ sufficiently large due to the uniform convergence of the $\gamma_j$. Letting $r<\min_i\textup{diam}(\gamma(J_i))/2$ we have $C_j\cap\partial B_r(x)\neq\emptyset$ for large $j$. We can now apply \cite[Lem.\,3.2]{PS} which yields $\Ha(C_j)\geq 2r-9r\varepsilon$. We next use the Portmanteau Theorem to get the desired result. We have
		\begin{equation*}
		\mu(\overline{B_r(x)})\geq\limsup_j\mu_j(\overline{B_r(x)})=\limsup_j\Ha(C_j)\geq 2r-9r\varepsilon.
		\end{equation*}
		By $r=r(\varepsilon)\to 0$ for $\varepsilon\to 0$ we thus obtain
		\begin{equation*}
		\Theta^{*1}(\mu,x)=\lim_{\varepsilon\to 0}\frac{\mu(\overline{B_r(x)})}{2r}\geq 1.
		\end{equation*}
		This holds for $\Ha$-almost every $x\in\gamma(J)$. Hence, we end up with
		\begin{equation*}
		\Ha(\gamma(J))\leq\mu(\gamma(J))\leq\mu(\R^n)\leq\liminf_j\mu_j(\R^n)=\liminf_j\Ha(\gamma_j(J)),
		\end{equation*}
		where we used \cite[Ch.\,1, Thm.\,3.3]{LS} in the first inequality.
	\end{proof}
	\begin{prop}[Go\l\k{a}b's theorem for images of Lipschitz paths]
		\label{Golabim}
		Let $(\gamma_j)\subset\Gamma_C$ with $\gamma_j\rightrightarrows\gamma$. Then we have
		\begin{equation*}
		\Ha(\gamma(T))\leq\liminf_j\Ha(\gamma_j(T))
		\end{equation*}
		for all $\mathcal{L}$-measurable sets $T\subset I$.
	\end{prop}
	\begin{proof}
		Let $\varepsilon>0$. By the regularity of $\mathcal{L}$ we can choose some relatively open set $O\subset I$ such that $T\subset O$ and $\mathcal{L}(O\setminus T)<\varepsilon/(2C)$. Write $O$ as a countable union of relatively open and connected sets $O_i\subset I$ and choose $N$ sufficiently large such that $\tilde{O}=O_1\cup\ldots\cup O_N$ satisfies $\mathcal{L}(O\setminus\tilde{O})<\varepsilon/(2C)$. This procedure is possible due to the $\sigma$-continuity of $\mathcal{L}$. By \cref{measund} we have
		\begin{equation*}
		\Ha(\gamma_j(\tilde{O}))\leq\Ha(\gamma_j(O))\leq\Ha(\gamma_j(T))+\Ha(\gamma_j(O\setminus T))\leq\Ha(\gamma_j(T))+\textup{Lip}(\gamma_j)\mathcal{L}(O\setminus T)\leq\Ha(\gamma_j(T))+\varepsilon/2
		\end{equation*}
		and therefore using \cref{versionGolab} (and again \cref{measund})
		\begin{equation*}
		\liminf_j\Ha(\gamma_j(T))\geq\Ha(\gamma(\tilde{O}))-\varepsilon/2\geq\Ha(\gamma(O))-\Ha(\gamma(O\setminus\tilde{O}))-\varepsilon/2\geq\Ha(\gamma(T))-\textup{Lip}(\gamma)\mathcal{L}(O\setminus\tilde{O})-\varepsilon/2\geq\Ha(\gamma(T))-\varepsilon.
		\end{equation*}
		The result now follows from the arbitrariness of $\varepsilon$.
	\end{proof}
	We can now prove our main result for this section, the lower semi-continuity property of $L$ .
	\begin{proof}[Proof of \cref{Lalmostlow}]
		By \cref{linf} we have $|\dot\gamma|\leq\textup{Lip}(\gamma)\leq\liminf_j\textup{Lip}(\gamma_j)=\liminf_j|\dot\gamma_j|$ $\mathcal{L}$-almost everywhere.
		Furthermore, the lower semi-continuity of $b$ on $S$ and $b\leq a$ imply $b(\gamma(t))\leq\liminf_j b(\gamma_j(t))$ for all $t\in\gamma^{-1}(S)$.
		Thus, with Fatou's lemma we obtain
		\begin{equation*}
		\int_{\gamma^{-1}(S)}b(\gamma)|\dot{\gamma}|\,\mathrm{d}\mathcal{L}
		%\leq\int_{\gamma^{-1}(S)}1_S(\gamma )b(\gamma)\liminf_j|\dot{\gamma}_j|\,\mathrm{d}\mathcal{L}
		\leq\int_{\gamma^{-1}(S)}\liminf_jb(\gamma_j)\liminf_j|\dot{\gamma}_j|\,\mathrm{d}\mathcal{L}
		\leq\liminf_j\int_{\gamma^{-1}(S)}b(\gamma_j)|\dot{\gamma}_j|\,\mathrm{d}\mathcal{L}.
		\end{equation*} 
		Hence we have
		\begin{equation*}
		L(\gamma)\leq\liminf_j\int_{\gamma^{-1}(S)}b(\gamma_j)|\dot{\gamma}_j|\,\mathrm{d}\mathcal{L}+\int_{\gamma^{-1}(\mathcal{C}\setminus S)}b(\gamma)|\dot{\gamma}|\,\mathrm{d}\mathcal{L}
		\end{equation*}
		so that it suffices to show
		\begin{equation*}
		\int_Tb(\gamma)|\dot{\gamma}|\,\mathrm{d}\mathcal{L}\leq\liminf_j\int_Tb(\gamma_j)|\dot{\gamma}_j|\,\mathrm{d}\mathcal{L}
		\end{equation*}
		for $T=\gamma^{-1}(\mathcal{C}\setminus S)\cap\{ \dot{\gamma}\textup{ exists and }\dot{\gamma}\neq 0 \}$. Assume to the contrary that 
		\begin{equation*}
		\liminf_j\int_Tb(\gamma_j)|\dot{\gamma}_j|\,\mathrm{d}\mathcal{L}<\int_Tb(\gamma)|\dot{\gamma}|\,\mathrm{d}\mathcal{L}=\int_Ta|\dot{\gamma}|\,\mathrm{d}\mathcal{L},
		\end{equation*}
		where by restricting to a subsequence we may assume the limit inferior to actually be a limit.
		
		We first show that in this inequality we may actually replace $T$ with a subset $A\subset T$ on which $\gamma$ is injective.
		Indeed, if $a=\infty$ (thus the right-hand side is infinite and the left-hand side finite) we may simply pick $A=\{t\in T\,|\,\gamma(t)\notin\gamma(T\cap[0,t))\}$ and obtain
		\begin{equation*}
		\liminf_j\int_Ab(\gamma_j)|\dot{\gamma}_j|\,\mathrm{d}\mathcal{L}
		\leq\liminf_j\int_Tb(\gamma_j)|\dot{\gamma}_j|\,\mathrm{d}\mathcal{L}
		<\int_Ta|\dot{\gamma}|\,\mathrm{d}\mathcal{L}
		=\infty
		=\int_Aa|\dot{\gamma}|\,\mathrm{d}\mathcal{L},
		\end{equation*}
		where the last equality follows from $\int_A|\dot\gamma|\,\mathrm d\mathcal{L}>0$
		(otherwise $0=\int_A|\dot\gamma|\,\mathrm d\mathcal{L}=\Ha(\gamma(A))=\Ha(\gamma(T))$ due to $\gamma(A)=\gamma(T)$,
		which by \cref{cov} contradicts $\dot\gamma\neq0$ on $T$).
		If $a<\infty$, on the other hand, the functions $f_j=b(\gamma_j)|\dot{\gamma}_j|1_T$ are essentially bounded by $aC$, and thus $f_j\ws f\in L^\infty(I)$ for some subsequence using the Banach\textendash Alaoglu theorem. More precisely, for each $g\in L^1(I)$ we have
		\begin{equation*}
		\int_Tb(\gamma_j)|\dot{\gamma}_j|g\,\mathrm{d}\mathcal{L}\to\int_Tfg\,\mathrm{d}\mathcal{L},
		\quad\text{ in particular }
		\int_Tb(\gamma_j)|\dot{\gamma}_j|\,\mathrm{d}\mathcal{L}\to\int_Tf\,\mathrm{d}\mathcal{L}.
		\end{equation*}
		By our assumption, $\tilde T=\{t\in T\,|\,f(t)<a|\dot\gamma(t)|\}$ has positive Lebesgue measure.
		Likewise, $A=\{t\in\tilde T\,|\,\gamma(t)\notin\gamma(\tilde T\cap[0,t))\}$ has positive Lebesgue measure
		(again, otherwise $0=\int_A|\dot\gamma|\,\mathrm d\mathcal{L}=\Ha(\gamma(A))=\Ha(\gamma(\tilde T))$, contradicting $\dot\gamma\neq0$ on $\tilde T$).
		Therefore
		\begin{equation*}
		\lim_j\int_Ab(\gamma_j)|\dot{\gamma}_j|\,\mathrm{d}\mathcal{L}
		=\int_Af\,\mathrm{d}\mathcal{L}
		<\int_Aa|\dot{\gamma}|\,\mathrm{d}\mathcal{L},
		\end{equation*}
		as desired.
		
		We will now derive a contradiction.
		Let us set
		\begin{equation*}
		\lambda_0=\lim_j\int_Ab(\gamma_j)|\dot\gamma_j|\mathrm{d}\mathcal{L}\Big/\int_A|\dot\gamma|\,\mathrm{d}\mathcal L.
		\end{equation*}
		Since $\lambda_0<a$ by assumption, we can pick another $\lambda_1\in(\lambda_0,a)$.
		For all $\delta>0$ it holds (\cref{cov})
		\begin{equation*}
		\lambda_1\int_{A\cap \gamma_j^{-1}\left(\mathcal{C}\setminus S_{\lambda_1}\right)}|\dot{\gamma}_j|\,\mathrm{d}\mathcal{L}
		\leq\int_Ab(\gamma_j)|\dot{\gamma}_j|\,\mathrm{d}\mathcal{L}
		<\lambda_0\int_A|\dot\gamma|\,\mathrm{d}\mathcal L+\delta
		=\lambda_0\Ha(\gamma(A))+\delta
		\end{equation*}
		for $j$ large enough. Furthermore, by our version of Go\l\k{a}b's theorem (\cref{Golabim}) we have
		\begin{equation*}
		\Ha(\gamma(A))\leq\liminf_j\Ha(\gamma_j(A)).
		\end{equation*}
		Now choose $\delta,\varepsilon>0$ with $(\lambda_1-\lambda_0)\Ha(\gamma(A))-\delta>\varepsilon\lambda_1$ so that (using \cref{measund})
		\begin{multline*}
		\Ha(\gamma_j(A)\cap S_{\lambda_1})
		=\Ha(\gamma_j(A))-\Ha(\gamma_j(A)\setminus S_{\lambda_1})\\
		\geq\Ha(\gamma_j(A))-\int_{A\cap \gamma_j^{-1}\left(\mathcal{C}\setminus S_{\lambda_1}\right)}|\dot{\gamma}_j|\,\mathrm{d}\mathcal{L}
		\geq\Ha(\gamma(A))-\frac{\lambda_0}{\lambda_1}\Ha(\gamma(A))-\frac{\delta}{\lambda_1}
		>\varepsilon
		\end{multline*}
		for all $j$ large enough. On the other hand, using the regularity of $\mathcal{L}$ we find compact sets $K\subset A$ and $K_j\subset K\cap\gamma_j^{-1}(S_{\lambda_1})$ such that
		\begin{equation*}
		\mathcal{L}(A\setminus K)<\frac{\varepsilon}{4C}\textup{\qquad and\qquad}\mathcal{L}((K\cap\gamma_j^{-1}(S_{\lambda_1}))\setminus K_j)<\frac{\varepsilon}{4C}.
		\end{equation*}
		We have $\gamma_j(K_j)\subset S_{\lambda_1}$ and $\gamma(K)\cap S=\emptyset$ and therefore $\gamma_j(K_j)\cap\gamma(K)=\emptyset$ for all $j$. Let $d_j=\textup{dist}(\gamma_j(K_j),\gamma(K))$. By the uniform convergence of the $\gamma_j$ to $\gamma$ we can pick a subsequence such that $d_H(\gamma_j,\gamma)<d_{j-1}$ for all $j$. Hence, the $\gamma_j(K_j)$ are pairwise disjoint and thus by \cref{Sassump}
		\begin{equation*}
		\infty>\Ha(S_{\lambda_1})\geq\sum_j\Ha(\gamma_j(K_j)),
		\end{equation*}
		which implies $\Ha(\gamma_j(K_j))\to 0$. Finally, we get
		\begin{multline*}
		\Ha(\gamma_j(A)\cap S_{\lambda_1})\leq\Ha(\gamma_j(K_j))+\Ha(\gamma_j(K\setminus K_j)\cap S_{\lambda_1})+\Ha(\gamma_j(A\setminus K)\cap S_{\lambda_1})\\
		\leq\Ha(\gamma_j(K_j))+\textup{Lip}(\gamma_j)\mathcal{L}((K\cap\gamma_j^{-1}(S_{\lambda_1}))\setminus K_j)+\textup{Lip}(\gamma_j)\mathcal{L}(A\setminus K)
		\leq\Ha(\gamma_j(K_j))+\varepsilon/2
		\end{multline*}
		and consequently
		\begin{equation*}
		\liminf_j\Ha(\gamma_j(A)\cap S_{\lambda_1})\leq\varepsilon/2,
		\end{equation*}
		which is the desired contradiction.
	\end{proof}
	In the case $a=\infty$ a simpler proof is actually possible: The main idea is to show that $\liminf_jL(\gamma_j)<\infty$ implies $\Ha(\gamma\setminus S)=0$ and therefore
	\begin{equation*}
	L(\gamma)=\int_Ib(\gamma)|\dot{\gamma}|\,\mathrm{d}\mathcal{L}
	=\int_{\gamma^{-1}(S)}b(\gamma)|\dot{\gamma}|\,\mathrm{d}\mathcal{L}.
	\end{equation*}
	We will then use $b(\gamma)\leq\liminf_jb(\gamma_j)$ on $S$, which is true due to the lower semi-continuity of $b$ on $S$, to get the desired result. To prove that $\liminf_jL(\gamma_j)<\infty$ implies $\Ha(\gamma\setminus S)=0$ we need the following two results.
	\begin{lem}[Curves intersect $S$]
		\label{lamdel}
		Let $a=\infty$. If $(\gamma_j)\subset\Gamma$ is a sequence with $L(\gamma_j)$ uniformly bounded, then for each $\delta>0$ there exists some $\lambda\in [0,\infty)$ such that
		\begin{equation*}
		\Ha(\gamma_j\setminus S_\lambda)\leq\delta
		\qquad\text{for all }j.
		\end{equation*}
	\end{lem}
	\begin{proof}
		For fixed $\delta>0$ and $\lambda\in (0,\infty)$ sufficiently large we have $\lambda\delta\geq L(\gamma_j)$ for all $j$. Additionally, by \cref{cov} we get 
		\begin{equation*}
		\lambda\Ha(\gamma_j\setminus S_\lambda)\leq \int_{\gamma_j\setminus S_\lambda }b\,\mathrm{d}\Ha\leq\int_{\gamma_j^{-1}(\mathcal{C}\setminus S_\lambda) }b(\gamma_j)|\dot{\gamma}_j|\,\mathrm{d}\mathcal{L}\leq L(\gamma_j)\leq\lambda\delta.\qedhere
		\end{equation*}
	\end{proof}
	\begin{prop}[Symmetric difference with limit path]
		\label{help}
		Let \cref{Sassump} be satisfied, $a=\infty$ and $(\gamma_j)\subset\Gamma_C$ with $L(\gamma_j)$ uniformly bounded. If $\gamma_j\rightrightarrows\gamma\in\Gamma_C$, then for any closed interval $J\subset I$ we have
		\begin{equation*}
		\Ha(\gamma_j(J)\setminus\gamma(J))\to 0
		\qquad\text{and}\qquad
		\Ha(\gamma(J)\setminus\gamma_j(J))\to 0.
		\end{equation*}
	\end{prop}
	\begin{proof}
		We prove the result for $J=I$, the general case then simply follows from considering reparameterizations of $\gamma_j(J),\gamma(J)$ as paths in $\Gamma_C$.
		We begin with the first limit, $\Ha(\gamma_j\setminus\gamma )\to 0$. Define $A_j= \{ t\in I\,|\,\gamma_j(t)\notin\gamma  \}$. For a contradiction, we assume that $\Ha(\gamma_j(A_j))>\delta$ (along a subsequence) for some $\delta>0$. The set $A_j=\gamma_j^{-1}(\R^n\setminus\gamma )$ is open in $I$. Hence, there exist closed sets $B_j\subset A_j$ such that $\Ha(\gamma_j(A_j\setminus B_j))<\delta/2$ (using that the arc length of $\gamma_j$ is bounded and the $\sigma$-continuity of $\Ha$) and thus $\Ha(\gamma_j(B_j))>\delta/2$. By choice of the $B_j$ we have 
		\begin{equation*}
		d_j=\textup{dist}(\gamma_j(B_j),\gamma )>0.
		\end{equation*}
		By $\gamma_j\rightrightarrows\gamma$ we can assume that $d_H(\gamma_{j+1},\gamma )<d_j$ for all $j$ by restricting to a subsequence. Thus, we get $\gamma_j(B_j)\cap\gamma_k(B_k)=\emptyset$ for $j\neq k$ by construction. Invoking \cref{lamdel} there exists some $\lambda\in[0,\infty)$ such that $\Ha(\gamma_j\setminus S_\lambda )<\delta/4$ for all $j$. Hence, we have
		\begin{equation*}
		\Ha(\gamma_j(B_j)\cap S_\lambda)=\Ha(\gamma_j(B_j))-\Ha(\gamma_j(B_j)\setminus S_\lambda)>\delta/4.
		\end{equation*}
		This yields the desired contradiction,
		\begin{equation*}
		\infty>\Ha(S_\lambda)\geq\Ha\left( \dot{\bigcup}_j\gamma_j(B_j)\cap S_\lambda\right)=\sum_j\Ha(\gamma_j(B_j)\cap S_\lambda)>\sum_j\delta/4=\infty.
		\end{equation*}
		As for the second limit, we note
		\begin{equation*}
		\limsup_j\Ha(\gamma\setminus\gamma_j)
		=\limsup_j\left(\Ha(\gamma)+\Ha(\gamma_j\setminus\gamma)-\Ha(\gamma_j)\right)
		=\Ha(\gamma)-\liminf_j\Ha(\gamma_j)
		\leq0,
		\end{equation*}
		where the inequality holds by Go\l\k{a}b's theorem (see for instance \cite[Thm.\,3.2]{BPSS} or our version \cref{Golabim}).
	\end{proof}
	\begin{prop}[Limit of paths with uniformly bounded costs]
		\label{corro}
		Let \cref{Sassump} be satisfied, $a=\infty$ and $(\gamma_j)\subset\Gamma_C$ with $L(\gamma_j)$ uniformly bounded. Assume that $\gamma_j\rightrightarrows\gamma\in\Gamma_C$. Then for each $\delta>0$ there exists a $\lambda\in[0,\infty)$ such that
		\begin{equation*}
		\Ha(\gamma\setminus S_\lambda)\leq\delta.
		\end{equation*}
		In particular, we have $\Ha(\gamma\setminus S)=0$.
	\end{prop}
	\begin{proof}
		Given $\delta>0$, take $\lambda\in[0,\infty)$ from \cref{lamdel}, then
		\begin{equation*}
		\Ha(\gamma\setminus S_\lambda)
		\leq\Ha(\gamma\setminus\gamma_j)+\Ha(\gamma_j\setminus S_\lambda)
		\leq\Ha(\gamma\setminus\gamma_j)+\delta.
		\end{equation*}
		The limit $j\to\infty$ together with \cref{help} now implies the result.
	\end{proof}
	Note that \cref{lamdel,help,corro} do not hold for $a<\infty$.
	We can now give an alternative proof of \cref{Lalmostlow} for the case $a=\infty$.
	\begin{proof}[Alternative proof of \cref{Lalmostlow} for $a=\infty$]
		If $\liminf_jL(\gamma_j)=\infty$, then there is nothing to show. Hence, we can assume $\liminf_jL(\gamma_j)<\infty$, and it is enough to prove the claim for a subsequence such that $\liminf_jL(\gamma_j)=\lim_jL(\gamma_j)$, which means that $L(\gamma_j)$ is uniformly bounded. \Cref{corro} implies $\Ha(\gamma\setminus S)=0$. We now invoke \cref{cov} and get $\dot{\gamma}=0$ $\mathcal{L}$-almost everywhere on $\gamma^{-1}(\mathcal{C}\setminus S)$. This yields the desired result,
		\begin{align*}
		L(\gamma)
		=\int_{\gamma^{-1}(S)}b(\gamma)|\dot{\gamma}|\,\mathrm{d}\mathcal{L}
		\leq\int_{\gamma^{-1}(S)}\liminf_jb(\gamma_j)\liminf_j|\dot{\gamma}_j|\,\mathrm{d}\mathcal{L}
		\leq\liminf_j\int_{\gamma^{-1}(S)}b(\gamma_j)|\dot{\gamma}_j|\,\mathrm{d}\mathcal{L}
		\leq\liminf_jL(\gamma_j)
		\end{align*}  
		using Fatou's lemma and the lower semi-continuity of $b$ on $S$ as well as $|\dot\gamma|\leq\liminf_j|\dot\gamma_j|$ $\mathcal{L}$-almost everywhere.
	\end{proof}
	\subsection{Properties of the generalized urban metric}
	\label{subs3}
	We now derive properties of the generalized urban metric based on the previous analysis of the path length.
	The following result is due to the fact that $d(x,y)$ can be written as an infimum over injective paths.
	\begin{lem}[Alternative formula for $d$]
		\label{alt}
		We have
		\begin{equation*}
		d(x,y)=\inf_{\gamma\in\Gamma^{xy}}L(\gamma)
		\end{equation*}
		for all $x,y\in\mathcal{C}$.
	\end{lem}
	\begin{proof}
		Clearly, the claim is true for $x=y$. 
		Furthermore, by \cref{cov} for any $\gamma\in\Gamma^{xy}$ we have
		\begin{equation*}
		d(x,y)\leq\int_{\gamma}b\,\mathrm{d}\Ha\leq\int_{I}b(\gamma)|\dot{\gamma }|\,\mathrm{d}\mathcal{L}=L(\gamma)
		\end{equation*}
		so that the claim holds as well for $d(x,y)=\infty$.
		Thus, we can assume $d(x,y)<\infty$ and $x\neq y$. Let $\gamma\in\Gamma^{xy}$ such that
		\begin{equation*}
		\int_{\gamma}b\,\mathrm{d}\Ha<\infty. 
		\end{equation*}
		By \cite[Lem.\,3.1]{Fal} there exists a continuous injection $\psi:I\to\mathcal{C}$ such that $\psi\subset\gamma$ and $\psi (0)=x,\psi (1)=y$. Obviously, the arc length of $\psi$ is bounded by $\textup{Lip}(\gamma)$.
		%We show that the length of $\psi$ is finite. If $a=\infty$, then
		%\begin{equation*}
		%\Ha(\psi)=\Ha(\psi\cap S_1)+\int_{\psi\cap S_1^\mathrm{C}}\,\mathrm{d}\Ha\leq\Ha(\psi\cap S_1)+\int_{\psi}b\,\mathrm{d}\Ha\leq \Ha(\psi\cap S_1)+\int_{\gamma}b\,\mathrm{d}\Ha<\infty.
		%\end{equation*}
		%For the case $a<\infty$ and some $\varepsilon>0$ with $a-\varepsilon>0$ we get
		%\begin{equation*}
		%\Ha(\psi)=\Ha(\psi\cap S_{a-\varepsilon})+\Ha(\psi\cap S_{a-\varepsilon}^\mathrm{C})\leq\Ha(\psi\cap S_{a-\varepsilon})+\frac{1}{a}\int_{\psi}b+\varepsilon\dHa\leq \Ha(\psi\cap S_{a-\varepsilon})+\frac{1}{a}\int_{\gamma}b\,\mathrm{d}\Ha+\frac{\varepsilon}{a}\int_I|\dot{\gamma}|\,\mathrm{d}\mathcal{L}<\infty.
		%\end{equation*}
		Hence, we can assume that $\psi$ is Lipschitz continuous. Finally, by the injectivity and \cref{cov}
		\begin{equation*}
		L(\psi)=\int_{\psi}b\,\mathrm{d}\Ha\leq\int_{\gamma}b\,\mathrm{d}\Ha.\qedhere
		\end{equation*}
	\end{proof}
	The next statement shows that $d(x,y)=\inf_{\gamma\in\Gamma^{xy}} L(\gamma)$ admits a minimizer $\gamma_{xy}$ which satisfies $\textup{len}(\gamma_{xy})\leq C_1+C_2d(x,y)$ with constants $C_1,C_2 >0$ that do not depend on $x,y$. We will need this result in \cref{subs4} to show the existence of an optimal measurable path selection.
	\begin{prop}[Existence and arc length of minimizer for $d(x,y)$]
		\label{minimiz}
		Let \cref{Sassump} be satisfied. For $x,y\in\mathcal{C}$ the problem $d(x,y)=\inf_{\gamma\in\Gamma^{xy}}L(\gamma)$ has a minimizer if $d(x,y)$ is finite. Moreover, at least one minimizer $\psi$ is injective and satisfies
		\begin{equation*}
		\textup{len}(\psi)\leq\begin{cases*}
		\Ha(S_1)+d(x,y)&if $a=\infty$,\\
		\Ha(S_{a/2})+\frac{2}{a}d(x,y)& if $a<\infty$.
		\end{cases*}
		\end{equation*}
	\end{prop}
	\begin{proof}
		We proceed by the direct method in the calculus of variations. Let $(\gamma_j)\subset\Gamma^{xy}$ be a sequence with $L(\gamma_j)\searrow\inf_{\gamma\in\Gamma^{xy}}L(\gamma)$. By \cite[Lem.\,3.1]{Fal} we can assume that each $\gamma_j$ is injective (this does not increase $L(\gamma_j)$). If $a=\infty$, then
		\begin{equation*}
		\textup{len}(\gamma_j)=\Ha(\gamma_j\cap S_1)+\Ha(\gamma_j\setminus S_1)\leq\Ha(S_1)+\int_{\gamma_j\setminus S_1}b\dHa\leq\Ha(S_1)+L(\gamma_j).
		\end{equation*}
		For the case $a<\infty$ we get
		\begin{equation*}
		\textup{len}(\gamma_j)=\Ha(\gamma_j\cap S_{a/2})+\Ha(\gamma_j\setminus S_{a/2})\leq\Ha(S_{a/2})+\frac{2}{a}\int_{\gamma_j\setminus S_{a/2}}b\dHa\leq \Ha(S_{a/2})+\frac{2}{a}L(\gamma_j).
		\end{equation*}
		Thus, the lengths of the $\gamma_j$ are uniformly bounded by some $C=C(a)$, and we can reparameterize the $\gamma_j$ such that each $\gamma_j$ has constant speed at most $C$. We further have $\gamma_j\rightrightarrows\gamma$ for some subsequence (see \cref{unisub}) and thus, using the lower semi-continuity property of $L$ from \cref{Lalmostlow},
		\begin{equation*}
		L(\gamma)\leq\liminf_jL(\gamma_j)=d(x,y),
		\end{equation*}
		which shows the optimality of $\gamma$. By the same argument as above we can choose an appropriate $\psi\in\Gamma^{xy}$ which satisfies the desired properties. More precisely, we replace $\gamma$ by an injective path and apply the same estimates as for the $\gamma_j$.
	\end{proof}
	The next result proves that $d$ is lower semi-continuous, which is important to make sure that the corresponding Wasserstein distance has a minimizer (see \cref{exoptplan} later).
	\begin{prop}[$d$ lower semi-continuous]
		\label{dLsc}
		If $a<\infty$, then $d$ is continuous. If $a=\infty$, then $d$ is lower semi-continuous under \cref{Sassump}.
	\end{prop}
	\begin{proof}
		Let $((x_j,y_j))\subset\mathcal{C}\times\mathcal{C}$ be a sequence with $(x_j,y_j)\to (x,y)\in\mathcal{C}\times\mathcal{C}$. For $a<\infty$ the triangle inequality implies
		\begin{equation*}
		\limsup_jd(x_j,y_j)\leq \limsup_jd(x,y)+a|x-x_j|+a|y-y_j|=d(x,y)\leq\liminf_ja|x_j-x|+d(x_j,y_j)+a|y_j-y|=\liminf_jd(x_j,y_j).
		\end{equation*}
		To show the lower semi-continuity for the case $a=\infty$ we suppose that $\liminf_jd(x_j,y_j)<\infty$ (since otherwise there is nothing to show) and extract a subsequence with $\liminf_jd(x_j,y_j)=\lim_jd(x_j,y_j)$. Using \cref{minimiz} there exists a sequence $\psi_j\in\Gamma^{x_jy_j}$ with $L(\psi_j)=d(x_j,y_j)$ and $\textup{len}(\psi_j)$ uniformly bounded. By \cref{unisub} we can further suppose that the $\psi_j$ have constant speed and $\psi_j\rightrightarrows\psi$ for some $\psi\in\Gamma^{xy}$. Application of \cref{Lalmostlow} yields
		\begin{equation*}
		d(x,y)\leq L(\psi)\leq\liminf_jL(\psi_j)=\liminf_jd(x_j,y_j).
		\end{equation*}
		Clearly, the inequality then also holds for the entire sequence.% Now consider the case $d(x,y)=\infty$. If we had $\liminf_jd(x_j,y_j)<\infty$, then the same procedure as above would lead to a contradiction.
	\end{proof}
	As the following two examples illustrate, the conditions are sharp.
	\begin{examp}[$d$ in general not upper semi-continuous if $a=\infty$]
		\label{exnotup}
		Let $a=\infty$. Assume that $S$ is given by a line segment $[x,y]\subset\mathcal{C}$, $b\equiv 1$ on $S$ and $(x_j,y_j)\in\mathcal{C}\times\mathcal{C}$ is any sequence with $(x_j,y_j)\to (x,y)$ such that $x_j,y_j\notin [x,y]$ (see \cref{sf:fig2a}). Then we have
		\begin{equation*}
		d(x,y)=\Ha([x,y])<\infty=\limsup_jd(x_j,y_j).
		\end{equation*}
		Thus, $d$ is not upper semi-continuous in $(x,y)$. This property may even be violated if the optimal paths between $x_j$ and $y_j$ lie entirely on $S$ (see \cref{sf:fig2b}): Set 
		\begin{equation*}
		S=[x,y]\cup\bigcup_{j=2}^\infty [x,x_j],
		\end{equation*}
		where $x_j$ is a sequence with $x_j\notin [x,y]$ and $x_j\to x$. Moreover, suppose that $b|_{[x,y]}=1$ and that $b$ is constant on each $(x,x_j]$ with $b|_{(x,x_j]}\Ha((x,x_j])=j$. Then we obtain
		\begin{equation*}
		d(x,y)=\Ha([x,y])<\infty=\lim_j\Ha([x,y])+j=\lim_jd(x_j,y).
		\end{equation*}
	\end{examp}
	\begin{examp}[$d$ in general not lower semi-continuous without \cref{Sassump}]
		\label{exnotlow}
		Assume that $a=\infty$, $n=2$. Let $b\equiv 1$ on
		\begin{equation*}
		S=\bigcup_j\{1/j \}\times [0,1]
		\end{equation*}
		(see \cref{sf:fig2c}) so that \cref{Sassump} is violated. We have $(x_j,y_j)=((1/j,0),(1/j,1))\to ((0,0),(0,1))=(x,y)$, but
		\begin{equation*}
		d(x_j,y_j)\equiv 1 <\infty = d(x,y),
		\end{equation*}
		because every path between $x$ and $y$ intersects a set of positive $\Ha$-measure in $\mathcal{C}\setminus S$.
	\end{examp}
	\begin{figure}
		\centering
		\begin{subfigure}[b]{0.3\textwidth}
			\centering
			\begin{tikzpicture}
			\draw (0,0) -- (3,1);
			\draw[dashed] plot [smooth] coordinates {(-0.3,-0.2)(0,0.2)(2,0.75)(3.3,1.3)};
			\node[circle,fill=black,inner sep=0.5pt,minimum size=0.1cm,label={-90: $x$}] at (0,0) {};
			\node[circle,fill=black,inner sep=0.5pt,minimum size=0.1cm,label={-90: $y$}] at (3,1) {};
			\node[circle,fill=black,inner sep=0.5pt,minimum size=0.1cm,label={180: $x_j$}] at (-0.3,-0.2) {};
			\node[circle,fill=black,inner sep=0.5pt,minimum size=0.1cm,label={90: $y_j$}] at (3.3,1.3) {};
			\end{tikzpicture}
			\caption{$(S,b)=([x,y],1)$.}
			\label{sf:fig2a}
		\end{subfigure}
		\begin{subfigure}[b]{0.3\textwidth}
			\centering
			\begin{tikzpicture}
			\draw (0,0) -- (0,4);
			\draw (0,0) -- (45:2);
			\draw (0,0) -- (22.5:1);
			
			\node[circle,fill=black,inner sep=0.5pt,minimum size=0.1cm,label={-90: $x$}] at (0,0) {};
			\node[circle,fill=black,inner sep=0.5pt,minimum size=0.1cm,label={90: $y$}] at (0,4) {};
			\node[circle,fill=black,inner sep=0.5pt,minimum size=0.1cm,label={0: $x_2$}] at (45:2) {};
			\node[circle,fill=black,inner sep=0.5pt,minimum size=0.1cm,label={0: $x_3$}] at (22.5:1) {};
			
			%	\node[label={[label distance=0cm]180: $S_1$}] at (0,2) {};
			%	\node[label={[label distance=0cm]90: $S_2$}] at (45:1) {};
			%\node[label={[label distance=0cm]45: $S_3$}] at (22.5:0.5) {};
			
			\node[label={[label distance=-0.45cm]45: $\iddots$}] at (0:0.5) {};
			\end{tikzpicture}	
			\caption{$b|_{(x,x_j]}\Ha((x,x_j])=j$.}
			\label{sf:fig2b}
		\end{subfigure}
		\begin{subfigure}[b]{0.3\textwidth}
			\centering
			\begin{tikzpicture}
			%\fill[draw=white,color=black] (0,0) rectangle (2,2);
			
			\draw (4,0) -- (4,4);
			\draw (2,0) -- (2,4);
			\draw ({4/3},0) -- ({4/3},4);
			\draw ({4/4},0) -- ({4/4},4);
			\draw ({4/5},0) -- ({4/5},4);
			\draw ({4/6},0) -- ({4/6},4);
			\draw ({4/7},0) -- ({4/7},4);
			\draw ({4/8},0) -- ({4/8},4);
			\draw ({4/9},0) -- ({4/9},4);
			\draw ({4/10},0) -- ({4/10},4);
			\draw ({4/11},0) -- ({4/11},4);
			\draw ({4/12},0) -- ({4/12},4);
			\draw ({4/13},0) -- ({4/13},4);
			\draw ({4/14},0) -- ({4/14},4);
			\draw ({4/15},0) -- ({4/15},4);
			\draw[black,dashed] plot [smooth] coordinates {(0,0)(0.4,1)(0.5,1.4)(0.6,2)(0.5,2.2)(0.4,2.8)(0.3,3.5)(0,4)};
			
			\node[circle,fill=black,inner sep=0.5pt,minimum size=0.1cm,label={-90: $x$}] at (0,0) {};
			\node[circle,fill=black,inner sep=0.5pt,minimum size=0.1cm,label={90: $y$}] at (0,4) {};
			\node[circle,fill=black,inner sep=0.5pt,minimum size=0.1cm,label={-90: $x_j$}] at (1,0) {};
			\node[circle,fill=black,inner sep=0.5pt,minimum size=0.1cm,label={90: $y_j$}] at (1,4) {};
			\end{tikzpicture}
			\caption{$d(x,y)=\infty>1\equiv d(x_j,y_j)$.}
			\label{sf:fig2c}
		\end{subfigure}
		\caption{Sketches for \cref{exnotup,exnotlow}.}
	\end{figure}
	Our final result in this section is that limit paths of sequences of optimal paths are again optimal. It will imply closedness of a certain subset $E\subset\mathcal{C}\times\mathcal{C}\times\{ \textup{paths} \}$ needed to prove the existence of a measurable path selection later in \cref{bpcor}. 
	\begin{prop}[Optimal path limit]
		\label{optpathlimit}
		Let \cref{Sassump} be satisfied and $((x_j,y_j))\subset\mathcal{C}\times\mathcal{C}$ be a sequence such that $d(x_j,y_j)$ is uniformly bounded. Further, let $\gamma_j\in\Gamma_C^{x_jy_j}$ such that $L(\gamma_j)=d(x_j,y_j)$ (cf.\ \cref{minimiz}) and assume that each $\gamma_j$ has constant speed. Suppose that $\gamma_j\rightrightarrows\gamma\in\Gamma_C^{xy}$. Then $L(\gamma)=d(x,y)$.
	\end{prop}
%\notinclude
	\begin{proof}%[Proof of \cref{optpathlimit}]
		\underline{Case $a<\infty$:} %By \cref{minimiz} we can always choose a sequence $\gamma_j$ as in the assumption for $C$ sufficiently large (note that $d\leq a\cdot\textup{diam}(\mathcal{C})$).
		The claim follows directly from \cref{Lalmostlow} and the triangle inequality,
		\begin{equation*}
		d(x,y)\leq L(\gamma)\leq\liminf_jL(\gamma_j)=\liminf_jd(x_j,y_j)\leq \liminf_jd(x,y)+a|x-x_j|+a|y-y_j|=d(x,y).
		\end{equation*}
		\underline{Case $a=\infty$:} We can assume that the arc length of $\gamma$ is positive (otherwise the optimality is obvious) and $\liminf_jL(\gamma_j)=\lim_jL(\gamma_j)$ by restricting to a subsequence. Application of \cref{Lalmostlow} yields $d(x,y)\leq L(\gamma)\leq\lim_jL(\gamma_j)<\infty$. Hence, $d(x,y)$ is finite, and by \cref{minimiz} there exists $\tilde{\gamma}\in\Gamma^{xy}$ with $L(\tilde{\gamma})=d(x,y)$. We assume for a contradiction that $L(\tilde{\gamma})<L(\gamma)$. Let $\delta>0$ be arbitrary and pick $0<\alpha<\beta<1$ such that
		\begin{equation*}
		0<L(\gamma|_{[0,\alpha]})<\delta/6
		\qquad\text{and}\qquad
		0<L(\gamma|_{[\beta,1]})<\delta/6.
		\end{equation*}
		Furthermore, let $j$ be sufficiently large such that
		\begin{equation*}
		L(\gamma|_{[\alpha,\beta]})\leq L(\gamma_j|_{[\alpha,\beta]})+\delta/3
		\qquad\text{as well as}\qquad
		\gamma([0,\alpha])\cap\gamma_j([0,\alpha])\neq\emptyset
		\text{ and }
		\gamma([\beta,1])\cap\gamma_j([\beta,1])\neq\emptyset,
		\end{equation*}
		which is possible by the lower semi-continuity of $L$ from \cref{Lalmostlow} and by the vanishing symmetric difference between $\gamma$ and the sequence $\gamma_j$ due to \cref{help}.
		Let $p\in \gamma([0,\alpha])\cap\gamma_j([0,\alpha])$, $q\in \gamma([\beta,1])\cap\gamma_j([\beta,1])$, and $t_p\in[0,\alpha],t_q\in[\beta,1]$ such that $\gamma_j(t_p)=p$ and $\gamma_j(t_q)=q$. Then we can estimate (using that $\gamma_j|_{[t_p,t_q]}$ is an optimal path with respect to $L$ connecting $p$ and $q$):
		\begin{equation*}
		L(\gamma)=L(\gamma|_{[0,\alpha]})+L(\gamma|_{[\alpha,\beta]})+L(\gamma|_{[\beta,1]})\leq L(\gamma_j|_{[\alpha,\beta]})+\frac{2\delta}{3}\leq L(\gamma_j|_{[t_p,t_q]})+\frac{2\delta}{3}=d(p,q)+\frac{2\delta}{3}.
		\end{equation*}
		We further have
		\begin{equation*}
		d(p,q)\leq d(p,x)+d(x,y)+d(y,q)\leq L(\gamma|_{[0,\alpha]})+L(\tilde{\gamma})+L(\gamma|_{[\beta,1]})\leq L(\tilde{\gamma})+\frac{\delta}{3}
		\end{equation*}
		and therefore $L(\gamma)\leq L(\tilde{\gamma})+\delta$. This is in contradiction to $L(\tilde{\gamma})<L(\gamma)$ ($\delta>0$ was arbitrary).
	\end{proof}
	\subsection{Existence of measurable optimal path selection}
	\label{subs4}
	In this section we will prove a selection result:
	Given any $x,y\in\mathcal{C}$ we can select a path $\rho(x,y)$ with $d(x,y)=L(\rho(x,y))$ such that the resulting map $\rho$ is Borel measurable.
	To this end we apply a measurable selection theorem from \cite{BP}. Since $L$ is invariant with respect to curve reparameterization, we first define an equivalence relation on $\Gamma$ by
	\begin{equation*}
	\gamma_1\sim\gamma_2\qquad\textup{if and only if}\qquad d_\Theta (\gamma_1,\gamma_2)=0,
	\end{equation*}
	where
	\begin{equation*}
	d_\Theta (\gamma_1,\gamma_2)=\inf\left\{ |\gamma_1-\gamma_2\circ\varphi|_{\infty,I}\,|\,\varphi:I\to I \textup{ increasing and bijective} \right\}
	\end{equation*}
	and equivalence classes will be denoted by $[\cdot]_\sim$.
	Then $d_\Theta$ is a metric \cite[p.\,7]{BPSS} on 
	\begin{equation*}
	\Theta=\{ [\gamma]_\sim\,|\,\gamma\in\Gamma \}.
	\end{equation*}
	\begin{rem}[$\Theta$ not complete]
		The space $(\Theta,d_\Theta )$ is separable, which follows from the fact that every continuous function $I\to\R$ can be approximated in the uniform norm by a polynomial with rational coefficients. Unfortunately, it is not complete. A counterexample is given by the Hilbert curve, which is a space-filling and thus not Lipschitz continuous path mapping $I$ onto $[0,1]^2$. While it does not lie in $\Gamma$, it can be approximated in $d_\Theta$ by Lipschitz paths (cf.~construction in \cite[Fig.\,2]{A}).
	\end{rem}
	For $C>0$ and $x,y\in\mathcal{C}$ define
	\begin{equation*}
	\Theta_C=\{ \theta\in\Theta\, |\,\textup{len}(\theta)\leq C \}\textup{\qquad and\qquad}\Theta_C^{xy}=\{ \theta\in\Theta_C\, |\,\theta (0)=x,\theta (1)=y\}.
	\end{equation*}
	\begin{lem}[$\Theta_C$ complete]
		\label{TCcompl}
		For all $C>0$ and $x,y\in\mathcal{C}$ the metric spaces $\Theta_C$ and $\Theta_C^{xy}$ (equipped with $d_\Theta$) are complete.
	\end{lem}
	\begin{proof}
		Let $C>0$. It suffices to prove the claim for $\Theta_C$. Assume that $(\theta_j)\subset\Theta_C$ is a Cauchy sequence. Let $\gamma_j\in\theta_j$ be the sequence of representations with constant speed and therefore $(\gamma_j)\subset\Gamma_C$. There exists a subsequence $(\gamma_{j'})\subset(\gamma_j)$ with $\gamma_{j'}\rightrightarrows\gamma$ for some $\gamma\in\Gamma_C$ by \cref{unisub}. Hence, we have $d_\Theta(\theta_{j'},\theta)\to 0$ with $\gamma\in\theta\in\Theta_C$. By the assumption that $\theta_j$ is a Cauchy sequence we must have $d_\Theta(\theta_{j},\theta)\to 0$ for the whole sequence.
	\end{proof}
	We want to apply the following measurable selection statement to prove the existence of a Borel measurable path selection $\rho:\mathcal{C}\times\mathcal{C}\to\Theta$ such that $d=L\circ\rho$.
	Note that $L(\theta)$ is well-defined for any $\theta\in\Theta$, because every representative of $\theta$ traverses in the same way.
	\begin{prop}[{{\cite[Thm.\,1]{BP}}}]
		\label{bp}
		Assume that $U$ and $V$ are separable and complete metric spaces and $E\subset U\times V$ is Borel measurable. If for each $u\in U$ the section $E_u=\{ v\in V\,|\,(u,v)\in E\}$ is $\sigma$-compact, then the projection of $E$ onto $U$, denoted by $\textup{proj}_U(E)$, is Borel measurable and there exists a Borel-selection $S\subset E$ of $E$, i.e.,
		\begin{itemize}
			\item $S$ is Borel measurable,
			\item $\textup{proj}_U(E)=\textup{proj}_U(S)$,
			\item there is a Borel measurable function $\rho:\textup{proj}_U(E)\to V$ which is uniquely defined by
			\begin{equation*}
			(u,\rho (u))\in S.
			\end{equation*}
		\end{itemize}
	\end{prop}
	For the rest of this section we write $$T_\lambda=\{ (x,y)\in\mathcal{C}\times\mathcal{C} \,|\, d(x,y)\leq\lambda \}$$ for $\lambda\in[0,\infty)$. Those sets are closed, because $d$ is lower semi-continuous by \cref{dLsc}. Further, we define
	\begin{equation}
	\label{eq:Clam}
	C_\lambda=\begin{cases*}
	\Ha(S_1)+\lambda&if $a=\infty$,\\
	\Ha(S_{a/2})+\frac{2}{a}\lambda&if $a<\infty$
	\end{cases*}
	\end{equation}
	for all such $\lambda$ (cf.\ \cref{minimiz}). A direct consequence of \cref{bp} is the following statement.
	\begin{corr}[Measurable bounded optimal path map]
		\label{bpcor}
		Let \cref{Sassump} be satisfied. For each $\lambda\in[0,\infty)$ there exists a Borel measurable map $\rho_\lambda:T_\lambda\to\Theta_{C_\lambda}$ such that
		\begin{equation*}
		\rho_\lambda(x,y)\in\{ \theta\in\Theta_{C_\lambda}^{xy}\,|\, L(\theta)=d(x,y) \}
		\qquad\text{for all }(x,y)\in T_\lambda.
		\end{equation*}
	\end{corr}
	\begin{proof}
		Fix $\lambda\in[0,\infty)$ and define complete and separable metric spaces by $U=\mathcal{C}\times\mathcal{C}$ and $V=\Theta_{C_\lambda}$ (see \cref{TCcompl}). Furthermore, let
		\begin{equation*}
		E=\left\{ ((x,y),\theta)\in U\times V\,\big|\, (x,y)\in T_\lambda,\theta\in\Theta_{C_\lambda}^{xy},L(\theta)=d(x,y) \right\}.
		\end{equation*}	
		We show that $E$ is Borel measurable. Actually, we prove that $E$ is closed. Let $(((x_j,y_j),\theta_j))\subset E$ be a sequence such that $(x_j,y_j)\to (x,y)$ in $\mathcal{C}\times\mathcal{C}$ and $\theta_j\to\theta$ with respect to $d_\Theta$. We have $d(x,y)\leq\lambda$ by the lower semi-continuity of $d$ (\cref{dLsc}). By $\textup{len}(\theta_j)\leq C_\lambda$ we can represent the $\theta_j$ by paths with constant speed $(\gamma_j)\subset\Gamma_{C_\lambda}$. Using \cref{unisub} we have $\gamma_{j'}\rightrightarrows\gamma\in\Gamma_{C_\lambda}^{xy}$ for some subsequence $(\gamma_{j'})\subset (\gamma_j)$. This implies $d_\Theta (\gamma_{j'},\gamma)\to 0$ which yields $\gamma\in\theta$. Hence, $\theta$ is optimal by \cref{optpathlimit} on the optimal path limit. This shows that $E$ is closed. Now consider the section $E_{(x,y)}=\{ \theta\,|\, ((x,y),\theta)\in E \}$ for $(x,y)\in\mathcal{C}\times\mathcal{C}$. We claim that $E_{(x,y)}$ is compact. If $d(x,y)>\lambda$, then we have $E_{(x,y)}=\emptyset$. Therefore, we can assume that $d(x,y)\leq\lambda$ and pick any sequence $(\theta_j)\subset E_{(x,y)}$. Again, the lengths of the $\theta_j$ are uniformly bounded by $C_\lambda$ and we can represent by paths with constant speed and extract a subsequence which converges uniformly to some $\theta\in\Theta_{C_\lambda}^{xy}$. By \cref{optpathlimit} $\theta$ is optimal and thus $\theta\in E_{(x,y)}$. Hence, $E_{(x,y)}$ is compact. By the statement about minimizers for $d$ (\cref{minimiz}) we obtain $\textup{proj}_U(E)=T_\lambda$. Finally, we apply the previous measurable selection theorem (\cref{bp}) to get the desired result. 
	\end{proof}
	Using an appropriate partition of $\mathcal{C}\times\mathcal{C}$ we can now show our main result for this section.
	\begin{prop}[Measurable optimal path map]
		\label{final}
		Let \cref{Sassump} be satisfied. There exists a Borel measurable map
		\begin{equation*}
		\rho:\mathcal{C}\times\mathcal{C}\to\Theta
		\end{equation*}
		such that for all $(x,y)\in\mathcal{C}\times\mathcal{C}$ we have
		\begin{equation*}
		L(\rho(x,y))=d(x,y)\qquad\textup{and}\qquad\rho(x,y)\in\Theta_{C_{d(x,y)+1}}^{xy}, \textup{ i.e., } \textup{len}(\rho(x,y))\leq C_{d(x,y)+1}.
		\end{equation*}
	\end{prop}
	\begin{proof}
		For each $\lambda\in[0,\infty)$ there is a Borel measurable function $\rho_\lambda:T_\lambda\to\Theta_{C_\lambda}$ by \cref{bpcor}. Clearly, $\rho_\lambda$ is Borel measurable as a function $\rho_\lambda:T_\lambda\to\Theta$, because $\Theta_{C_\lambda}$ is complete (and therefore closed) by \cref{TCcompl}. We define a mapping $\hat{\rho}:\{ d<\infty \}\to\Theta$. Consider the following partition of $\{ d<\infty \}$:
		\begin{equation*}
		P_1= T_1\textup{\qquad and\qquad}P_{i+1}= T_{i+1}\setminus T_i\textup{ for }i\in\mathbb{N}.
		\end{equation*}
		Then $P_i$ is Borel measurable for all $i$ by the Borel measurability of all $T_i$. For $(x,y)\in P_i$ for some $i$ let
		\begin{equation*}
		\hat{\rho} (x,y)=\rho_i(x,y).
		\end{equation*}
		We have $\textup{len}(\rho_i(x,y))\leq C_i\leq C_{d(x,y)+1}$ by definition of $P_i$ and \cref{bpcor}. Furthermore, for $B\in\mathcal{B}(\Theta)$ we get
		\begin{equation*}
		\hat{\rho}^{-1}(B)=\bigcup_i\rho_i^{-1}(B)\cap P_i,
		\end{equation*}
		which is Borel measurable.
		% Let $(x,y)\in\mathcal{C}\times\mathcal{C}$ with $d(x,y)<\infty$. By \cref{minimiz} there exists a $\gamma\in\Gamma_{xy}$ such that $L(\gamma )=d(x,y)$. Additionally, the length of $\gamma$ is bounded by $\Ha(S_1)+d(x,y)$ (cf.\ proof of \cref{unisub}). In particular, we get $\gamma\in\Theta_{\Ha(S_1)+d(x,y)}$. Set
		%\begin{equation*}
		%\rho (x,y)=\rho_{\Ha(S_1)+d(x,y)}(x,y).
		%\end{equation*}
		%This is well-defined since the definition does not depend on $\gamma$. We now show that $\rho$ is Borel measurable. Let $B\in\mathcal{B}(\Theta)$. Then we have
		%\begin{align*}
		%\rho^{-1}(B)&=\rho^{-1}(\{ \gamma\in B\,|\,\gamma\textup{ has finite arc length} \})\\
		%&=\bigcup_{j\in\mathbb{N}}\rho^{-1}\left( B\cap\Theta_{\Ha(S_1)+j} \right)\\
		%&=\bigcup_{j\in\mathbb{N}}\rho_{\Ha(S_1)+j}^{-1}\left( B\cap\Theta_{\Ha(S_1)+j} \right).
		%\end{align*}
		%The sets $A_j=\rho_{\Ha(S_1)+j}^{-1}( B\cap\Theta_{\Ha(S_1)+j} )$ are Borel measurable subsets of $\mathcal{C}_{\Ha(S_1)+j}^2$. Hence, $A_j=B_j\cap\mathcal{C}_{\Ha(S_1)+j}^2$ with $B_j\in\mathcal{B}(\mathcal{C}\times\mathcal{C} )$. Moreover, $\mathcal{C}_{\Ha(S_1)+j}^2$ is Borel measurable by \cref{bpcor} which implies $A_j\in\mathcal{B}(\mathcal{C}\times\mathcal{C} )$ and, in consequence, $\rho^{-1}(B)=\bigcup_jA_j\in\mathcal{B}(\mathcal{C}\times\mathcal{C} )$. 
		Finally, $\hat{\rho}$ can be continued to a Borel measurable function $\rho:\mathcal{C}\times\mathcal{C}\to\Theta$. Let $\theta_{xy}\in\Theta$ be the straight line connection from $x$ to $y$ and set
		\begin{equation*}
		\rho(x,y)=\begin{cases*}
		\hat{\rho} (x,y)&if $d(x,y)<\infty$,\\
		\theta_{xy}&else.
		\end{cases*}
		\end{equation*}
		Since $\rho$ equals the Borel measurable $\hat\rho$ on the Borel set $\{d<\infty\}=\bigcup_iP_i$ and $\rho$ is continuous and thus Borel measurable on the complement $\{d=\infty\}$,
		the map $\rho$ is Borel measurable on all of $\mathcal{C}\times\mathcal{C}$.
		\notinclude{Then for each $B\in\mathcal{B}(\Theta )$ we have
			\begin{equation*}
			\rho^{-1}(B)=\hat{\rho}^{-1}(B\setminus\{\theta_0\})\cup\rho^{-1}(B\cap\{\theta_0\}).
			\end{equation*}
			$\{\theta_0\}$ is closed in $\Theta$. Thus, $\hat{\rho}^{-1}(B\setminus\{\theta_0\})=\tilde{B}\cap\{ d<\infty \}$ for some $\tilde{B}\in\mathcal{B}(\mathcal{C}\times\mathcal{C})$ by the measurability of $\hat{\rho}$. This is a Borel measurable subset of $\mathcal{C}\times\mathcal{C}$, because $d$ is lower semi-continuous and thus Borel measurable (\cref{dLsc}). The second set is empty or equal to $\{d=\infty\}\cup\hat{\rho}^{-1}(\theta_0)$ and hence also Borel measurable.
		}%\notinclude
	\end{proof}
	\subsection{Wasserstein distance with generalized urban metric as Beckmann problem}
	\label{subs5}
	In this section we prove \cref{BWfinal}. 
	%	i.e., we show that under \cref{Sassump}
	%	\begin{equation*}
	%	W_d(\mu_+,\mu_-)=B^{\mu_+,\mu_-}=\inf_{\theta,\F}\int_{\mathcal{C}}  b\,\mathrm{d}|\theta\Ha\mres S+\F|,
	%	\end{equation*}
	%	where $\theta\in L^1(\Ha\mres S;\R^n)$ and $\F\in\mathcal{M}^n(\mathcal{C})$ with $\F\mres S=0$ satisfy $\textup{div}(\theta\Ha\mres S+\F)=\mu_+-\mu_-$. Further, we will explicitly construct minimizers for one from the other problem.~
	We will use the idea that a mass flux can be seen as a measure on paths, which is more accurately defined as follows (see \cite[Def.\,2.5]{BPSS}).
	\begin{defin}[Mass flux measure]
		\label{mfmeas}
		Any measure $\eta:(\Theta,\mathcal{B}(\Theta))\to[0,\infty)$ is called \textbf{mass flux measure} (recall the definition of $\Theta$ at the beginning of \cref{subs4}). Further, $\eta$ \textbf{moves $\mu_+$ onto $\mu_-$} if 
		\begin{equation*}
		\mu_+(B)=\eta(\{ \theta\in\Theta\,|\,\theta(0)\in B \})\qquad\textup{and}\qquad \mu_-(B)=\eta(\{ \theta\in\Theta\,|\,\theta(1)\in B \})
		\end{equation*}
		for all Borel sets $B\in\mathcal{B}(\mathcal{C})$, thus $(\mu_+,\mu_-)$ is the pushforward of $\eta$ under the map $\theta\mapsto(\theta(0),\theta(1))$.
	\end{defin}
	To translate back and forth between mass flux measures and mass fluxes we will need the following type of measures.
	\begin{lem}[Line integral measure]
		\label{lintm}
		Let $\gamma\in\Gamma$ be injective and define the Radon measure $\F_\gamma$ by
		\begin{equation*}
		\langle\varphi,\F_\gamma\rangle=\int_I\varphi(\gamma)\cdot\dot{\gamma}\,\mathrm{d}\mathcal{L}
		\qquad\text{for all }\varphi\in C(\mathcal{C};\R^n).
		\end{equation*}
		Then we have $|\F_\gamma|=\Ha\mres\gamma$ or equivalently
		\begin{equation*}
		\langle\varphi,|\F_\gamma|\rangle=\int_I\varphi(\gamma)|\dot{\gamma}|\,\mathrm{d}\mathcal{L}
		\qquad\text{for all }\varphi\in C(\mathcal{C}).
		\end{equation*}
	\end{lem}
	\begin{proof}
		Without loss of generality we can replace $I$ by $[0,\textup{len}(\gamma)]$ and assume that $\gamma$ is parameterized by arc length. Using the assumption that $\gamma$ is injective and the last formula in \cref{cov} we get
		\begin{equation*}
		(\Ha\mres\gamma)(B)=\Ha(\gamma(\gamma^{-1}(B)))=\int_{\gamma^{-1}(B)}|\dot{\gamma}|\,\mathrm{d}\mathcal{L}=\mathcal{L}(\gamma^{-1}(B))=(\gamma_\#(\Le\mres[0,\textup{len}(\gamma)]))(B)
		\end{equation*}
		for all $B\in\mathcal{B}(\mathcal{C})$. Hence, we have $\Ha\mres\gamma=\gamma_\#(\Le\mres[0,\textup{len}(\gamma)])$. Now define $v_\gamma(x)=\dot{\gamma}(\gamma^{-1}(x))\in\mathcal{S}^{n-1}$ for $\Ha$-almost all $x\in\gamma$. For $\varphi\in C(\mathcal{C};\R^n)$ we obtain
		\begin{equation*}
		\langle\varphi,\F_\gamma\rangle=\int_{[0,\textup{len}(\gamma)]}\varphi(\gamma)\cdot\dot{\gamma}\dLe=\int_{[0,\textup{len}(\gamma)]}\varphi(\gamma)\cdot v_\gamma(\gamma)\dLe=\int_{\mathcal{C}}\varphi\cdot v_\gamma\,\mathrm{d}(\gamma_\#(\Le\mres[0,\textup{len}(\gamma)]))=\int_{\mathcal{C}}\varphi\cdot v_\gamma\dHa\mres\gamma.
		\end{equation*}
		This yields $\F_\gamma=v_\gamma\Ha\mres\gamma$ and therefore $|F_\gamma|=|v_\gamma|\Ha\mres\gamma=\Ha\mres\gamma$. In particular, by \cref{cov} we get
		\begin{equation*}
		\langle\varphi,|\F_\gamma|\rangle=\int_{\mathcal{C}}\varphi\dHa\mres\gamma=\int_{[0,\textup{len}(\gamma)]}\varphi(\gamma)|\dot{\gamma}|\,\mathrm{d}\mathcal{L}
		\end{equation*}
		for $\varphi\in C(\mathcal{C})$.% and $|\F_\gamma|(B)=\int_B\dHa\mres\gamma=\Ha(\gamma\cap B)$ for $B\in\mathcal{B}(\mathcal{C})$.
	\end{proof}
	\dualityApproach{\todo[inline]{If Section 3 stays, add ``under \cref{Sassump}'' to next sentence and title of proof.}}
	We can now prove \cref{BWfinal}. First, we show that the minimum Beckman cost is no smaller than the Wasserstein distance: From an admissible mass flux for the Beckmann problem we substract the part with vanishing divergence and represent the remainder by a mass flux measure \cite[Thm.\,C]{S}. We then transform that mass flux measure into an admissible transport plan with no larger energy (as in \cite[Def.\,3.4.9]{Bran}). For the reverse inequality we pick an admissible optimal transport plan and push forward under the measurable path selection from \cref{final} to get a mass flux measure, which in turn induces a mass flux.
	\begin{proof}[Proof of \cref{BWfinal}]
		\underline{$W_d(\mu_+,\mu_-)\leq\inf_{\xi,\F^\perp}\int_{\mathcal{C}} b\,\mathrm{d}|\xi\Ha\mres S+\F^\perp|$:} We can assume that the Beckmann problem is finite. Thus, there exist $\xi\in L^1(\Ha\mres S;\R^n)$ and $\F^\perp\in\mathcal{M}^n(\mathcal{C})$ such that $\F^\perp\mres S=0$, $\textup{div}(\xi\Ha\mres S+\F^\perp)=\mu_+-\mu_-$ and
		\begin{equation*}
		\int_{\mathcal{C}} b\,\mathrm{d}|\xi\Ha\mres S+\F^\perp|<\infty.
		\end{equation*}
		We use the same idea as in the proof of \cite[Prop.\,4.1]{BW} to replace the mass flux $\xi\Ha\mres S+\F^\perp$ by a mass flux measure on $\Theta$. By \cite[Thm.\,C]{S} we have $\xi\Ha\mres S+\F^\perp=\hat{\F}+\G$ with $\textup{div}(\hat{\F})=0$, $\textup{div}(\G)=\mu_+-\mu_-$ and $|\xi\Ha\mres S+\F^\perp|=|\hat{\F}|+|\G|$. Hence, we get $|\G|\leq |\xi\Ha\mres S+\F^\perp|$ and thus
		\begin{equation*}
		\int_{\mathcal{C}} b\,\mathrm{d}|\G|\leq\int_{\mathcal{C}} b\,\mathrm{d}|\xi\Ha\mres S+\F^\perp|.
		\end{equation*}
		%Therefore, we can assume that $\theta\Ha\mres S+\F^\perp=\G$.
		Again by \cite[Thm.\,C]{S} we get that $\G$ can be associated with a mass flux measure $\eta$ on $\Theta$ moving $\mu_+$ onto $\mu_-$, which is supported on loop-free paths \cite[(1.14)]{S}, i.e.,
		\begin{equation*}
		\int_{\mathcal{C}}\varphi\cdot\,\mathrm{d}\G
		=\int_{\Theta}\langle\varphi,\F_\gamma\rangle\,\mathrm{d}\eta (\gamma)
		=\int_{\Theta}\int_{I}\varphi (\gamma)\cdot\dot{\gamma}\,\mathrm{d}\mathcal{L}\,\mathrm{d}\eta (\gamma)
		\end{equation*}
		for all $\varphi\in C(\mathcal{C};\R^n)$, using the Radon measure $\F_\gamma$ from \cref{lintm} (note that for simplicity we identify paths $\gamma$ with their equivalence classes $[\gamma]_\sim\in\Theta$). By \cite[(1.10)]{S} we have
		\begin{equation}
		\label{eq:gfg}
		|\G|(B)=\int_{\Theta}|\F_\gamma|(B)\,\mathrm{d}\eta (\gamma)
		\end{equation}
		for all $B\in\mathcal{B}(\mathcal{C})$. As in \cite[Def.\,3.4.9]{Bran} we define a transport plan $\pi\in\Pi(\mu_+,\mu_-)$ (recall \cref{WdisTp}) by
		\begin{equation*}
		\int_{\mathcal{C}\times\mathcal{C}}\varphi\,\mathrm{d}\pi=\int_{\Theta}\varphi (\gamma (0),\gamma (1))\,\mathrm{d}\eta (\gamma)
		\end{equation*}
		for $\varphi\in C(\mathcal{C}\times\mathcal{C} )$. We show $\int d\,\mathrm{d}\pi\leq\int b\,\mathrm{d}|\G|$ to get the desired result. First, we prove
		\begin{equation*}
		\int_{\mathcal{C}}b\,\mathrm{d}|\G|=\int_{\Theta}\int_Ib(\gamma)|\dot{\gamma}|\,\mathrm{d}\mathcal{L}\,\mathrm{d}\eta (\gamma).
		\end{equation*}
		Let $\varphi\in C(\mathcal{C})$ with $\varphi\geq 0$. By construction of the Lebesgue integral there exist simple functions $\varphi_i=\sum_{j=1}^{N(i)}c_j^i1_{B_j^i}:\mathcal{C}\to [0,\infty)$ with $\varphi_i\nearrow\varphi$ pointwise, where $c_j^i\geq 0$ and $B_j^i\in\mathcal{B}(\mathcal{C})$. Using the monotone convergence theorem and the fact that $|F_\gamma|(B_j^i)=\Ha(\gamma\cap B_j^i)$ (\cref{lintm}) we get
		\begin{multline*}
		\int_{\mathcal{C}}\varphi\,\mathrm{d}|\G|
		=\lim_i\sum_{j=1}^{N(i)}c_j^i|\G|(B_j^i)
		=\lim_i\sum_{j=1}^{N(i)}c_j^i\int_{\Theta}|\F_\gamma|(B_j^i)\,\mathrm{d}\eta(\gamma)
		=\lim_i\sum_{j=1}^{N(i)}c_j^i\int_{\Theta}\int_I1_{B_j^i}(\gamma)|\dot{\gamma}|\,\mathrm{d}\mathcal{L}\,\mathrm{d}\eta(\gamma)\\
		=\lim_i\int_{\Theta}\int_I\varphi_i(\gamma)|\dot{\gamma}|\,\mathrm{d}\mathcal{L}\,\mathrm{d}\eta(\gamma)
		=\int_{\Theta}\int_I\varphi(\gamma)|\dot{\gamma}|\,\mathrm{d}\mathcal{L}\,\mathrm{d}\eta(\gamma).
		\end{multline*} 
		The same argumentation shows the formula for general $\varphi\in C(\mathcal{C})$ via decomposition into positive and negative part. To show the formula $\int b\,\mathrm{d}|\G|=\int\int b(\gamma)|\dot{\gamma}|\dLe\,\mathrm{d}\eta(\gamma)$ we now distinguish two cases. For the case $a<\infty$ let $S^N$ be an approximating sequence for $S$ (recall the definition from \cref{notdefs}). The difference $Z=S\setminus\bigcup_NS^N$ satisfies $\Ha(Z)=0$. We can thus assume $S=\bigcup_NS^N$, because the divergence constraint stays satisfied if we neglect the null set $Z$. We now introduce lower semi-continuous approximations (recall that $S^N$ is closed) of $b$ by
		\begin{equation*}
		b_N=\begin{cases*}
		b&on $S^N$,\\
		a&else.
		\end{cases*}
		\end{equation*}
		Each $b_N$ can be approximated by (Lipschitz) continuous functions $f_i^N:\mathcal{C}\to [0,\infty)$ with $f_i^N\leq f_{i+1}^N$ and $f_i^N\to b_N$ pointwise for $i\to\infty$ \cite[Box 1.5]{San}, for instance using the Moreau envelope. Therefore, the monotone convergence theorem yields
		\begin{equation*}
		\int_{\mathcal{C}}b_N\,\mathrm{d}|\G|=\int_{\Theta}\int_Ib_N(\gamma)|\dot{\gamma}|\,\mathrm{d}\mathcal{L}\,\mathrm{d}\eta(\gamma)
		\end{equation*}
		and thus the desired formula for $b$ using again the monotone convergence theorem as $b_N\searrow b$ pointwise. For the case $a=\infty$ we cannot apply monotone convergence, because the above integrals may not be finite and $b_N$ is decreasing in $N$. Due to $\int b\,\mathrm{d}|\G|<\infty$ we must have $\G=\xi\Ha\mres S$. The function $b$ is lower semi-continuous on $S$, and thus there exist Lipschitz functions $b_j:S\to[0,\infty)$ with $b_j\nearrow b$ pointwise on $S$ (again, see e.g.\ \cite[Box 1.5]{San}). We can now continuously extend each $b_j$ to $\mathcal{C}$ \cite[Ch.\,2, Thm.\,1.2]{LS}. Further, we have $|\G|(\mathcal{C}\setminus S)=0$ and thus $|\F_\gamma|(\mathcal{C}\setminus S)=0$ for $\eta$-almost all $\gamma\in\Theta$ by equation \eqref{eq:gfg} (again we identify paths with their equivalence classes in $\Theta$). Hence we obtain $\Ha(\gamma\setminus S)=0$ for $\eta$-almost all $\gamma\in\Theta$ using \cref{lintm}. By monotone convergence and the formula $\int\varphi\,\mathrm{d}|\G|=\int\int\varphi(\gamma)|\dot{\gamma}|\dLe\,\mathrm{d}\eta(\gamma)$ for continuous $\varphi$ we get
		\begin{multline*}
		\int_{\mathcal{C}}b\,\mathrm{d}|\G|
		=\int_Sb\,\mathrm{d}|\G|=\lim_j\int_Sb_j\,\mathrm{d}|\G|
		=\lim_j\int_{\mathcal{C}}b_j\,\mathrm{d}|\G|
		=\lim_j\int_{\Theta}\int_Ib_j(\gamma)|\dot{\gamma}|\,\mathrm{d}\Le\,\mathrm{d}\eta(\gamma)\\
		=\lim_j\int_{\Theta}\int_I1_S(\gamma)b_j(\gamma)|\dot{\gamma}|\dLe\,\mathrm{d}\eta(\gamma)
		=\int_{\Theta}\int_I1_S(\gamma)b(\gamma)|\dot{\gamma}|\dLe\,\mathrm{d}\eta(\gamma)
		=\int_{\Theta}\int_Ib(\gamma)|\dot{\gamma}|\dLe\,\mathrm{d}\eta(\gamma).
		\end{multline*}
		We can now finish the proof. By \cref{dLsc} $d$ is lower semi-continuous. Thus, again using Lipschitz approximations and the monotone convergence theorem, we obtain
		\begin{equation*}
		\int_{\mathcal{C}\times\mathcal{C}}d(x,y)\,\mathrm{d}\pi(x,y)=\int_{\Theta}d(\gamma(0),\gamma(1))\,\mathrm{d}\eta(\gamma)\leq\int_{\Theta}\int_Ib(\gamma)|\dot{\gamma}|\,\mathrm{d}\mathcal{L}\,\mathrm{d}\eta(\gamma)=\int_{\mathcal{C}}b\,\mathrm{d}|\G|.
		\end{equation*}
		\underline{$W_d(\mu_+,\mu_-)\geq\inf_{\xi,\F^\perp}\int_{\mathcal{C}} b\,\mathrm{d}|\xi\Ha\mres S+\F^\perp|$:}
		Assume that there exists a transport plan $\pi\in\Pi(\mu_+,\mu_-)$ such that
		\begin{equation*}
		\int_{\mathcal{C}\times\mathcal{C}}d(x,y)\,\mathrm{d}\pi (x,y)<\infty.
		\end{equation*}
		We consider the push-forward $\eta$ of $\pi$ under the Borel measurable function $\rho:\mathcal{C}\times\mathcal{C}\to\Theta$ from \cref{final} which maps onto optimal paths with respect to $d$. We get
		\begin{equation*}
		\int_{\mathcal{C}\times\mathcal{C}}d(x,y)\,\mathrm{d}\pi (x,y)=\int_{\mathcal{C}\times\mathcal{C}}L(\rho (x,y))\,\mathrm{d}\pi (x,y)=\int_{\Theta}L(\gamma )\,\mathrm{d}\eta (\gamma)=\int_{\Theta}\int_{I}b(\gamma )|\dot{\gamma}|\,\mathrm{d}\mathcal{L}\,\mathrm{d}\eta (\gamma ).
		\end{equation*}
		This motivates the definition of the functional $\F$ via
		\begin{equation*}
		\langle\varphi,\F\rangle=\int_{\Theta}\int_{I}\varphi(\gamma)\cdot\dot{\gamma}\,\mathrm{d}\mathcal{L}\,\mathrm{d}\eta (\gamma)
		\qquad\text{for all }\varphi\in C(\mathcal{C};\R^n). 
		\end{equation*}
		Clearly, $\F$ is linear. Its continuity follows from \cref{final} and the definition of $C_\lambda$ (equation \eqref{eq:Clam} in \cref{subs4}),
		\begin{multline*}
		|\langle\varphi,\F\rangle|
		\leq |\varphi|_{\infty,\mathcal{C}}\int_{\Theta}\int_{I}|\dot{\gamma}|\,\mathrm{d}\mathcal{L}\,\mathrm{d}\eta (\gamma)
		=|\varphi|_{\infty,\mathcal{C}}\int_{\mathcal{C}\times\mathcal{C}}\textup{len}(\rho(x,y))\,\mathrm{d}\pi(x,y)
		\leq|\varphi|_{\infty,\mathcal{C}}\int_{\mathcal{C}\times\mathcal{C}}C_{d(x,y)+1}\,\mathrm{d}\pi(x,y)\\
		=|\varphi|_{\infty,\mathcal{C}}\cdot\begin{cases*}
		\int_{\mathcal{C}\times\mathcal{C}}d(x,y)\,\mathrm{d}\pi(x,y)+\Ha(S_1)+1&if $a=\infty$,\\
		\frac{2}{a}\int_{\mathcal{C}\times\mathcal{C}}d(x,y)\,\mathrm{d}\pi(x,y)+\Ha(S_{a/2})+\frac{2}{a}&if $a<\infty$.
		\end{cases*}
		\end{multline*}
		Note that the constant on the right-hand side is finite due to \cref{Sassump} and the choice of $\pi$. In particular, $\F$ is a Radon measure and moreover a mass flux between $\mu_+$ and $\mu_-$, since for all $\varphi\in C^1(\mathcal{C})$ we have
		\begin{multline*}
		\langle\varphi ,\textup{div}(\F)\rangle
		=-\int_{\mathcal{C}}\nabla\varphi\cdot\,\mathrm{d}\F
		=-\int_{\Theta}\int_{I}(\nabla\varphi)(\gamma)\cdot\dot{\gamma}\,\mathrm{d}\mathcal{L}\,\mathrm{d}\eta (\gamma)
		=-\int_{\Theta}\int_{I}\frac{\,\mathrm{d}}{\,\mathrm{d}t}(\varphi\circ\gamma)\,\mathrm{d}\mathcal{L}\,\mathrm{d}\eta (\gamma)\\
		=\int_{\Theta}[\varphi (\gamma(0))-\varphi (\gamma(1))]\,\mathrm{d}(\rho_\#\pi)(\gamma)
		=\int_{\mathcal{C}}\varphi\,\mathrm{d}[(p_1)_\#\pi-(p_2)_\#\pi]
		=\langle\varphi,\mu_+-\mu_-\rangle
		\end{multline*}
		using the fundamental theorem of calculus in the fourth equality. We now show
		\begin{equation*}
		\int_{\Theta}\int_{I}\varphi (\gamma)|\dot{\gamma}|\,\mathrm{d}\mathcal{L}\,\mathrm{d}\eta (\gamma)\geq \int_{\mathcal{C}}\varphi\,\mathrm{d}|\F|
		\end{equation*}
		for all $\varphi\in C(\mathcal{C})$ with $\varphi\geq 0$. Let $f$ denote the Radon\textendash Nikodym derivative of $\F\in\mathcal{M}^n(\mathcal{C})$ with repect to $|\F|$, i.e.,
		\begin{equation*}
		\int_{\mathcal{C}}\varphi\cdot\,\mathrm{d}\F=\int_{\mathcal{C}}\varphi\cdot f\,\mathrm{d}|\F|
		\end{equation*}
		for all $\varphi\in C(\mathcal{C};\R^n)$. By \cite[Box 4.2]{San} we have $|f|=1$ $|\F|$-almost everywhere on $\mathcal{C}$. Additionally, using $f\in L^1(|\F|;\R^n)$ there exists\footnote{It is straightforwad to prove this using the regularity of $|\F|$ and Lusin's theorem (which is applicable due to \cite[Theorem \& Remark]{F81}).}  a sequence $(f_k)\subset C(\mathcal{C};\R^n)$ such that
		\begin{equation*}
		\int_{\mathcal{C}}|f-f_k|\,\mathrm{d}|\F|\to 0.
		\end{equation*}
		By restricting to a subsequence we have $f_k\to f$ pointwise $|\F|$-almost everywhere on $\mathcal{C}$. Further, we can suppose that $|f_k|\leq 1$, because the continuous functions
		\begin{equation*}
		\tilde{f}_k=\begin{cases*}
		f_k&if $|f_k|\leq 1$,\\
		\frac{1}{|f_k|}f_k&else
		\end{cases*}
		\end{equation*}
		are better approximations of $f$ ($|\F|$-almost everywhere we have $f\in\mathcal{S}^{n-1}$). Hence we get
		\begin{equation*}
		\int_{\Theta}\int_{I}\varphi (\gamma)|\dot{\gamma}|\,\mathrm{d}\mathcal{L}\,\mathrm{d}\eta (\gamma)\geq\int_{\Theta}\int_{I}\varphi (\gamma)|f_k(\gamma)||\dot{\gamma}|\,\mathrm{d}\mathcal{L}\,\mathrm{d}\eta (\gamma)
		\geq\int_{\Theta}\int_{I}(\varphi f_k)(\gamma)\cdot\dot{\gamma}\,\mathrm{d}\mathcal{L}\,\mathrm{d}\eta (\gamma)
		=\int_{\mathcal{C}}\varphi f_k\cdot\,\mathrm{d}\F
		=\int_{\mathcal{C}}\varphi f_k\cdot f\,\mathrm{d}|\F|
		\end{equation*}
		for all $\varphi\in C(\mathcal{C})$ with $\varphi\geq 0$. Moreover, $|\F|$-almost everywhere on $\mathcal{C}$ we have $\varphi f_k\cdot f\to\varphi f^2=\varphi$ as well as $|\varphi f_k\cdot f|\leq\varphi$. Thus, by Lebesgue's dominated convergence theorem we obtain
		\begin{equation*}
		\int_{\Theta}\int_{I}\varphi (\gamma)|\dot{\gamma}|\,\mathrm{d}\mathcal{L}\,\mathrm{d}\eta (\gamma)\geq\int_{\mathcal{C}}\varphi\,\mathrm{d}|\F|.
		\end{equation*}
		By \cite[Thm.\,3.1]{Sil} we get $\F=\xi\Ha\mres S+\F^\perp$ for some $\xi\in L^1(\Ha\mres S;\R^n)$ and $\F^\perp\in\mathcal{M}^n(\mathcal{C})$ with $\F^\perp\mres S=0$. The same argument as in the first half of the proof (assuming that $S$ is the set-theoretic limit of an approximating sequence $S^N$, defining the $b_N$, exploiting monotone convergence et cetera) shows
		\begin{equation*}
		\int_{\Theta}\int_Ib(\gamma)|\dot{\gamma}|\,\mathrm{d}\mathcal{L}\,\mathrm{d}\eta(\gamma)\geq\int_{\mathcal{C}}b\,\mathrm{d}|\F|.
		\end{equation*}
		Hence we get
		\begin{equation*}
		\int_{\mathcal{C}\times\mathcal{C}}d(x,y)\,\mathrm{d}\pi (x,y)=\int_{\Theta}\int_{I}b(\gamma)|\dot{\gamma}|\,\mathrm{d}\mathcal{L}\,\mathrm{d}\eta (\gamma)\geq\int_{\mathcal{C}}b\,\mathrm{d}|\F|.\qedhere
		\end{equation*}
	\end{proof}
	We close this section with a brief discussion of existence of minimizers for the Wasserstein and the Beckmann problem.
	First note that without \cref{Sassump} an optimal mass flux for the Beckmann problem may not exist as the next example shows.
	\begin{examp}[Non-existence of optimal mass flux]
		\label{exnotex}
		Let $n=2$ and set $x=(0,0),y=(1,0)$. Assume that $b\equiv 1$ on
		\begin{equation*}
		S=\bigcup_j([x,z_j]\cup [z_j,y]),
		\end{equation*}
		where $z_j=(1/2,1/j)$ (see \cref{fig3}). Clearly, \cref{Sassump} is not satisfied. Note that $[x,y]\cap S=\{ x,y \}$. If $\mu_+=\delta_x$ and $\mu_-=\delta_y$, then an optimal transport plan for the Wasserstein distance is clearly given by $\pi=\delta_{(x,y)}$. We have $d(x,y)=\inf_{\Gamma^{xy}}L=1$, where a sequence of minimizing paths is given by the injective paths $\gamma_j$ that parameterize $[x,z_j]\cup [z_j,y]$. The $\gamma_j$ induce a minimizing sequence $\F_j=\xi_j\Ha\mres([x,z_j]\cup [z_j,y])$ for the Beckmann problem. More specifically, we have
		\begin{equation*}
		\xi_j=\begin{cases*}
		(z_j-x)/|z_j-x|&$\Ha$-a.e. on $[x,z_j]$,\\
		(y-z_j)/|y-z_j|&$\Ha$-a.e. on $[z_j,y]$.
		\end{cases*}
		\end{equation*}
		Nevertheless, for all $a>1=b$ there does not exist an optimal mass flux for the Beckmann problem. In other words, there is no optimal path between $x$ and $y$ with respect to $d$.
	\end{examp}
	\begin{figure}[t]
		\centering
		\begin{tikzpicture}
		\draw (0,0) -- (2,4);
		\draw (2,4) -- (4,0);
		\draw (0,0) -- (2,2);
		\draw (2,2) -- (4,0);
		\draw (0,0) -- (2,1) -- (4,0);
		
		%	\node[label={180: $S_1$}] at (1,2) {};
		%	\node[label={90: $S_2$}] at (1,1) {};
		%	\node[label={-90: $S_3$}] at (1,.5) {};
		\node[label={-90: $\vdots$}] at (2,1) {};
		
		\node[circle,fill=black,inner sep=1pt,minimum size=0.1cm,label={180: $\mu_+=\delta_x$}] at (0,0) {};
		\node[draw=black,circle,fill=white,inner sep=1pt,minimum size=0.1cm,label={0: $\mu_-=\delta_y$}] at (4,0) {};
		\end{tikzpicture}
		\caption{Sketch for \cref{exnotex}. If $a>b\equiv const.>0$ on $S$ and $\Ha(S)=\infty$, then it may happen that there does not exist an optimal mass flux for the Beckmann problem.}
		\label{fig3}
	\end{figure}
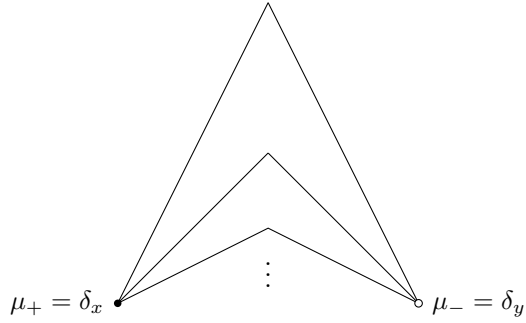
	We used \cref{Sassump} to prove that $d$ is lower semi-continuous for $a=\infty$ in \cref{dLsc}. From this property we get the existence of optimal transport plans (\cite[Thm.\,1.5]{San}).
	\begin{prop}[Existence of optimal transport plan]
		\label{exoptplan}
		Let \cref{Sassump} be satisfied or $a<\infty$. Then there exists an optimal transport plan $\pi\in\Pi(\mu_+,\mu_-)$ such that \begin{equation*}
		W_d(\mu_+,\mu_-)=\int_{\mathcal{C}\times\mathcal{C}}d(x,y)\,\mathrm{d}\pi(x,y).
		\end{equation*}
	\end{prop}
	Since in the proof of \cref{BWfinal}, from any transport plan we constructed a mass flux with no larger cost, this immediately implies the existence of an optimal mass flux for the Beckmann problem.
	\begin{corr}[Existence of optimal mass flux]\label{thm:existenceMassFlux}
		Let \cref{Sassump} be satisfied. Then there exists an optimal mass flux $\F=\xi\Ha\mres S+\F^\perp$
		with $\xi\in L^1(\Ha\mres S;\R^n)$ and $\F^\perp\in\mathcal{M}^n(\mathcal{C})$ such that $\F^\perp\mres S=0$ and $\textup{div}(\xi\Ha\mres S+\F^\perp)=\mu_+-\mu_-$ as well as
		\begin{equation*}
		W_d(\mu_+,\mu_-)=\int_{S}b|\xi|\,\mathrm{d}\Ha+a|\F^\perp|(\mathcal{C}).
		\end{equation*}
	\end{corr}
%\dualityApproach
	
	\section{Bilevel formulation of the branched transport problem}
	\label{sec4}
	Let $\mu_+, \mu_-$ be given probability measures on $\mathcal{B}(\mathcal{C})$. In this section we will prove \cref{finalthm}: The branched transport problem (\cref{defags,defbtp}) of finding an optimal mass flux from $\mu_+$ to $\mu_-$ with respect to a (concave) transportation cost $\tau$ can be equivalently written as a generalized version of the urban planning problem (\cref{def:urbPlnCost,def:urbPlnProblem}). We briefly recapitulate the setting from \cref{intro}. In the urban planning problem one optimizes over countably $1$-rectifiable and Borel measurable networks $S\subset\mathcal{C}$ and lower semi-continuous friction coefficients $b:S\to [0,\infty)$ representing a street or pipe network,
	\begin{equation*}
	\inf_{S,b}\mathcal{U}^{c,\mu_+,\mu_-}[S,b]=\inf_{S,b}W_{d_{S,a,b}}(\mu_+,\mu_-)+\int_Sc(b)\dHa.
	\end{equation*}
	The optimization depends on a fixed, decreasing maintenance cost $c:\R\to[0,\infty]$, and the cost for motion outside the network is defined by $a=\inf c^{-1}(0)$. In the branched transport problem on the other hand the transportation cost is a concave function $\tau:[0,\infty)\to[0,\infty)$ with $\tau(0)=0$, and one looks for an optimal mass flux $\F\in\mathcal{DM}^n(\R^n)$ with $\textup{div}(\F)=\mu_+-\mu_-$,
	\begin{equation*}
	\inf_\F\mathcal{J}^{\tau,\mu_+,\mu_-}[\F],
	\end{equation*}
	where $\mathcal{J}^{\tau,\mu_+,\mu_-}$ is defined via relaxation of a discrete energy.
	We extend $\tau$ to a function $\R\to[-\infty,\infty)$ by setting $\tau(m)=-\infty$ for $m<0$. Moreover, we set 
	\begin{equation*}
	\tau'(0)=\lim_{m\searrow 0}\frac{\tau(m)}{m}\in[0,\infty].
	\end{equation*}
	We use the convex conjugate of $-\tau$ to define a maintenance cost for our generalized urban planning problem which will be shown to be equivalent to the branched transport problem for $\tau$,
	\begin{equation*}
	\varepsilon(v)=(-\tau)^*(-v)=\sup_{m\geq 0}\tau(m)-mv
	\qquad\text{for }v\in\R.
	\end{equation*}
	We observe that by definition
	\begin{equation*}
	\tau'(0)=\inf\varepsilon^{-1}(0)=a.
	\end{equation*}
	The actual statement of \cref{finalthm} to be shown in this section is
	\begin{equation*}
	\inf_{\F}\mathcal{J}^{\tau,\mu_+,\mu_-}[\F]=\inf_{S,b}\mathcal{U}^{\varepsilon,\mu_+,\mu_-}[S,b].
	\end{equation*}
	%where the infima are taken over divergence measure vector fields $\F\in\mathcal{DM}^n(\R^n)$ with $\textup{div}(\F)=\mu_+-\mu_-$, countably $1$-rectifiable and Borel measurable sets $S\subset\mathcal{C}$ and lower semi-continuous functions $b:S\to[0,\infty)$ with $\inf\textup{dom}(\varepsilon)\leq b\leq\tau'(0)$ on $S$.
	%, where $\textup{dom}(\varepsilon)=\varepsilon^{-1}([0,\infty))$ denotes the effective domain of $\varepsilon$. 
	In \cref{subs31} we will formulate an appropriate version of the branched transport problem that will later naturally lead to our Beckmann formulation of the Wasserstein distance from \cref{BWfinal}.
	\dualityApproach{\todo[inline]{If Section 4.2 stays, uncomment lines referring to that.}
		\notinclude{
			The formula in \cite[Prop.\,2.32]{BW} shows that the cost of an admissible mass flux for the branched transport problem can be decomposed into costs for the mass flux on a network $S$ like above and a diffuse part for the costs of the flux outside the network. \textcolor{purple}{In \cref{subs32} we will prove a dual formula for the first term which will be used to introduce the friction coefficient functions $b$ (\cref{subst}). Based on these preparations},
		}%\notinclude
	}%\dualityApproach
	In \cref{bilevel} we then establish the equivalence between the branched transport and the urban planning problem, and we discuss the relation between minimizers of each.
	%Note that the generalized urban planning problem is a bilevel optimization problem. More precisely, the inner optimization corresponds to the convex problem of finding an optimal transport plan and the outer minimization relates to the infimum above. 
	\subsection{Version of the branched transport problem}
	\label{subs31}
	In this section we introduce the reformulation of the branched transport problem from \cite[Prop.\,2.3]{BW} as a generalized Gilbert energy in order to prepare the equivalence proof in \cref{bilevel}. Further, we highlight some properties of the variables appearing in the reformulation. We first note that it suffices to concentrate on mass fluxes with support in $\mathcal{C}$ (in fact, one may replace $\mathcal{C}$ by the convex hull of $\textup{supp}(\mu_+)\cup\textup{supp}(\mu_-)$, see \cite[Lem.\,5.15]{BCM}).
	\begin{lem}[{{\cite[Def.\,2.2 \& Lem.\,2.4]{BW}}}]
		\label{supplem}
		We have
		\begin{equation*}
		\inf_{\F\in\mathcal{DM}^n(\R^n) }\mathcal{J}^{\tau,\mu_+,\mu_-}[\F]=\inf_{\F\in\mathcal{DM}^n(\mathcal{C})}\mathcal{J}^{\tau,\mu_+,\mu_-}[\F].
		\end{equation*}
	\end{lem}
	We will in this section work with the following expression for the branched transport cost.
	\begin{prop}[{{\cite[Prop.\,2.32 and its proof]{BW}}}]
		\label{gengilbert} Every $\F\in\mathcal{DM}^n(\mathcal{C} )$ satisfies $\mathcal{J}^{\tau,\mu_+,\mu_-}[\F]<\infty$ if and only if
		\begin{itemize}
			\item $\textup{div}(\F)=\mu_+-\mu_-$,
			\item $\F=\xi\Ha \mres S+\F^\perp $ with countably $1$-rectifiable $S\subset\mathcal{C}$, $\xi :S\to\R^n$ $\Ha\mres S$-measurable (tangent to $S$ $\Ha$-almost everywhere) and $\F^\perp $ singular with respect to $\Ha\mres R$ for any countably $1$-rectifiable set $R\subset\mathcal{C} $.
		\end{itemize}
		Assume that $\mathcal{J}^{\tau,\mu_+,\mu_-}[\F]<\infty$ and $\F=\xi\Ha \mres S+\F^\perp$ as above. Then the branched transport cost of $\F$ is given by
		\begin{equation*}
		\mathcal{J}^{\tau,\mu_+,\mu_-}[\F]=\int_S\tau (|\xi|)\,\mathrm{d}\Ha+\tau'(0)|\F^\perp|(\mathcal{C}).
		\end{equation*}
		Moreover, we can always choose $S=\{ \Theta^{*1}(|\F|,.)>0\}$ and $\F^\perp=\F\mres (\mathcal{C}\setminus S)$.
	\end{prop}
	\begin{rem}[$S$ is Borel]
		\label{decomp}
		The set $S=\{ \Theta^{*1}(|\F|,.)>0\}$ is Borel measurable by \cite[Prop.\,1.1]{Edg}.
	\end{rem}
	\begin{examp}[$S$ not closed in general]
		\label{exclos}
		For polyhedral mass fluxes (that are supported on finitely many line segments) the set $S$ can clearly be chosen to be closed. In general this is not the case. A simple example is given by $\tau (m)=m$, $\mu_+=\sum_{x\in\mathbb{Q}\cap [0,1]}\varphi (x)\delta_{(x,0)}$ and $\mu_-=\sum_{x\in\mathbb{Q}\cap [0,1]}\varphi (x)\delta_{(x,1)}$, where $\varphi:\mathbb{Q}\cap [0,1]\to [0,1]$ satisfies $\sum_x\varphi (x)=1$ (see \cref{fig5}).
	\end{examp}
	\begin{figure}[t]
		\centering
		\begin{tikzpicture}
		\draw[->] (1,0) -- (1,0.3);
		\draw[->] (2,0) -- (2,0.8);
		\draw[->] (2.5,0) -- (2.5,0.6);
		\draw[->] (3,0) -- (3,0.5);
		
		\draw (1,0) -- (1,4);
		\draw (2,0) -- (2,4);
		\draw (2.5,0) -- (2.5,4);
		\draw (3,0) -- (3,4);
		
		\node[label={[label distance=-0.225cm]180: $\varphi(x)$}] at (2,0.4) {};
		%\node[label={-90: $\mu_+=\sum_{x\in\mathbb{Q}\cap[0,1]}\varphi(x)\delta_{x}$}] at (2,-0.2) {};
		%\node[label={90: $\mu_-=(y\mapsto y+(0,1))_{\#}\mu_+$}] at (2,4) {};
		\node[label={-90: $S=(\mathbb{Q}\cap[0,1])\times[0,1]$}] at (2,-0.2) {};
		
		\node[circle,fill=black,inner sep=0.5pt,minimum size=0.1cm] at (1,0) {};
		\node[circle,fill=black,inner sep=0.5pt,minimum size=0.1cm,label={-90: $x$}] at (2,0) {};
		\node[circle,fill=black,inner sep=0.5pt,minimum size=0.1cm] at (2.5,0) {};
		\node[circle,fill=black,inner sep=0.5pt,minimum size=0.1cm] at (3,0) {};
		
		\node[circle,fill=black,inner sep=0.5pt,minimum size=0.01cm] at (0.5,0) {};
		\node[circle,fill=black,inner sep=0.5pt,minimum size=0.01cm] at (1.5,0) {};
		\node[circle,fill=black,inner sep=0.5pt,minimum size=0.01cm] at (2.25,0) {};
		\node[circle,fill=black,inner sep=0.5pt,minimum size=0.01cm] at (2.75,0) {};
		\node[circle,fill=black,inner sep=0.5pt,minimum size=0.01cm] at (3.5,0) {};
		
		\node[circle,fill=black,inner sep=0.5pt,minimum size=0.01cm] at (0.25,0) {};
		\node[circle,fill=black,inner sep=0.5pt,minimum size=0.01cm] at (1.25,0) {};
		\node[circle,fill=black,inner sep=0.5pt,minimum size=0.01cm] at (2.125,0) {};
		\node[circle,fill=black,inner sep=0.5pt,minimum size=0.01cm] at (2.625,0) {};
		\node[circle,fill=black,inner sep=0.5pt,minimum size=0.01cm] at (3.25,0) {};
		
		\node[circle,fill=black,inner sep=0.5pt,minimum size=0.01cm] at (0.75,0) {};
		\node[circle,fill=black,inner sep=0.5pt,minimum size=0.01cm] at (1.75,0) {};
		\node[circle,fill=black,inner sep=0.5pt,minimum size=0.01cm] at (2.375,0) {};
		\node[circle,fill=black,inner sep=0.5pt,minimum size=0.01cm] at (2.875,0) {};
		\node[circle,fill=black,inner sep=0.5pt,minimum size=0.01cm] at (3.75,0) {};
		
		\node[draw=black,circle,fill=white,inner sep=0.5pt,minimum size=0.1cm] at (1,4) {};
		\node[draw=black,circle,fill=white,inner sep=0.5pt,minimum size=0.1cm] at (2,4) {};
		\node[draw=black,circle,fill=white,inner sep=0.5pt,minimum size=0.1cm] at (2.5,4) {};
		\node[draw=black,circle,fill=white,inner sep=0.5pt,minimum size=0.1cm] at (3,4) {};
		
		\node[circle,fill=black,inner sep=0.5pt,minimum size=0.01cm] at (0.5,4) {};
		\node[circle,fill=black,inner sep=0.5pt,minimum size=0.01cm] at (1.5,4) {};
		\node[circle,fill=black,inner sep=0.5pt,minimum size=0.01cm] at (2.25,4) {};
		\node[circle,fill=black,inner sep=0.5pt,minimum size=0.01cm] at (2.75,4) {};
		\node[circle,fill=black,inner sep=0.5pt,minimum size=0.01cm] at (3.5,4) {};
		
		\node[circle,fill=black,inner sep=0.5pt,minimum size=0.01cm] at (0.25,4) {};
		\node[circle,fill=black,inner sep=0.5pt,minimum size=0.01cm] at (1.25,4) {};
		\node[circle,fill=black,inner sep=0.5pt,minimum size=0.01cm] at (2.125,4) {};
		\node[circle,fill=black,inner sep=0.5pt,minimum size=0.01cm] at (2.625,4) {};
		\node[circle,fill=black,inner sep=0.5pt,minimum size=0.01cm] at (3.25,4) {};
		
		\node[circle,fill=black,inner sep=0.5pt,minimum size=0.01cm] at (0.75,4) {};
		\node[circle,fill=black,inner sep=0.5pt,minimum size=0.01cm] at (1.75,4) {};
		\node[circle,fill=black,inner sep=0.5pt,minimum size=0.01cm] at (2.375,4) {};
		\node[circle,fill=black,inner sep=0.5pt,minimum size=0.01cm] at (2.875,4) {};
		\node[circle,fill=black,inner sep=0.5pt,minimum size=0.01cm] at (3.75,4) {};
		
		%	\draw (0,0) -- (3,1);
		%	\draw[dashed] plot [smooth] coordinates {(-0.3,-0.2)(0,0.2)(2,0.75)(3.3,1.3)};
		%	\node[circle,fill=black,inner sep=0.5pt,minimum size=0.1cm,label={-90: $x$}] at (0,0) {};
		%	\node[circle,fill=black,inner sep=0.5pt,minimum size=0.1cm,label={-90: $y$}] at (3,1) {};
		%	\node[circle,fill=black,inner sep=0.5pt,minimum size=0.1cm,label={180: $x_j$}] at (-0.3,-0.2) {};
		%	\node[circle,fill=black,inner sep=0.5pt,minimum size=0.1cm,label={90: $y_j$}] at (3.3,1.3) {};
		\end{tikzpicture}
		\caption{Sketch for \cref{exclos}.}
		\label{fig5}
	\end{figure}
	Using \cite[Thm.\,3.1]{Sil} it is easy to see that the following properties of $\xi$ and $\F^\perp$ hold true.
	\begin{corr}[Integrability of mass density and property of diffuse part]
		\label{intdens}
		Assume that $\F\in\mathcal{DM}^n(\mathcal{C})$ satisfies $\mathcal{J}^{\tau,\mu_+,\mu_-}[\F]<\infty$ and write $\F=\xi\Ha\mres S+\F^\perp$ as in \cref{gengilbert}. Then the function $\xi$ is integrable with respect to $\Ha\mres S$ and $\F^\perp\mres S=0$. Those properties are independent of the triple $(\xi,S,\F^\perp)$.
	\end{corr}
	\begin{proof}
		By \cite[Thm.\,3.1]{Sil} every $\F\in\mathcal{DM}^n(\mathcal{C})$ can be written as $\F=\vartheta\Ha\mres M+\G+\psi\mathcal{L}^n\mres\mathcal{C}$ with $M\subset\mathcal{C}$ countably $1$-rectifiable, $\vartheta\in L^1(\Ha\mres M;\R^n)$ tangent to $M$ $\Ha$-almost everywhere, $\G\in\mathcal{M}^n(\mathcal{C})$ $\Ha$-diffuse and $\mathcal{L}^n$-singular as well as $\psi\in L^1(\mathcal{L}^n\mres\mathcal{C};\R^n)$. If $\mathcal{J}^{\tau,\mu_+,\mu_-}[\F]<\infty$, we have 
		\begin{equation*}
		\F=\xi\Ha\mres S+\F^\perp=\vartheta\Ha\mres M+\G+\psi\mathcal{L}^n\mres\mathcal{C}
		\end{equation*}
		with $S,\xi,\F^\perp$ as in \cref{gengilbert}. If $n>1$, we get $(\G+\psi\mathcal{L}^n\mres\mathcal{C})\mres S=0$. Using this and $\F^\perp\perp\Ha\mres S$ (\cref{gengilbert}) we obtain $\F^\perp\mres S=0$ and $\xi\Ha\mres S=\vartheta\Ha\mres (M\cap S)$, which yields $\xi\in L^1(\Ha\mres S;\R^n)$. For the case $n=1$ we can write $\xi\Ha\mres S+\F^\perp=\zeta\Ha\mres\mathcal{C}$ with the $(\Ha\mres\mathcal{C})$-integrable function
		\begin{equation*}
		\zeta=\begin{cases*}
		\vartheta+\psi&on $M$,\\
		\psi&\textup{on} $\mathcal{C}\setminus M$.
		\end{cases*}
		\end{equation*}
		The same argument as for the case $n>1$ then yields $\xi\in L^1(\Ha\mres S;\R^n)$ and $\F^\perp\mres S=0$.
	\end{proof}
	For the next result we will need the notion of irrigation patterns (see \cite{MSM03,BCM05,MS13}), which are an alternative way to describe mass fluxes involving a time dependency. Broadly speaking, a mass flux between $\mu_+$ and $\mu_-$ can be seen as a superposition of particle trajectories. The so-called standard space $([0,1],\mathcal{B}([0,1]),\mathcal{L}\mres [0,1])$ can be used as a parameterization of all particles \cite[Ch.\,15, Thm.\,16]{R88}.
	\begin{defin}[Reference space, irrigation pattern between $\mu_+,\mu_-$, total mass flux through $x$]
		We define a map $\chi:[0,1]\times I\to\mathcal{C}$ such that $\chi(p,t)$ describes the position of particle $p$ at time $t$.
		\begin{itemize}
			\item The \textbf{reference space} for particles is the measure space $([0,1],\mathcal{B}([0,1]),\mathcal{L}\mres [0,1])$.
			\item An \textbf{irrigation pattern} is a Borel measurable map $\chi:[0,1]\times I\to\mathcal{C}$ such that $\chi(p,.)$ is absolutely continuous for $\mathcal{L}$-almost all $p\in [0,1]$.
		\end{itemize}
		Let $\chi$ be an irrigation pattern.
		\begin{itemize}
			\item We say that $\chi$ is an \textbf{irrigation pattern between the probability measures} $\mu_+$ and $\mu_-$ if and only if
			\begin{equation*}
			\mu_+(B)=\mathcal{L}(\{ p\in [0,1]\,|\,\chi(p,0)\in B \})\qquad\textup{and}\qquad\mu_-(B)=\mathcal{L}(\{ p\in [0,1]\,|\,\chi(p,1)\in B\})
			\end{equation*}
			for all $B\in\mathcal{B}(\mathcal{C})$.
			\item For $x\in\mathcal{C}$ we set $[x]_\chi=\{ p\in [0,1]\,|\, x\in\chi(p,I) \}$. The \textbf{total mass flux} through $x$ is defined by $m_\chi(x)=\mathcal{L}([x]_\chi)$.
		\end{itemize}
	\end{defin}
	The following summarizing statement will be used in \cref{bilevel} to explicitly construct a lower semi-continuous friction coefficient $b:S\to[0,\infty)$ based on $x\mapsto|\xi(x)|$. Note that the first two points follow directly from \cref{gengilbert,decomp,intdens}.
	\begin{corr}[Other properties of mass density]
		\label{propmassdens}
		Assume that there is a mass flux $\G\in\mathcal{DM}^n(\mathcal{C} )$ with $\mathcal{J}^{\tau,\mu_+,\mu_-}[\G]<\infty$. Then there exists some $\F\in\mathcal{DM}^n(\mathcal{C} )$ with $\mathcal{J}^{\tau,\mu_+,\mu_-}[\F]\leq\mathcal{J}^{\tau,\mu_+,\mu_-}[\G]$ and decomposition $\F=\xi\Ha \mres S+\F^\perp $ as in \cref{gengilbert} such that
		\begin{itemize}
			\item $S$ is countably $1$-rectifiable and Borel measurable,
			\item $\xi\in L^1(\Ha\mres S;\R^n)$ and $\F^\perp\mres S=0$,
			\item we can choose a representative of $\xi$ such that $|\xi|$ is bounded by the total mass $\mu_+(\mathcal{C})=1$ on $S$, $S\ni x\mapsto |\xi(x)|$ is upper semi-continuous, $\{ |\xi |\geq m \}=\{ x\in S\,|\, |\xi(x)|\geq m \}$ is closed in $\mathcal{C}$ and $\Ha(\{ |\xi |\geq m \})<\infty$ for every $m>0$,
			\item we can write $S=\bigcup_{m>0}\{ |\xi |\geq m \}$.
		\end{itemize}
	\end{corr}
	\begin{proof}
		By \cref{gengilbert,decomp,intdens} we can write $\G=\xi_\G\Ha\mres S+\G^\perp$ with $S$ countably $1$-rectifiable and Borel measurable, $\xi_\G\in L^1(\Ha\mres S;\R^n)$ and $\G^\perp\in\mathcal{M}^n(\mathcal{C})$ with $\G^\perp\mres S=0$. As in the proof of the first inequality of \cref{BWfinal} in \cref{subs5} we use the idea in the proof of \cite[Prop.\,4.1]{BW}; we briefly recapitulate the steps: By \cite[Thm.\,C]{S} we have $\G=\F_0+\F$, where $\textup{div}(\F_0)=0$ and $\F$ can be decomposed into simple oriented curves of finite length, i.e., into measures of type $\F_\gamma$ from \cref{lintm} (see also \cite[Exm.\,1]{S}). The measure $\F$ can be decomposed as $\F=\xi\Ha \mres S+\F^\perp\in\mathcal{DM}^n(\mathcal{C})$ with $|\xi|\leq|\xi_\G|$ and $|\F^\perp|(\mathcal{C})\leq|\G^\perp|(\mathcal{C})$ and satisfies $\textup{div}(\F)=\mu_+-\mu_-$ as well as $\mathcal{J}^{\tau,\mu_+,\mu_-}[\F]\leq\mathcal{J}^{\tau,\mu_+,\mu_-}[\G]$. Furthermore, $\F$ can be associated with a mass flux measure $\eta$ on $\Theta$ moving $\mu_+$ onto $\mu_-$ (recall \cref{mfmeas}) by \cite[Thm.\,C]{S}. More precisely, we have
		\begin{align*}
		\int_{\mathcal{C}}\varphi\cdot\,\mathrm{d}\F&=\int_{\Theta }\int_I\varphi (\gamma (t) )\cdot\dot{\gamma }(t)\,\mathrm{d}\Le(t)\,\mathrm{d}\eta (\gamma )\quad\textup{ for all }\varphi\in C(\mathcal{C};\R^n)\textup{ and}\\
		\int_{\mathcal{C} }\psi\,\mathrm{d}|\F|&=\int_{\Theta }\int_I\psi (\gamma (t))|\dot{\gamma (t)}|\,\mathrm{d}\Le(t)\,\mathrm{d}\eta (\gamma )\quad\textup{ for all }\psi\in C(\mathcal{C}).
		\end{align*}
		Using Skorohod's theorem \cite[Thm.\,6.7]{B} there is an irrigation pattern $\chi_\F$ between $\mu_+$ and $\mu_-$ which induces $\eta$,
		\begin{align*}
		\int_{\mathcal{C}}\varphi\cdot\,\mathrm{d}\F&=\int_{[0,1] }\int_I\varphi (\chi_\F(p,t) )\cdot \dot{\chi }_\F(p,t)\,\mathrm{d}\Le(t)\,\mathrm{d}\Le (p)\quad\textup{ for all }\varphi\in C(\mathcal{C};\R^n)\textup{ and}\\
		\int_{\mathcal{C} }\psi\,\mathrm{d}|\F|&=\int_{[0,1] }\int_I\psi (\chi_\F(p,t))|\dot{\chi }_\F(p,t)|\,\mathrm{d}\Le(t)\,\mathrm{d}\Le (p) \quad\textup{ for all }\psi\in C(\mathcal{C}).
		\end{align*}
		Here $\dot{\chi}_\F$ denotes the derivative with respect to the second argument (which exists $\Le\mres I$-almost everywhere). By \cite[Prop.\,4.2]{BW} we may assume $S=\{ m_{\chi_\F}>0\}$. Additionally, $m_{\chi_\F}(x)=|\xi (x)|$ for $\Ha$-almost every $x\in S$ by the proof of \cite[Prop.\,4.1]{BW}. This shows the desired formula for $S$ by changing $\xi$ such that $m_{\chi_\F}=|\xi |$ on $S$. More specifically, represent $\xi$ by
		\begin{equation*}
		\tilde{\xi}=\begin{cases}
		\xi&\textup{on }\{ |\xi| =m_{\chi_\F}\},\\
		(m_{\chi_\F},0,\ldots ,0)^T&\textup{on }\{ |\xi| \neq m_{\chi_\F}\}.
		\end{cases}
		\end{equation*}
		%		We show that $\tilde{\xi }$ is Borel measurable. Let $A\subset\R^n$ be Borel measurable and write $X_1=\R\times \{ 0\}^{n-1}$. Then
		%		\begin{align*}
		%			\tilde{\xi}^{-1}(A)&=\tilde{\xi}^{-1}(A\cap X_1)\cup\tilde{\xi}^{-1}(A\setminus X_1)
		%			=(\tilde{\xi}^{-1}(A\cap X_1)\cap\{ |\xi |\neq m_{\chi_\F}\})\cup (\tilde{\xi}^{-1}(A\cap X_1)\cap \{ |\xi|=m_{\chi_\F}\} )\cup\tilde{\xi}^{-1}(A\setminus X_1)\\
		%			&=(m_{\chi_\F}^{-1}(p_1(A\cap X_1))\cap\{ |\xi |\neq m_{\chi_\F}\})\cup (\xi^{-1}(A\cap X_1)\cap \{ |\xi|=m_{\chi_\F}\} )\cup(\xi^{-1}(A\setminus X_1)\cap \{ |\xi|=m_{\chi_\F} \})
		%		\end{align*}
		%		is Borel measurable.
		%The end of the proof of \cite[Prop.\,4.1]{BW} shows
		%\begin{equation*}
		%\mathcal{J}^{\tau,\mu_+,\mu_-}[\chi_\F]=\mathcal{J}^{\tau,\mu_+,\mu_-}[\F]=\int_S\tau (|\theta |)\,\mathrm{d}\Ha+\tau'(0)|\F^\perp |(\mathcal{C} ).
		%\end{equation*}
		%Thus, we see that the energy does not change and one still has that $|\theta |(x)$ is tangent to $S$ for $\Ha$-almost every $x\in S$. 
		Now fix $m>0$ and let $(x_i)\subset\{|\xi|\geq m\}$ be a sequence with $x_i\to x\in\mathcal{C}$. By \cite[Lem.\,3.25]{BCM} the function $\mathcal{C}\ni y\mapsto m_{\chi_\F}(y)$ is upper semi-continuous. This implies $m_{\chi_\F}(x)\geq\limsup_im_{\chi_\F}(x_i)=\limsup_i|\xi (x_i)|\geq m$. In particular, we have $m_{\chi_\F}(x)>0$ and thus $x\in S$, which implies $|\xi(x)|=m_{\chi_\F}(x)\geq m$. This proves the closedness of $\{|\xi|\geq m\}$. Moreover, $|\xi|$ is upper semi-continuous by $|\xi|=m_{\chi_\F}$ on $S$. The boundedness of $|\xi|$ follows from
		\begin{equation*}
		|\xi|=m_{\chi_\F}=\mathcal{L}([.]_{\chi_\F})\leq 1.
		\end{equation*}
		Finally, we have
		\begin{align*}
		\infty>\mathcal{J}^{\tau,\mu_+,\mu_-}[\F]=\int_S\tau (|\xi |)\,\mathrm{d}\Ha+\tau'(0)|\F^\perp |(\mathcal{C} )\geq\int_{\{ |\xi|\geq m\}}\tau (|\xi |)\,\mathrm{d}\Ha\geq \int_{\{ |\xi|\geq m\}}\tau (m)\,\mathrm{d}\Ha=\tau (m)\Ha(\{ |\xi|\geq m\})
		\end{align*}
		and thus $\Ha(\{ |\xi|\geq m\})<\infty$ for all $m>0$.
	\end{proof}
	We end this subsection by reformulating the branched transport problem such that the variables to be optimized are as in the setting of the Beckmann problem from \cref{BWfinal}.
	\begin{lem}[Version of the branched transport problem]
		\label{version}
		The branched transport problem can be written as
		\begin{equation*}
		\inf_{\F\in\mathcal{DM}^n(\mathcal{C}) }\mathcal{J}^{\tau,\mu_+,\mu_-}[\F]=\inf_{S,\xi,\F^\perp}\int_S\tau (|\xi|)\,\mathrm{d}\Ha+\tau'(0)|\F^\perp|(\mathcal{C})
		\end{equation*}
		with $S\subset \mathcal{C}$ countably $1$-rectifiable and Borel measurable, $\xi\in L^1(\Ha\mres S;\R^n)$ and $\F^\perp\in\mathcal{M}^n(\mathcal{C})$ with $\F^\perp\mres S=0$ and $\textup{div}(\xi\Ha\mres S+\F^\perp)=\mu_+-\mu_-$.
	\end{lem}
	\begin{proof}
		By \cref{supplem,gengilbert,intdens,decomp} the right-hand side is automatically smaller than or equal to the left-hand side. For the reverse inequality, let $S,\xi,\F^\perp$ satisfy the stated properties. We can assume that $\textup{div}(\xi\Ha\mres S+\F^\perp)=\mu_+-\mu_-$ (otherwise the inequality is obvious). By \cite[Thm.\,3.1]{Sil} we have $\xi\Ha\mres S+\F^\perp=\vartheta\Ha\mres M+\G$ with $M$ countably $1$-rectifiable, $\vartheta\in L^1(\Ha\mres M;\R^n)$ tangent to $M$ $\Ha\mres M$-almost everywhere and $\G$ $\Ha$-diffuse. This is an admissible decomposition in the sense of \cref{gengilbert}. We have $\xi\Ha\mres S=\vartheta\Ha\mres (M\cap S)$ and $\xi=\vartheta$ $\Ha$-almost everywhere on $M\cap S$ as well as $\xi=0$ $\Ha$-almost everywhere on $S\setminus M$. Additionally, we get $\F^\perp=\F^\perp\mres (\mathcal{C}\setminus S)=\vartheta\Ha\mres (M\setminus S)+\G\mres (\mathcal{C}\setminus S)$. To conclude, we estimate 
		\begin{align*}
		\mathcal{J}^{\tau,\mu_+,\mu_-}[\vartheta\Ha\mres M+\G]&=\int_M\tau (|\vartheta|)\,\mathrm{d}\Ha+\tau'(0)|\G|(\mathcal{C})=\int_{M\cap S}\tau (|\xi|)\,\mathrm{d}\Ha+\int_{M\setminus S}\tau (|\vartheta|)\,\mathrm{d}\Ha+\tau'(0)|\G|(\mathcal{C})\\
		&\leq \int_{M\cap S}\tau (|\xi|)\,\mathrm{d}\Ha+\tau'(0)\int_{M\setminus S} |\vartheta|\,\mathrm{d}\Ha+\tau'(0)|\G|(\mathcal{C})\\
		&= \int_{S}\tau (|\xi|)\,\mathrm{d}\Ha+\tau'(0)|\vartheta\Ha\mres (M\setminus S)+\G|(\mathcal{C})\\
		&=\int_{S}\tau (|\xi|)\,\mathrm{d}\Ha+\tau'(0)|\F^\perp|(\mathcal{C}).\qedhere
		\end{align*}
	\end{proof}
%\dualityApproach
	
	\subsection{Branched transport problem as generalized urban planning problem}
	\label{bilevel}
	In this section we finally prove \cref{finalthm}, essentially by constructing a minimizer for one problem from one of the other. We will also discuss a few examples illustrating the relation between the minimizers. Throughout we assume that $(\tau,\mu_+,\mu_-)$ is a triple for the branched transport problem and the corresponding maintenance cost $\varepsilon$ is defined via $\tau$ as in \cref{maintcost}.
	%	\begin{thm}[Bilevel formulation of the branched transport problem with concave transportation cost $\tau$]
	%		\label{finalthm}
	%		The branched transport problem can be written
	%		\begin{equation*}
	%		\inf_{\F}\mathcal{J}^{\tau,\mu_+,\mu_-}[\F]=\inf_{S,b}\mathcal{U}^{\varepsilon,\mu_+,\mu_-}[S,b],
	%		\end{equation*}
	%		where the infima are taken over $\F\in\mathcal{DM}^n(\mathcal{C})$ with $\textup{div}(\F)=\mu_+-\mu_-$, countably $1$-rectifiable and Borel measurable $S\subset\mathcal{C}$ and lower semi-continuous functions $b:S\to[0,\infty)$ with $\inf\textup{dom}(\varepsilon)\leq b\leq\tau'(0)$ on $S$.
	%	\end{thm}
	% Unless otherwise specified, for the rest of this section we will always assume that $\F,S,b$ satisfy the conditions from \cref{finalthm}.
	Let $\tau_+(0)=\lim_{m\searrow 0}\tau(m)$. We define 
	\begin{equation*}
	\tilde{\tau}(m)=\begin{cases*}
	\tau(m)&if $m\leq 0$,\\
	\tau(m)-\tau_+(0)&else
	\end{cases*}
	\quad\text{for }m\in\R
	\qquad\text{and}\qquad
	\tilde{\varepsilon}(v)=(-\tilde{\tau})^*(-v)
	\quad\text{for }v\in\R.
	\end{equation*}
	We will need the following properties of $\tilde{\tau}$.
	\begin{lem}[Properties of $\tilde{\tau}$]
		\label{tildetau}
		The transportation cost $\tilde{\tau}$ is right-continuous in $0$. Furthermore, we have $\tilde{\varepsilon}=\varepsilon-\tau_+(0)$ and 
		\begin{equation*}
		\int_S\tau(|\xi|)\dHa=\int_S\tilde{\tau}(|\xi|)\dHa+\tau_+(0)\Ha(\{ |\xi|>0 \})
		\end{equation*}
		for all $\xi\in L^1(\Ha\mres S;\R^n)$.
	\end{lem}
	\begin{proof}
		The first property follows by definition. Let $v\in\R$ and assume that $\tilde{\varepsilon}(v)>0$. We have
		\begin{equation*}
		\tilde{\varepsilon}(v)=\sup_{m>0}\tilde{\tau}(m)-mv=\sup_{m>0}\tau(m)-\tau_+(0)-mv=\sup_{m\in\R}\tau(m)-mv-\tau_+(0)=\varepsilon(v)-\tau_+(0).
		\end{equation*}
		If $\tilde{\varepsilon}(v)=0$, then we get $v\geq\tilde{\tau}'(0)$ and therefore $\varepsilon(v)=\tau_+(0)$. Finally, we observe that
		\begin{equation*}
		\int_S\tau(|\xi|)\dHa=\int_{\{ |\xi|>0 \}}\tau(|\xi|)\dHa=\int_{\{ |\xi|>0 \}}\tilde{\tau}(|\xi|)+\tau_+(0)\dHa=\int_{\{ |\xi|>0 \}}\tilde{\tau}(|\xi|)\dHa+\tau_+(0)\Ha(\{ |\xi|>0 \}).\qedhere
		\end{equation*}
	\end{proof}
	For the inequality $\inf\mathcal{J}^{\tau,\mu_+,\mu_-}\geq\inf\mathcal{U}^{\varepsilon,\mu_+,\mu_-}$ we will use the following standard composition property which we state without proof.
	\begin{lem}[Composition of semi-continuous functions]
		\label{uplow}
		Let $(X,d_X)$ be a metric space and $A\subset\R$. If $g:X\to A$ is upper semi-continuous and $f:A\to\R$ lower semi-continuous and decreasing, then $f\circ g$ is lower semi-continuous.
	\end{lem}
	\notinclude{\begin{proof}
			Let $\delta>0$ and $x,x_j\in X$ such that $d_X(x,x_j)\to 0$. We have $g(x_j)<g(x)+\delta$ for $j$ sufficiently large, because $g$ is upper semi-continuous. Hence, we get $f(g(x)+\delta)\leq f(g(x_j))$ for all such $j$ by the assumption that $f$ is decreasing. Thus, we have
			\begin{equation*}
			f(g(x)+\delta)\leq\liminf_j f(g(x_j))
			\end{equation*}
			for all $\delta>0$ and therefore, using that $f$ is lower semi-continuous,
			\begin{equation*}
			f(g(x))\leq\liminf_{\delta\to 0}f(g(x)+\delta)\leq\liminf_j f(g(x_j)).\qedhere
			\end{equation*}
		\end{proof}
	}%\notinlcude
	\begin{proof}[Proof of \cref{finalthm}]
		\underline{$\inf\mathcal{J}^{\tau,\mu_+,\mu_-}\leq\inf\mathcal{U}^{\varepsilon,\mu_+,\mu_-}$:} Assume that
		\begin{equation*}
		\mathcal{U}^{\varepsilon,\mu_+,\mu_-}[S,b]<\infty
		\end{equation*}
		for some admissible pair $(S,b)$.
		As discussed in the introduction of \cref{sect2}, this automatically implies the validity of \cref{Sassump}.
		Thus by \cref{thm:existenceMassFlux} we have
		\begin{equation*}
		W_{d_{S,\tau'(0),b}}(\mu_+,\mu_-)=\int_Sb|\xi|\,\mathrm{d}\Ha+\tau'(0)|\F^\perp|(\mathcal{C})
		\end{equation*}
		for some $\xi\in L^1(\Ha\mres S;\R^n)$ and $\F^\perp\in\mathcal{M}^n(\mathcal{C})$ with $\F^\perp\mres S=0$.
		We show that $\mathcal{J}^{\tau,\mu_+,\mu_-}[\F]\leq\mathcal{U}^{\varepsilon,\mu_+,\mu_-}[S,b]$ for $\F=\xi\Ha\mres S+\F^\perp$. If $\tau$ is right-continuous in $0$, then $-\tau$ is lower semi-continuous and convex and thus equals its biconjugate \cite[Prop.\,2.28]{R}. This yields 
		\begin{equation*}
		\tau (m)=-(-\tau )^{**}(m)=-\left(\sup_{v\in\R}vm-(-\tau )^*(v)\right)=-\left(\sup_{v\in\R}-vm-\varepsilon (v)\right)=-\varepsilon^*(-m)
		\end{equation*}
		and thus
		\begin{multline*}
		\mathcal{U}^{\varepsilon,\mu_+,\mu_-}[S,b]=\int_Sb|\xi|\dHa+\tau'(0)|\F^\perp|(\mathcal{C})+\int_S\varepsilon(b)\dHa\geq\int_S\inf_{v\in\R}|\xi|v+\varepsilon(v)\,\mathrm{d}\Ha+\tau'(0)|\F^\perp|(\mathcal{C})\\
		=-\int_S\varepsilon^*(-|\xi|)\,\mathrm{d}\Ha+\tau'(0)|\F^\perp|(\mathcal{C})=\int_S\tau(|\xi|)\dHa+\tau'(0)|\F^\perp|(\mathcal{C})=\mathcal{J}^{\tau,\mu_+,\mu_-}[\F]
		\end{multline*}
		by \cref{gengilbert}.
		If $\tau$ is not right-continuous, then we have $\tau'(0)=\infty$ and therefore $\F^\perp=0$. Moreover, using \cref{tildetau} and the previous estimate we obtain
		\begin{multline*}
		\mathcal{U}^{\varepsilon,\mu_+,\mu_-}[S,b]
		=\int_Sb|\xi|\dHa+\int_S\tilde{\varepsilon}(b)\dHa+\tau_+(0)\Ha(S)
		\geq\int_S\tilde{\tau}(|\xi|)\dHa+\tau_+(0)\Ha(S)\\
		\geq\int_S\tilde{\tau}(|\xi|)\dHa+\tau_+(0)\Ha(\{ |\xi|>0 \})
		=\int_S\tau(|\xi|)\dHa
		=\mathcal{J}^{\tau,\mu_+,\mu_-}[\F].
		\end{multline*}
		\underline{$\inf\mathcal{J}^{\tau,\mu_+,\mu_-}\geq\inf\mathcal{U}^{\varepsilon,\mu_+,\mu_-}$:} Assume that there exists some $\G\in\mathcal{DM}^n(\mathcal{C})$ with $\mathcal{J}^{\tau,\mu_+,\mu_-}[\G]<\infty$. Let $\F=\xi\Ha\mres S+\F^\perp$ be as in \cref{propmassdens} ($\F$ can be constructed by removing divergence-free parts of $\G$). The function $S\ni x\mapsto g(x)=|\xi(x)|$ is upper semi-continuous by \cref{propmassdens}. Thus, $\{ |\xi|=0 \}\subset S$ is Borel measurable and we can assume without loss of generality that $S=\{ |\xi|>0 \}$. For $m>0$ we define
		\begin{equation*}
		f(m)=-\max(\partial(-\tau)(m)),
		\end{equation*}
		which is well-defined, because the subdifferential is closed. Furthermore, $f$ is decreasing and lower semi-continuous on $(0,\infty)$ by construction. Using \cref{uplow} the function $b:S\to[0,\infty)$ defined by
		\begin{equation*}
		b(x)=f(g(x))=-\max(\partial(-\tau)(|\xi(x)|))
		\end{equation*}
		is lower semi-continuous on $S$. Additionally, we have $\varepsilon(b)=\tau(|\xi|)-|\xi|b$ by definition. Therefore, by \cref{gengilbert} we get
		\begin{equation*}
		\infty>\mathcal{J}^{\tau,\mu_+,\mu_-}[\G]\geq\mathcal{J}^{\tau,\mu_+,\mu_-}[\F]=\int_S\tau(|\xi|)\dHa+\tau'(0)|\F^\perp|(\mathcal{C})=\int_Sb|\xi|\dHa+\tau'(0)|\F^\perp|(\mathcal{C})+\int_S\varepsilon(b)\dHa.
		\end{equation*}
		Thus we must have $\int_S\varepsilon(b)\dHa<\infty$ which shows that \cref{Sassump} is satisfied. Hence we can apply \cref{BWfinal} and continue the estimation,
		\begin{equation*}
		\mathcal{J}^{\tau,\mu_+,\mu_-}[\G]\geq\int_Sb|\xi|\dHa+\tau'(0)|\F^\perp|(\mathcal{C})+\int_S\varepsilon(b)\dHa\geq W_{d_{S,\tau'(0),b}}(\mu_+,\mu_-)+\int_S\varepsilon(b)\dHa=\mathcal{U}^{\varepsilon,\mu_+,\mu_-}[S,b].\qedhere
		\end{equation*}
	\end{proof}
	In the remainder of the section we discuss a few consequences of the proof.
	In \cite[Thm.\,2.10]{BW} it is shown that the generalized branched transport problem either has a minimizer or is infeasible,
	i.e., there is no mass flux of finite energy transporting $\mu_+$ to $\mu_-$.
	(Note that under additional growth conditions on $\tau$ near $0$ one can always obtain existence of a minimizer independent of $\mu_+$ and $\mu_-$ \cite[Cor.\,2.20]{BW}.)
	Since in the proof of \cref{finalthm} we constructed from each feasible candidate for one problem a feasible candidate for the other, this immediately implies the following.
	\begin{corr}[Existence of optimizers]
		The generalized branched transport problem and the associated generalized urban planning problem either both admit a minimizer or are both infeasible.
	\end{corr}
	While for $\tau'(0)=\infty$ the optimal mass flux of the generalized branched transport problem is known to be rectifiable \cite[Thm.\,7.1]{Wh99},
	one gets an even stronger result if $\tau$ is not right-continuous in $0$.
	\begin{rem}[Finite network length]
		If $\tau$ is not right-continuous in $0$, then any street network $(S,b)$ with finite urban planning cost satisfies $\Ha(S)<\infty$.
		Indeed, from \cref{maintcost} of $\varepsilon$ we obtain $\varepsilon\geq\tau_+(0)$ so that
		\begin{equation*}
		\infty
		>\mathcal{U}^{\varepsilon,\mu_+,\mu_-}[S,b]
		\geq\int_S\varepsilon(b)\,\mathrm{d}\Ha
		\geq\tau_+(0)\Ha(S).
		\end{equation*}
	\end{rem}
	Finally, let us briefly discuss and illustrate the relation between minimizers of the branched transport and the urban planning problem.
	A natural question is whether they are in one-to-one correspondence.
	However, this is not to be expected for the following reason.
	In our equivalence proof, the central step to switch between the mass flux and the friction coefficient as variables was the relation
	\begin{equation*}
	\tau(|\xi|)=b|\xi|+\varepsilon(b),
	\end{equation*}
	which we exploited to construct one variable from the other.
	Note that $\tau(|\xi|)\leq b|\xi|+\varepsilon(b)$ is nothing else than the Fenchel--Young inequality,
	and if both $\xi$ and $b$ should be optimal one needs to have equality.
	However, this equality only yields a one-to-one relation between $|\xi|$ and $b$ if both $\tau$ and $\varepsilon$ are differentiable or equivalently strictly convex.
	If, however, $\tau$ has a kink at $|\xi|$, then there exist multiple $b$ satisfying the equality.
	Likewise, if $\varepsilon$ has a kink at $b$, then there exist multiple solutions $|\xi|$.
	Consequently, a single minimizer of the branched transport problem will sometimes correspond to multiple minimizers of the urban planning problem and vice versa.
	We close this section by illustrating this fact with three examples.
	The first example is standard in classical optimal transport theory and illustrates
	that there may be multiple optimal mass fluxes, while the optimal friction coefficient is unique.
	In fact, the classical Wasserstein cost is a rather degenerate case of branched transport for which $\inf\{ \varepsilon<\infty \}=\tau'(0)$ (see \cref{fig1})
	so that the friction coefficient (which has to lie in between both values) is uniquely fixed a priori.
	We provide a less degenerate example directly after, in which the uniqueness of the friction coefficient comes from a kink in $\varepsilon$.
	\begin{examp}[Infinitely many optimal mass fluxes, but unique friction coefficient]
		\label{infitF}
		Let $\tau(m)=m$ and consider %the mass distributions given by 
		\begin{equation*}
		\mu_+=\tfrac{1}{2}\delta_{x_1}+\tfrac{1}{2}\delta_{x_2}\textup{\qquad and\qquad}\mu_-=\tfrac{1}{2}\delta_{y_1}+\tfrac{1}{2}\delta_{y_2}
		\end{equation*}
		for $x_1=(0,0)$, $x_2=(2,0)$, $y_1=(1,1)$, $y_2=(1,-1)$ (see \cref{fig6}). Then the optimal solution for the urban planning problem is given by $b=1$ $\Ha$-almost everywhere ($S$ arbitrary). Since this choice of $b$ is the only feasible one, it is unique.
		Nevertheless, there exist infinitely many solutions to the branched transport problem. Each one is associated with an optimal transport plan for the Wasserstein-1-distance between $\mu_+$ and $\mu_-$, where the family of optimal transport plans can be parameterized by 
		\begin{equation*}
		m=\pi(\{  x_1\}\times\{ y_1 \})\in[0,\tfrac12].
		\end{equation*}
	\end{examp}
	\begin{figure}
		\centering
		\begin{tikzpicture}[scale=2]
		\draw[thick] (0,0) -- (1,1) -- (2,0) -- (1,-1) -- (0,0);
		\draw[-{Straight Barb[width=3mm,length=3mm]},thick] (0,0) -- (0.5,0.5);
		\draw[-{Straight Barb[width=3mm,length=3mm]},thick] (0,0) -- (0.5,-0.5);
		\draw[-{Straight Barb[width=3mm,length=3mm]},thick] (2,0) -- ++(-0.5,0.5);
		\draw[-{Straight Barb[width=3mm,length=3mm]},thick] (2,0) -- ++(-0.5,-0.5);
		\node[circle,fill=black,inner sep=0.5pt,minimum size=0.1cm] at (0,0) {};
		\node[circle,fill=black,inner sep=0.5pt,minimum size=0.1cm] at (2,0) {};
		\node[draw=black,circle,fill=white,inner sep=0.5pt,minimum size=0.1cm] at (1,1) {};
		\node[draw=black,circle,fill=white,inner sep=0.5pt,minimum size=0.1cm] at (1,-1) {};
		\node[label={[label distance=-0.3cm]90: $\F_m$}] at (1,0) {};
		\node[label={180: $m$}] at ({0.5/sqrt(2)},{0.5/sqrt(2)}) {};
		\node[label={0: $1/2-m$}] at ({2-0.5/sqrt(2)},{0.5/sqrt(2)}) {};
		\node[label={180: $1/2-m$}] at ({0.5/sqrt(2)},{-0.5/sqrt(2)}) {};
		\node[label={0: $m$}] at ({2-0.5/sqrt(2)},{-0.5/sqrt(2)}) {};
		\node[label={180: $x_1$}] at (0,0) {};
		\node[label={90: $y_1$}] at (1,1) {};
		\node[label={0: $x_2$}] at (2,0) {};
		\node[label={270: $y_2$}] at (1,-1) {};
		\end{tikzpicture}
		\caption{Sketch illustrating \cref{infitF}. The mass flux $\F_m$ is optimal for all $m\in[0,1/2]$.}
		\label{fig6}
	\end{figure}
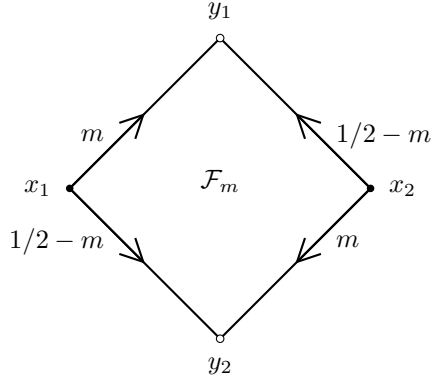
	\begin{examp}[$\varepsilon$ has a kink]
		\label{epskink}
		Consider the initial and final mass
		\begin{equation*}
		\mu_+=\tfrac25\delta_{(0,0)}+\tfrac15\delta_{(1,-\ell)}+\tfrac25\delta_{(2,0)}\textup{\qquad and\qquad}\mu_-=\tfrac25\delta_{(0,1)}+\tfrac15\delta_{(1,1+\ell)}+\tfrac25\delta_{(2,1)}
		\end{equation*}
		with parameter $\ell> 0$.
		If $\tau(\frac35)<\tau(\frac15)+\tau(\frac25)$ and $\ell$ is chosen sufficiently large we obtain two symmetric solutions for an optimal mass flux $\F$, no matter how $\tau$ looks like: The mass from $(1,-\ell)$ will be jointly transported with either the mass from $(0,0)$ or the mass from $(2,0)$ (see \cref{sf:fig8c} left). Moreover, one can argue that for large enough $\ell$ the mass from $(1,-\ell)$ will in fact in one case be transported through $(0,0)$ as well as $(0,1)$ and in the other case through $(2,0)$ as well as $(2,1)$:
		Indeed, if the mass from $(0,0)$ and $(-1,\ell)$ would combine in a point $p_\ell\neq(0,0)$,
		then the so-called momentum conservation (a local optimality condition at triple junctions \cite[Prop.\,4.5]{Xi04}) reads
		\begin{equation*}
		\tau(\tfrac15+\tfrac25)\cdot(0,1)=\tau(\tfrac25)\cdot v_\ell+\tau(\tfrac15)\cdot w_\ell\textup{ with }v_\ell=\frac{p_\ell}{|p_\ell|}\textup{ and }w_\ell=\frac{p_\ell-(1,-\ell)}{|p_\ell-(1,-\ell)|}.
		\end{equation*}
		For $\ell\to\infty$ the angle between $v_\ell$ and $w_\ell$ would go to zero which would violate this equality.
		Triple junctions near the other points can be excluded analogously. Now even though there are two solutions to the branched transport problem,
		if $\tau$ is such that $\varepsilon$ has a kink in the right place or equivalently if $\tau$ is affine at least on $[\frac25,\frac35]$,
		the corresponding urban planning problem has a unique solution.
		For instance, define a differentiable transportation cost (see \cref{sf:fig8a} middle) by
		\begin{equation*}
		\tau(m)=\begin{cases*}
		m&if $m\in[0,\tfrac15]$,\\
		\frac25-\frac54(m-\frac35)^2&if $m\in(\frac15,\frac25]$,\\
		\frac{3}{20}+\frac12m&if $m\in(\frac25,\frac35]$,\\
		-\frac{11}{20}+\sqrt{m+\frac{2}{5}}&if $m>\frac35$.
		\end{cases*}
		\end{equation*}
		The corresponding maintenance cost (\cref{sf:fig8b} right) is given by
		\begin{equation*}
		\varepsilon(b)=\begin{cases*}
		\frac1{4b}-\frac{11}{20}+\frac25b&if $b\in(0,\frac12]$,\\
		\frac1{5}b^2-\frac35b+\frac25&if $b\in(\frac12,1]$,\\
		0&if $b>1$.
		\end{cases*}
		\end{equation*}
		Note that $\varepsilon$ has a kink at $b=\frac12$ with left derivative $-\frac35$ and right derivative $-\frac25$.
		The solution $(S,b)$ to the urban planning problem is uniquely determined in the following sense: Ignoring sets where $b=\tau'(0)=1$ (which one may always do without loss of generality) and identifying $b$ in the $\Ha$-almost everywhere sense, the optimal pair $(S,b)$ is uniquely given by
		\begin{equation*}
		S=[(0,0),(0,1)]\cup[(2,0),(2,1)]\textup{\qquad and\qquad}b\equiv\tfrac12.
		\end{equation*}
	\end{examp}
	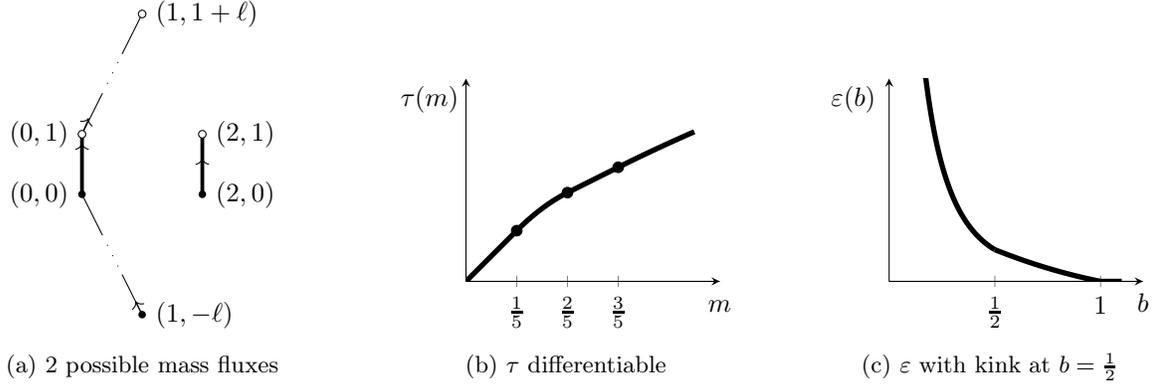
\begin{figure}
		\centering
		\begin{subfigure}[b]{0.3\textwidth}
			\centering
			\begin{tikzpicture}[scale=0.8]
			\draw[line width = 0.05cm] (0,0) -- (0,1);
			\draw[line width = 0.05cm] (2,0) -- (2,1);
			\draw (0,1) -- ++({1/3},{2/3});
			\draw[loosely dotted] ({1/3},{1+2/3}) -- ({2/3},{1+4/3});
			\draw ({2/3},{1+4/3}) -- ({3/3},{1+6/3});
			\draw (0,0) -- ++({1/3},{-2/3});
			\draw[loosely dotted] ({1/3},{-2/3}) -- ({2/3},{-4/3});
			\draw ({2/3},{-4/3}) -- ({3/3},{-6/3});
			\draw[->] (0,1) -- ++({0.25*1/sqrt(5)},{0.3*2/sqrt(5)});
			\draw[->] (0,0) -- (0,0.85);
			\draw[->] (1,-2) -- ++({-0.25*1/sqrt(5)},{0.25*2/sqrt(5)});
			\draw[->] (2,0) -- ++(0,{0.6});
			\node[circle,fill=black,inner sep=0.5pt,minimum size=0.1cm,label={180:$(0,0)$}] at (0,0) {};
			\node[circle,fill=black,inner sep=0.5pt,minimum size=0.1cm,label={0:$(2,0)$}] at (2,0) {};
			\node[circle,fill=black,inner sep=0.5pt,minimum size=0.1cm,label={0:$(1,-\ell)$}] at (1,-2) {};
			\node[draw=black,circle,fill=white,inner sep=0.5pt,minimum size=0.1cm,label={180:$(0,1)$}] at (0,1) {};
			\node[draw=black,circle,fill=white,inner sep=0.5pt,minimum size=0.1cm,label={0:$(2,1)$}] at (2,1) {};
			\node[draw=black,circle,fill=white,inner sep=0.5pt,minimum size=0.1cm,label={0:$(1,1+\ell)$}] at (1,3) {};
			\end{tikzpicture}
			\caption{$2$ possible mass fluxes}
			\label{sf:fig8c}
		\end{subfigure}
		\begin{subfigure}[b]{0.3\textwidth}
			\centering
			\begin{tikzpicture}
			\begin{axis}[x = 9.6em, y = 9.6em,
			samples = 500, 
			domain = 0:1,
			ymax = .8,
			axis x line=bottom, 
			axis y line=left, 
			xlabel = { $m$}, 
			ylabel = { $\tau (m)$}, 
			x label style={at={(axis description cs:1,0)},anchor=south},
			y label style={at={(axis description cs:.07,0.9)},rotate=270,anchor=west},
			xtick = {0.2,0.4,0.6},
			xticklabels = {$\tfrac15$,$\tfrac25$,$\tfrac35$},
			ytick = \empty,
			]
			% Sprung erster Ordnung
			\addplot[black,domain=0:0.2,line width=2pt] expression {x};
			\addplot[black,domain=0.2:0.4,line width=2pt] expression {.4-5/4*(x-.6)^2};
			\addplot[black,domain=0.4:0.6,line width=2pt] expression {.15+.5*x};
			\addplot[black,domain=0.6:.9,line width=2pt] expression {-11/20+sqrt(x+.4)};
			\addplot[black,line width=0pt] expression {0};
			\addplot[only marks,mark=*,mark options={scale=1, fill=black},text mark as node=true] coordinates{({1/5},{1/5})};
			\addplot[only marks,mark=*,mark options={scale=1, fill=black},text mark as node=true] coordinates{({2/5},{7/20})};
			\addplot[only marks,mark=*,mark options={scale=1, fill=black},text mark as node=true] coordinates{({3/5},{9/20})};
			\end{axis}
			\end{tikzpicture}%
			\caption{$\tau$ differentiable}
			\label{sf:fig8a}
		\end{subfigure}
		\begin{subfigure}[b]{0.3\textwidth}
			\centering
			\begin{tikzpicture}
			\begin{axis}[x = 8em, y = 8em,
			samples = 500, 
			domain = 0:1.2,
			ymax = 0.96,
			axis x line=bottom, 
			axis y line=left, 
			xlabel = {$b$}, 
			ylabel = {$\varepsilon (b)$}, 
			x label style={at={(axis description cs:1,0)},anchor=south},
			y label style={at={(axis description cs:.1,0.9)},rotate=270,anchor=west},
			xtick = {.5,1},
			xticklabels = {$\frac12$,$1$},
			ytick = \empty,
			]
			%\addplot[only marks,mark=*,mark options={scale=1.7, fill=black},text mark as node=true] coordinates{(1,0)};
			\addplot[black,domain=0.1:.5,line width=2pt] expression {0.25/x-11/20+.4*x};
			\addplot[black,domain=0.5:1,line width=2pt] expression {x^2/5-3*x/5+.4};
			\addplot[black,domain=1:1.1,line width=2pt] expression {0};
			\addplot[black,line width=0pt] expression {0};
			\end{axis}
			\end{tikzpicture}
			\caption{$\varepsilon$ with kink at $b=\frac12$}
			\label{sf:fig8b}
		\end{subfigure}
		\caption{Sketch illustrating \cref{epskink}. Due to a kink in $\varepsilon$ there are $2$ possible optimal mass fluxes for $\ell$ sufficiently large.}
	\end{figure}
	The final example illustrates the reverse situation: The transportation cost $\tau$ exhibits a nondifferentiability so that to one $\xi$ there may correspond multiple $b$.
	As a result there will be a unique optimal mass flux for the branched transport problem, but multiple optimal solutions of the urban planning problem.
	\begin{examp}[$\tau$ has a kink]
		\label{taukink}
		Assume that $\mu_+=\delta_x$ and $\mu_-=\delta_y$ with $x\neq y$ (see \cref{fig7} left). Then the solution of the branched transport problem is unique (independent of $\tau$) and given by $\F_{opt}=\vec{e}\Ha\mres e$, where $e=[x,y]$ and $\vec{e}=(y-x)/|y-x|$. Further, we have $\mathcal{J}^{\tau,\mu_+,\mu_-}[\F_{opt}]=\tau(1)\Ha(e)$. Now set $S=e$ and let $\tau$ have a kink, for instance $\tau(m)=\min(m,1)$. Then for every spatially constant (and in fact even non-constant) $b\in [0,1]$ we obtain
		\begin{equation*}
		\mathcal{U}^{\varepsilon,\mu_+,\mu_-}[S,b]=W_{d_{e,1,b}}(\delta_x,\delta_y)+\int_e\varepsilon(b)\dHa=b\Ha(e)+\varepsilon(b)\Ha(e)=\Ha(e)=\mathcal{J}^{\tau,\mu_+,\mu_-}[\F_{opt}]
		\end{equation*}
		(see \cref{fig7} right) so that there exist infinitely many optimal friction coefficients $b$.
		Note that in this example $\varepsilon$ is differentiable below $a=\tau'(0)$.
	\end{examp}
	\begin{figure}
		\centering
		\begin{subfigure}[b]{0.33\textwidth}
			\centering
			\begin{tikzpicture}
			\node[label=180:$x$,circle,fill=black,inner sep=0.5pt,minimum size=0.1cm] (x) at (0,0) {};
			\node[label=0:$y$,draw=black,circle,fill=white,inner sep=0.5pt,minimum size=0.1cm] (y) at (2,1.5) {};
			\draw[thick] (x) -- (y);
			\draw[-{Straight Barb[width=3mm,length=3mm]},thick] (x) --++(1,0.75);
			\end{tikzpicture}
			\caption{optimal mass flux independent of $\tau$}
		\end{subfigure}%
		\begin{subfigure}[b]{0.33\textwidth}
			\centering
			\begin{tikzpicture}
			\begin{axis}[x = 4em, y = 4em,
			samples = 500, 
			domain = 0:2, 
			ymax = 1.5,
			axis x line=bottom, 
			axis y line=left, 
			xlabel = { $m$}, 
			ylabel = { $\tau (m)$}, 
			x label style={at={(axis description cs:1,0)},anchor=south},
			y label style={at={(axis description cs:.07,0.9)},rotate=270,anchor=west},
			xtick = {1},
			xticklabels = {$1$},
			ytick = {1},
			yticklabels = {$1$},
			]
			% Sprung erster Ordnung
			\addplot[black,domain=0:1.9,line width = 2pt] expression {min(x,1)};
			\addplot[black,line width=0pt] expression {0};
			\end{axis}
			\end{tikzpicture}%
			\caption{$-\partial(-\tau)(1)= [0,1]$}
		\end{subfigure}%
		\begin{subfigure}[b]{0.33\textwidth}
			\centering
			\begin{tikzpicture}
			\begin{axis}[x = 4em, y = 4em,
			samples = 500, 
			domain = 0:2,
			ymax = 1.5,
			axis x line=bottom, 
			axis y line=left, 
			xlabel = {$b$}, 
			ylabel = {$\varepsilon (b)$}, 
			x label style={at={(axis description cs:1,0)},anchor=south},
			y label style={at={(axis description cs:.1,0.9)},rotate=270,anchor=west},
			xtick = {1},
			xticklabels = {$1$},
			ytick = {1},
			yticklabels = {$1$},
			]
			\addplot[black,domain=0:1.9,line width = 2pt] expression {max(0,1-x)};
			\addplot[black,line width=0pt] expression {0};
			\end{axis}
			\end{tikzpicture}
			\caption{$b+\varepsilon(b)\equiv 1$ for $b\in [0,1]$}
		\end{subfigure}
		\caption{Sketch illustrating \cref{taukink}. A kink in $\tau$ leads to multiple optimal friction coefficients.}
		\label{fig7}
	\end{figure}
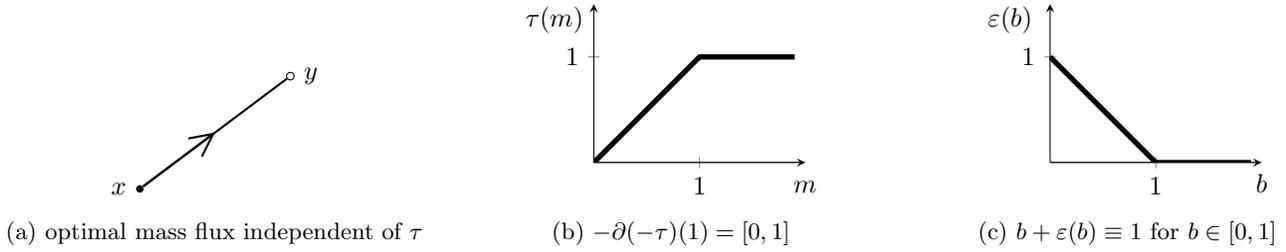
	%\textcolor{orange}{\section{Discussion}}
	%\todo[inline]{Julius: Maybe it makes sense to prove ``$\tau$ differentiable $\implies$ $b$'s unique'' \& ``$\varepsilon$ differentiable $\implies$ $m$'s unique'' etc.}
	\section{Acknowledgements}
	This work was supported by the Deutsche Forschungsgemeinschaft (DFG, German Research Foundation) under the priority program SPP 1962, grant WI 4654/1-1, and under Germany's Excellence Strategy EXC 2044 -- 390685587, Mathematics M\"unster: Dynamics--Geometry--Structure. B.W.’s and J.L.'s research was supported by the Alfried Krupp Prize for Young University Teachers awarded by the Alfried Krupp von Bohlen und Halbach-Stiftung. The publication was supported by the Open Access Publication Fund of the University of M\"unster.
	\printbibliography
\end{document}